\documentclass[12pt,reqno]{amsart}
\usepackage{cite}
\usepackage{xcolor}
\usepackage{tikz-cd}
\usepackage{float}
\usepackage{bm}
\usepackage{mathtools}
\usepackage{graphicx}
\usepackage{enumerate}
\usepackage{enumitem}

\usepackage{amsfonts,amssymb,verbatim,amsmath,amsthm,latexsym,textcomp,amscd,hyperref}
\usepackage{latexsym,amsfonts,amssymb,epsfig,verbatim}
\usepackage{graphicx}
\usepackage{color}
\usepackage{url}
\usepackage{pgfplots}
\usepackage{tikz}

\graphicspath{ {./images/} }
\newtheorem{theorem}{Theorem}[section]
\newtheorem*{theorem*}{Theorem}
\newtheorem*{lemma*}{Lemma}
\newtheorem*{proposition*}{Proposition}
\newtheorem*{corollary*}{Corollary}
\newtheorem{lemma}[theorem]{Lemma}
\newtheorem{defn}[theorem]{Definition}
\newtheorem{prop}[theorem]{Proposition}
\newtheorem{cor}[theorem]{Corollary}
\newtheorem{example}[theorem]{Example}
\theoremstyle{definition}
\newtheorem{remark}[theorem]{Remark}

\newtheorem{convention}[theorem]{Convention}

\def\RR{\mathbb R_{\geq 0}}
\def\dl{\delta}

\def\ep{\epsilon}

\def\pr{^\prime}
\def\prr{^{\prime\prime}}
\def\lm{\lambda}
\def\Lm{\Lambda}
\def\ri{\rightarrow}
\def\hri{\hookrightarrow}
\def\sse{\subseteq}
\def\gm{\gamma}
\def\bt{\beta}
\def\al{\alpha}
\def\fa{\forall}
\def\pa{\partial}
\def\map{\rightarrow}
\def\lgl{\langle}
\def\rgl{\rangle}

\newcommand\MF{\mathfrak}

\newcommand\LMY{\lim^Y_{n \map \infty}}
\newcommand\LMX{\lim^X_{n \map \infty}}
\newcommand\RED{\textcolor{red}}
\newcommand\BLUE{\textcolor{blue}}

\newcommand\CYAN{\textcolor{cyan}}

\def\L{\mathcal{L}}
\def\V{\mathcal{V}}

\def\F{\mathcal{F}}

\def\Y{\mathcal{Y}}
\def\G{\mathcal{G}}
\def\T{\mathcal{T}}

\def\A{\mathcal{A}}

\def\R{\mathbb {R}}

\def\N{\mathbb {N}}

\def\Z{\mathbb {Z}}

\def\Y{\mathcal {Y}}

\begin{document}
	
\title[Embeddings of trees of hyperbolic metric spaces and CT maps]{Embeddings of trees of hyperbolic metric spaces and Cannon--Thurston maps}
\author{Rakesh Halder and PRANAB SARDAR}

\address{Tata Institute of Fundamental Research, Mumbai, India}

\email{rhalder.math@gmail.com}
\address{Indian Institute of Science Education and Research (IISER) Mohali, Knowledge City,  Sector 81, S.A.S. Nagar 140306, Punjab, India}
\email{psardar@iisermohali.ac.in}

\thanks{ This paper is part of R. Halder’s
	PhD thesis written under the supervision of P. Sardar.}

%\date{\today}
	
\maketitle
\begin{abstract}
Given a tree of hyperbolic metric spaces $\pi:X\to T$ a la Bestvina--Feighn (\cite{BF}), and a hyperbolic subspace $Y$ of $X$ with an induced tree of hyperbolic spaces structure over a subtree $S\subset T$, we address the question as to when the Cannon--Thurston (CT) map exists for the inclusion $Y\to X$. In this paper, we find additional sufficient conditions under which the CT map $\partial Y \to \partial X$ exists. However, we show with examples that this may fail to hold in general. These results about trees of spaces are then applied to graphs of hyperbolic groups to prove various existence results for CT maps. A very special instance of these results is the following: 

\emph{Suppose $G_1$ and $G_2$ are hyperbolic groups with a common quasiconvex subgroup $H$, and the free product with amalgamation $G = G_1 *_H G_2$ is hyperbolic. Suppose $K_i < G_i$, $i = 1,2$ are hyperbolic subgroups containing $H$ and $K=K_1*_H K_2$. Then $K$ (is hyperbolic and,) the inclusion $K\to G$ admits a CT map if the inclusions $K_i\to G_i$, $i=1,2$ admit CT maps.}
%	
%Suppose $G_1, G_2$ are two hyperbolic groups with a common quasiconvex subgroup $H$ such that the free product
%with amalgamation $G=G_1*_H G_2$ is hyperbolic. Suppose $K_i<G_i$, $i=1,2$ are also quasiconvex subgroups where
%$H<K_i$, $i=1,2$; let $K=K_1*_H K_2$. Then it follows from the work of Bestvina and Feighn (\cite{BF}) that $K$ is
%hyperbolic and it follows from the work of the second named author jointly with M. Kapovich in \cite[Theorem $8.71$]{ps-kap}
%that the inclusion $K\map G$ extends continuously to their (Gromov) boundaries. 
%
%\noindent The current paper arose out of the question if the
%same conclusion holds when $K_i<G_i$ are no longer quasiconvex but $K_i$'s are merely hyperbolic. As an application
%of the main result of this paper, we answer this question affirmatively provided both the inclusions $K_i\map G_i$ 
%extends continuously to their boundaries. 
%Moreover, in the course of this work, we found a nonexistence theorem for CT maps too which is similar to that of Baker-Riley (\cite{baker-riley}) but which is conceptually somewhat easier to understand.
\end{abstract}

\tableofcontents

%%%%%%%%%%%%%%%%%%%%%%%%%%%%%%%%%%%%%%%%%%%%%%%%%5
%%%%%%%%%%%%%%%%%%%%%%%%%%%%%%%%%%%%%%%%%%%%%%%%%

\section{Introduction}\label{introduction}
Given a continuous map $f:X\map Y$ between metric spaces and compactifications
$\bar{X}$ and $\bar{Y}$ of $X$, $Y$ respectively, one may ask whether there is a natural extension of $f$ to a continuous map $\bar{f}:\bar{X}\map \bar{Y}$. We note that every proper hyperbolic metric space $X$ admits a natural compactification $\bar{X}=X\cup \pa X$ where $\pa X$ is called the Gromov boundary of $X$.
For maps $f:X\map Y$ between proper Gromov hyperbolic metric spaces continuous extension of $f$ to $\bar{f}:\bar{X}\map \bar{Y}$ is known as the Cannon--Thurston (CT) map. This terminology, coined by Mahan Mitra (Mj) (\cite{mitra-trees}, \cite{mitra-ct}), was motivated by the seminal work of Cannon and Thurston who discovered the first nontrivial instance of such a phenomenon in \cite{CT,CTpub}. Mitra formally introduced this question in Geometric Group Theory; for instance the following question has motivated a lot of research.

\smallskip
\noindent{\bf Question 1.} (\cite[Question 1.19]{bestvinaprob}) {\em Let $G$ be a (Gromov) hyperbolic group \textup{(see \cite{gromov-hypgps})} and $H$ be a hyperbolic subgroup. Does the inclusion $H\ri G$ extend to a continuous map $\pa H\ri \pa G$ at the level of the Gromov boundaries of the respective groups?}

\smallskip
 One is referred to  \cite{mahan-icm} for a detailed survey of results on CT maps. Although the answer to the question of Mitra was finally answered negatively (\cite{baker-riley}), the sets of positive examples and negative examples in the context of this question are still very limited. In this paper, we have an existence theorem for CT maps which adds to the positive examples. Also towards the end we have mentioned some new negative examples.

 As the answer to Question $1$ is negative in general, one would hope to get a positive answer only in the presence of additional hypotheses. In this paper we deal with groups which have graphs of groups structure. More precisely, we have the following as one of the motivating questions for  our results.

\smallskip

\noindent{\bf Question 2.} {\em Suppose $(\G,\Y)$ is a graph of hyperbolic groups with the qi embedded condition such that the fundamental group $G=\pi_1(\G,\Y)$ is hyperbolic. Assume that $H<G$ is a hyperbolic subgroup which admits a subgraph of subgroups structure (see Definition \ref{subgraph of gps}) over a connected subgraph $\Y'$ of $\Y$, and $(\G',\Y')$ is the corresponding subgraph of hyperbolic subgroups  of $(\G,\Y)$ with qi embedded condition. Suppose that for all vertex $\mathfrak{u}$ in $\Y'$, the CT map $\pa H_{\mathfrak{u}}\map \pa G_{\mathfrak{u}}$ exists. Then, does the CT map $\pa H\map \pa G$ exist?}

\smallskip

It follows from Bestvina--Feighn (\cite{BF-Adn}) and S. Gersten (\cite[Corollary 6.7]{gersten}) that $H$ is hyperbolic. However, the very special case of Question $2$ where $\Y'$ is a vertex of $\Y$ was answered positively by Mitra in \cite{mitra-trees}. Motivated by that, M. Kapovich and the second named author considered the case where each inclusion $H_\mathfrak{u}\map G_{\mathfrak{u}}$ is a qi embedding and answered Question $2$ positively in this case (See \cite[Theorem $8.71$]{ps-kap}). In this paper, we extend both of these theorems as follows. In the following theorem, the superscript {\em prime} notation refers to the vertex and edge groups of $(\G',\Y')$.\smallskip

%\begin{theorem}[Theorem \ref{main-app-ct-gen}]\label{thm-main group ct}
\noindent{\bf Theorem A} (Theorem \ref{main-app-ct-gen}){\bf.}~{\em Suppose $(\G,\Y)$ is a graph of hyperbolic groups with the qi embedded condition such that the fundamental group $G=\pi_1(\G,\Y)$ is hyperbolic. Let $(\G',\Y')$ be a subgraph of hyperbolic subgroups of $(\G,\Y)$ with qi embedded conditions over a connected subgraph $\Y'$ of $\Y$ . Let $G'=\pi_1(\G', \Y')$. We also assume the following.

\begin{enumerate}
	\item For each vertex $\mathfrak{u}$ of $\Y'$,  the CT map $\pa G'_{\mathfrak{u}}\ri \pa G_{\mathfrak{u}}$ exists.
	
	\item For any edge $\mathfrak{e}$ of $\Y'$ incident on a vertex $\mathfrak{u}$
    the following hold:
	
	\begin{enumerate}
		
		\item The inclusion $G'_{\mathfrak{e}}\to G_{\mathfrak{e}}$ is a qi embedding.
		
		\item (Compatible pairwise projection condition) There is $R_0\ge0$ such that for all $g\in G'_{\mathfrak{u}}$, we have $$d_{G_{\mathfrak{u}}}(P^{G_{\mathfrak{u}}}_{G_{\mathfrak{e}}}(g),P^{G'_{\mathfrak{u}}}_{G'_{\mathfrak{e}}}(g))\le R_0$$ where $G'_{\mathfrak{e}}<G_{\mathfrak{e}}<G_{\mathfrak{u}}$, $G'_{\mathfrak{e}}<G'_{\mathfrak{u}}$ via injective homomorphisms, and $P^X_{U}:X\ri U$ denotes a nearest point projection map from a metric space $X$ onto $U\sse X$.
	\end{enumerate} 
	
\end{enumerate}
Then the CT map $\pa G' \map \pa G$ exists.}

\begin{remark}
We note that the compatible pairwise projection condition in $2$ $(b)$ above is a natural hypothesis. In our case, we have $G'_{\mathfrak{e}}=G_{\mathfrak{e}}\cap G'_{\mathfrak{u}}$ (see Proposition \ref{prop-inter property}), and $2$ $(b)$ is equivalent to requiring that 
$$ 2~(c)\hspace{1cm} \Lambda_{G_{\mathfrak{u}}}(G'_{\mathfrak{u}})\cap \Lambda_{G_{\mathfrak{u}}}(G_{\mathfrak{e}})=\Lambda_{G_{\mathfrak{u}}}(G'_{\mathfrak{e}}).$$
See Lemma \ref{lem-char of proj cond}. Also, this holds in case all inclusions of the vertex group $G'_{\mathfrak{u}}\map G_{\mathfrak{u}}$ are qi embeddings as in \cite[Theorem $8.71$]{ps-kap}. The following is another instance where the compatible pairwise projection condition holds.
\end{remark}

\noindent{\bf Theorem B} (Theorem \ref{thm-app-ct-finite index}){\bf.}~{\em
Suppose $(\G,\Y)$ is a graph of hyperbolic groups with the qi embedded condition such that the fundamental group $G=\pi_1(\G,\Y)$ is hyperbolic. Let $(\G',\Y')$ be a subgraph of hyperbolic subgroups of $(\G,\Y)$ with the qi embedded conditions over a connected subgraph $\Y'$ of $\Y$ . Let $G'=\pi_1(\G', \Y')$. We also assume the following.

\begin{enumerate}
	\item For each vertex $\mathfrak{u}$ of $\Y'$,  the CT map $\pa G'_{\mathfrak{u}}\to \pa G_{\mathfrak{u}}$ exists.
	
	\item  For each edge $\mathfrak{e}$ of $\Y'$, the inclusion $G'_{\mathfrak{e}}\to G_{\mathfrak{e}}$ is an isomorphism onto a finite index subgroup of $G_{\mathfrak{e}}$.
		
\end{enumerate}
	
Then the CT map $\pa G'\map \pa G$ exists.}\smallskip

As immediate applications of Theorem B, we have the following examples. We note that these examples do not follow from the results in existing literature.

\begin{example}\label{exp-ct-par}
$(1)$ Suppose $G_1$ and $G_2$ are hyperbolic groups with a common quasiconvex subgroup $H$, and that the free product with amalgamation $G = G_1 *_H G_2$ is hyperbolic. Suppose $K_i < G_i$, $i = 1,2$ are hyperbolic subgroups containing $H$ and $K=K_1*_H K_2$. Then $K$ (is hyperbolic and,) the inclusion $K\map G$ admits a CT map provided both the inclusions $K_i\map G_i$, $i=1,2$ admit CT maps.

$(2)$	Consider a hyperbolic group $H$ of the form $H=N\rtimes Q$, where $N$ is either the fundamental group of a closed orientable surface of genus at least $2$ or a finitely generated free group of rank at least $3$, and $Q$ is a finitely generated free group of rank at least $2$. Examples of this sort are well-known; e.g. see \textup{\cite{farb-mosher}, \cite{BFH-lam}}. It is easy to see that $Q$ is a malnormal quasiconvex subgroup of $G$. 

	Now let $F<Q$ be a malnormal free subgroup of rank at least $3$ and let $\phi:F\map F$ be a hyperbolic automorphism. 
	Then it follows from \textup{\cite{BF}} that the HNN extensions $H_2=H_1*_{F=\phi(F)}<G=H*_{F=\phi(F)}$,  where
	$H_1=N\rtimes F$, are both hyperbolic. (We note that $H_1$ is hyperbolic by the results of \textup{\cite{pranab-mahan}}.) Moreover, the inclusion $H_1\map H$ admits a CT map by \textup{\cite{ps-krishna}}.
	
	However, it easily follows from Theorem B that the inclusion $H_2\map G$ admits a CT map. 
\end{example}

However, the above group theoretic results are deduced using  the results, which we describe below, about trees of spaces. The following question served as a motivation for us to prove these theorems.\smallskip

\noindent{\bf Question 3.} {\em
	
	\begin{enumerate}
		\item Suppose $\pi:X\ri T$ is a tree of hyperbolic metric spaces satisfying the qi embedded condition such that $X$ is hyperbolic (Definition \ref{tree-of-sps}). % Further, we assume the following. %\textcolor{red}{Properness of X, Y?}
		
		\item Let $Y\sse X$ be a hyperbolic subspace such that the inclusion $i:Y\ri X$ is a proper embedding.

		\item The restriction of $\pi$ on $Y$, $\pi|_Y:Y\ri S=\pi(Y)$, is a tree of hyperbolic metric spaces over $S$ with the qi embedded condition. % where for $Y$ the maps from edge spaces to corresponding vertex spaces are the restriction of those from $\pi:X\ri T$. %This structure is named (in \cite{ps-halder-ct}) induced subtree of spaces.
		
		\item For all vertex $u\in V(S)$ and for all edge $e\in E(S)$, inclusions between the vertex spaces $Y_u\ri X_u$ and the edge spaces $Y_e\ri X_e$ admit CT maps.
		
		\item Both $X$ and $Y$ are proper metric spaces.
		
	\end{enumerate}
Does the CT map $\pa Y \map \pa X$ exist?	}
	\smallskip

	\begin{remark}\label{nece flar}
		Under the above hypotheses hyperbolicity of $Y$ is ensured. Indeed, since $X$ is hyperbolic, $\pi:X\ri T$ satisfies flaring condition which implies the same for $Y$. Basically the proof of  \cite[Proposition $5.8$]{pranab-mahan} works in this case too.  Hence, by \cite{BF}, $Y$ is hyperbolic.
		
		%	$(2)$ Conditions $(4)$ in setup {\bf S} is not artificial, its natural in group-theoretic setting. 
	\end{remark}
M. Mitra answered Question $3$ positively in the case $S$ is a vertex $v$ of $T$ and $Y_v=X_v$. Recently, in the book \cite{ps-kap}, M. Kapovich and the second named author answered this question positively when $S$ is any subtree and $Y=\pi^{-1}(S)$, although their proof gives a more general theorem; see Theorem D. However, we note that the answer to Question $3$ is negative in general as can be seen from the group-theoretic counter-examples. See for instance, Example \ref{counterexample}. Nevertheless, we are able to prove the following theorem that extends the theorems of \cite{mitra-trees} and \cite{ps-kap}.\smallskip

%	\begin{theorem}[Theorem \ref{main thm}]\label{main-theorem-ct}

\noindent{\bf Theorem C} (See Theroem \ref{main thm}){\bf.}~{\em Suppose we have the hypotheses (1)-(5) above plus the following.
		
		\begin{enumerate}
			\item The inclusion $Y_e\ri X_e$ is a (uniform) qi embedding for all edge $e\in E(S)$.
			
			\item The {\bf compatible fiberwise projection condition} holds:

There is a constant $R_0\geq 0$ such that for all vertex $u\in V(S)$ and edge
	$e\in E(S)$ incident on $u$, and for all $y\in Y_u$ we have
	$$d_{X_u}(P^{X_u}_{X_{eu}}(y), P^{Y_u}_{Y_{eu}}(y))\leq R_0$$ where $Y_{eu}$ (resp. $X_{eu}$) denotes the image of edge space $Y_e$ (resp. $X_e$) in the vertex space $Y_u$ (resp. $X_u$), and $P^W_{U}:W\ri U$ denotes a nearest point projection map from a metric space $W$ onto $U\sse W$.
		\end{enumerate}
		Then the inclusion $Y\ri X$ admits a CT map.
	}\smallskip

%	The idea of the proof is that given two such sequences $\{y_n\},\{y\pr_n\}$ we find a
%	new pair of sequences, say $\{w_n\}$, $\{w'_n\}$ of $Y$, with additional properties 
%	so that checking if $\LMX w_n=\LMX w'_n$ is easier whereas by construction we have
%	$\LMX w_n =\LMX y_n$ and $\LMX w'_n =\LMX y'_n$. Sometimes we may need to do this a number
%	of times.
%	
%\noindent The construction of new sequences $\{w_n\},~\{w'_n\}$ depends on certain properties of the sets hull$\{\pi(y_n):n\in\N\}$ and hull$\{\pi(y\pr_n):n\in\N\}$. In each step of constructing new sequences, we find common (uniform) quasiconvex subsets of $Y$ and $X$ containing $y_n$ and $w_n$ (resp. $y'_n$ and $w'_n$). The construction of a common quasiconvex subset of $Y$ and $X$ containing $y_n$ and $w_n$ depends on the flow spaces constructed in the book\cite{ps-kap}.\qed\smallskip

	When the maps in the vertex space levels are uniform qi embeddings then we have the following stronger consequence.\smallskip
	
\noindent{\bf Theorem D} (See Theorem \ref{thm-CT proj plus qi emb}){\bf.}~{\em Suppose we have the hypotheses (1)-(4) mentioned above. Moreover, suppose that for all vertex $u\in V(S)$ and for all edge $e\in E(S)$, the inclusions $Y_u\ri X_u$ and $Y_e\ri X_e$ are uniformly qi embeddings, and the compatible fiberwise projection condition holds. Then
	\begin{enumerate}
		\item the inclusion $Y\ri X_S$ is a qi embedding where $X_S:=\pi^{-1}(S)$, and
		
		\item hence by \textup{\cite[Theorem 8.11]{ps-kap}}, the inclusion $Y\ri X$ admits a CT map.
	\end{enumerate}}\smallskip
	
\noindent{\bf Cannon--Thurston lamination.}

	Using the techniques developed to prove Theorem C, we next delve into the properties of the CT lamination.
	We note that `CT lamination' was defined by Mitra in \cite{mitra-endlam} to study noninjectivity of CT maps. 

\begin{defn}[Cannon--Thurston (CT) lamination]
	Suppose $W\sse Z$ are hyperbolic metric spaces where $W$ is equipped with the path metric from $Z$. Assume that the inclusion $i_{W,Z}:W\map Z$ admits a CT map $\pa i_{W,Z}:\pa W\map\pa Z$. The CT lamination is then defined as $$\L_{CT}(W,Z):=\{(p,q)\in\pa^2W:\pa i_{W,Z}(p)=\pa i_{W,Z}(q)\}.$$
	
	 Suppose $(p,q)\in \L_{CT}(W,Z)$ and $\alpha$ is a geodesic line in $W$ joining $p,q$. Then we shall refer to $\alpha$ as a {\em leaf of the CT lamination} $\L_{CT}(W,Z)$.
\end{defn}

In this connection, among other results, we prove the following.\smallskip

\noindent{\bf Theorem E} (Theorem \ref{thm-at least one contain ray}){\bf.}~{\em 	
Suppose $Y\sse X$ are hyperbolic spaces satisfying the assumptions of Theorem C. Assume that $\al:\R\map Y$ is a geodesic line in $Y$ such that $\pi(\al)$ contains a geodesic ray. Then $\al$ is not a leaf of the CT lamination $\L_{CT}(Y,X)$.}\smallskip

We need the following definition for our next result. {\em Suppose $\pi:X\map T$ is a tree of hyperbolic spaces with the qi embedded condition such that $X$ is hyperbolic (see Definition \ref{tree-of-sps}). Let $u\in V(T)$ and $\xi\in\pa T$. Define $$\pa^{\xi}(X_u,X)=\{\eta\in\pa i_{X_u,X}(\pa X_u):\exists\text{ qi lift, }\gm\text{ say, of }[u,\xi)\text{ such that }\gm(\infty)=\eta\}$$and $$\Lm^{\xi}(X_u,X)=\{(p,q)\in\pa^2X_u:\pa i_{X_u,X}(p),\pa i_{X_u,X}(q)\in\pa^{\xi}(X_u,X)\}.$$}

\noindent{\bf Theorem F} (Theorem \ref{thm-both does not con geo ray}){\bf.}~{\em 
Suppose $Y\sse X$ are hyperbolic spaces satisfying the assumptions of Theorem C. Suppose $\al:\R\map Y$ is a geodesic line in $Y$, and that $\al$ is a leaf of the CT lamination $\L_{CT}(Y,X)$. 

Then there is a vertex $u\in V(S)$ and a geodesic line $\gm:\R\map Y_u$ such that $\pa i_{Y_u,Y}(\gm(\pm\infty))=\al(\pm\infty)$, and exactly one of the following holds.  

\begin{enumerate}
	\item $\gamma$ is a leaf of the CT lamination $\L_{CT}(Y_u,X_u)$.
	
	\item $\gamma$ is not a leaf of the CT lamination $\L_{CT}(Y_u,X_u)$, and
	$$(\pa i_{Y_u,X_u}(\gm(-\infty)),\pa i_{Y_u,X_u}(\gm(\infty)))\in\Lm^{\xi}(X_u,X)$$for some $\xi\in\pa T\setminus\pa S$.
\end{enumerate}
}

As an application of Theorem F (and Theorem \ref{thm-qc}), we have the following group-theoretic result. We recall that subgroup $A$ of a group $B$ is said to be {\em weakly malnormal} if $A\cap gAg^{-1}$ is finite for all $g\in B\setminus A$.\smallskip
	
\noindent{\bf Theorem G} (Corollary \ref{cor-Swarup's qsn}){\bf.}~{\em
Suppose $(\G,\Y)$ is a graph of hyperbolic groups with the qi embedded condition such that the fundamental group $\pi_1(\G,\Y)$ is hyperbolic. Assume that $(\G',\Y')$ is a subgraph of subgroups with the qi embedded condition. Moreover, suppose that 

\begin{enumerate}
	\item $\Y'=\Y$.
	
	\item For each vertex $\mathfrak{u}$ of $\Y'$, the vertex group $G'_{\mathfrak{u}}$, of $(\G',\Y')$ corresponding to $\mathfrak{u}$, is quasiconvex subgroup of $G_{\mathfrak{u}}$.
	
	\item For each edge $\mathfrak{e}$ of $\Y'$, we have $G'_{\mathfrak{e}}=G_{\mathfrak{e}}$.
\end{enumerate}
If $\pi_1(\G',\Y')$ is weakly malnormal in $\pi_1(\G,\Y)$ then $\pi_1(\G',\Y')$ is quasiconvex in $\pi_1(\G,\Y)$.}\smallskip

Theorem G generalizes the main result of \cite{mitra-ht} in a special case when height of the subgroup is $1$.\smallskip

\noindent{\bf A few words on the proof of Theorem C.} Let $\{y_n\}$ and $\{y\pr_n\}$ be two unbounded sequences in $Y$ such that $\LMY y_n=\LMY y\pr_n\in\pa Y$, and suppose that $\LMX y_n$, $\LMX y\pr_n\in\pa X$ exist in $\pa X$. We show that $\LMX y_n=\LMX y\pr_n$ in $\pa X$. This completes the proof by Lemma \ref{not-ct-not-comp}. 

First, by \cite[Theorem $8.11$]{ps-kap}, it suffices to prove Theorem C in the case $T=S$. By (Lemma \ref{vert or horiz} and) Corollary \ref{vert or horiz replace}, we deduce that the sequences have specific forms: either $\pi(y_n)$ is a fixed vertex ({\em vertical sequence}) or $\pi(y_n)$ lies on a geodesic ray $[u,\xi)\sse T$ and $\lim^T_{n\to\infty}\pi(y_n)=\xi\in\pa T$ ({\em horizontal sequence}). This reduction is carried out in Section \ref{sec-subtree of spaces} and constitutes the most important and technical part of the argument. We are thus reduced to the following three cases: 

\begin{enumerate}
	\item Both $\{y_n\}$ and $\{y'_n\}$ are vertical sequences. This case follows from Lemma \ref{finite-proj-case}.
	
	\item $\{y_n\}$ is a vertical sequence and $\{y'_n\}$ is a horizontal sequence.
	
	\item Both $\{y_n\}$ and $\{y'_n\}$ are horizontal sequences.
\end{enumerate}

\noindent By a further analysis of the sequences, Case $2$ is reduced to Case $1$. The key input here is that the flow spaces of edge spaces of $Y$ are uniformly quasiconvex in both $Y$ and $X$ (Corollary \ref{flow-qc} and Lemma \ref{Y flow qc in X}). Finally, using the key input, in Case $3$, if $\lim^T_{n\to\infty}\pi(y_n)\ne\lim^T_{n\to\infty}\pi(y'_n)$, then we reduce to Case $2$; otherwise, we construct a uniformly quasiconvex subset of both $Y$ and $X$ containing $y_n$ and $y'_n$. This concludes the proof.\qed\smallskip

\noindent{\bf Organization of the paper}: In Section \ref{prelim}, we recall basic definitions and results in hyperbolic metric spaces which are used in the subsequent sections. We recall trees of metric spaces in Section \ref{trees-of-metric-sps} and prove some results related to this for later use. Following \cite[Chapter $3$]{ps-kap}, in Subsection \ref{subsec-flow space}, we recall flow spaces in a modified form and prove some of their properties. Some results related to Gromov boundary of trees of metric spaces are proved in Subsection \ref{subsec-boundary of trees of sps}. In Section \ref{sec-subtree of spaces}, we define subtrees of subspaces of trees of spaces and prove that any unbounded sequence in such a subspace that is Gromov in both spaces admits a representative that is a Gromov sequence in both spaces and is either vertical or horizontal, as mentioned above (see Corollary \ref{vert or horiz replace}). This is the heart underlying the proof of the main theorem (Theorem C).
	Section \ref{sec-main thm} contains the proof of our main theorem based on the study in Section \ref{sec-subtree of spaces}. In the end of this section, we prove Theorem D. Finally, as applications of Theorem C, in Section \ref{sec-CT lamination}, we study Cannon--Thurston lamination and prove Theorems E and F. In Section \ref{sec-application}, we study group-theoretic applications of Theorems C and F. For instance, in this Section, we prove Theorems A, B and G, and other related results including nonexistence of CT maps.\smallskip

\noindent{\bf Acknowledgements.}
Authors thank Mahan Mj for useful comments on an earlier version of the paper, which helped to improve the exposition. Authors also thank Ravi Tomar for pointing out Lemma \ref{acyl-case}.

%	\BLUE{Modification of the above example to a generalized version}: Consider a hyperbolic group $H$ of the form $H=N\rtimes Q$, where $N$ is either the fundamental group of a closed orientable surface of genus at least $2$ or a finitely generated free group of rank at least $3$, and $Q$ is a finitely generated free group of rank at least $2$. Examples of this sort are well-known; e.g. see \textup{\cite{farb-mosher}, \cite{BFH-lam}}. It is easy to see that $Q$ is a malnormal quasiconvex subgroup of $G$. Now let $F<Q$ be a malnormal free subgroup of rank at least $3$ and let $\phi:F\map F$ be a hyperbolic endomorphism such that $\phi(F)$ is a malnormal subgroup of $F$ (see \cite{mutangend}). Then it follows from \textup{\cite{BF}} that the HNN extensions $H_2=H_1*_{\phi}<G=H*_{\phi}$,  where $H_1=N\rtimes F$, are both hyperbolic. (We note that $H_1$ is hyperbolic by the results of \textup{\cite{pranab-mahan}}.) However, it easily follows from Theorem \ref{main-app-ct-gen} that the inclusion $H_2\map G$ admits the CT map.

\section{Preliminaries}\label{prelim}
We start this section with some standard definitions from coarse geometry.
\begin{defn}
	\begin{enumerate}

		%\item {\tiny By a proper map $\psi:Z\ri W$ between two topological spaces we mean that the inverse image of compact sets are compact.}

		\item {\bf Metrically proper map:} A map $\psi:\R_{\geq 0}\map \R_{\geq 0}$ will be called a (metrically) proper map if
		%(i) $\psi(r)\geq r$ for all $r\in \R_{\geq 0}$ and (ii)
		inverse image of a bounded set under $\psi$ is bounded or equivalently $\lim_{r\map \infty} \psi(r)=\infty$.

		\item {\bf Proper embedding of metric spaces:} A map $f: Y\map X$ between metric spaces is said to be a (metrically) $\phi$-proper embedding for a proper map
		$\phi:\R_{\ge0}\ri\R_{\ge0}$ if $d_X(f(y),f(y\pr))\le r$ implies $d_Y(y,y\pr)\le \phi(r)$ for all $y,y'\in Y$.

		\item {\bf Geodesic:} A {\em geodesic} in a metric space $X$ is an isometric embedding $\alpha:I\map X$ where $I$
		is an interval in $\mathbb R$. If $I$ is a finite closed interval then $\alpha$ is called a {\em geodesic
		segment}. If $I$ is of the form $[a,\infty)$ or $(-\infty, a]$ for some $a\in \mathbb R$,
		then $\alpha$ is called a {\em geodesic ray} in $X$. If $I=\mathbb R$ then $\alpha$ is called a {\em geodesic line}
		in $X$. {\em However, later on wherever we write `$\gamma$ is a geodesic ray'
		without explicitly mentioning the domain $I$ it will be assumed that $I=[0,\infty)$.}
\item {\bf Concatenation of paths:} By concatenation of paths we shall mean the following.

If $\alpha_1:[a,b]\map X$ and $\alpha_2: [c,d]\map X$ are two paths in a metric space $X$ with
$\alpha_1(b)=\alpha_2(c)$ then their {\em concatenation} is the path $\alpha= \alpha_1 * \alpha_2:[0,(b-a)+(d-c)]\map X$
where $\alpha(t)= \alpha_1(t+a)$ for $0\leq t\leq b-a$ and
$\alpha(t)=\alpha_2(t-(b-a)+c)$ for $ b-a\leq t\leq (b-a)+(d-c)$.
\item {\bf Induced path metric:} Suppose $X$ is a geodesic metric space and $Y\subset X$. Suppose for all $y, y'\in Y$
there is a rectifiable path of $X$ contained in $Y$. In this case, by the {\em path metric on $Y$ induced from $X$} we shall
mean the metric $d_Y$ on $Y$ given by the following: For all $y,y'\in Y$ one has
$d_Y(y,y')=\inf \{ \mbox{length of} \,\, \alpha\}$ where the infimum is taken over all paths $ \alpha$ in $X$
with end points $y,y'$ which are contained in $Y$. For details one is referred to \cite[Chapter I.1 and Chapter I.3]{bridson-haefliger}.
\item {\bf Hausdorff distance:} Suppose $X$ is a metric space. Let $A,B\sse X$ be any two subsets. Then the {\em Hausdorff distance} between $A$ and $B$ is denoted and defined as follows. $$Hd(A,B):=\text{inf}~\{r:A\sse N_r(B),~B\sse N_r(A)\}.$$
	\end{enumerate}
\end{defn}

%{\bf Convention.} However, when we talk about a proper function $\psi:\R_{\geq 0}\map \R_{\geq 0}$
%we additionally assume that $\psi(r)\geq r$ for all $r\in \R_{\geq 0}$.

\begin{lemma}\label{rev-of-pro-map}
	For any proper function $\psi:\R_{\geq 0}\map \R_{\geq 0}$ there is a proper function $\bar{\psi}:\R_{\geq 0}\map \R_{\geq 0}$
	such that $\psi(\bar{\psi}(r))\leq r$ for all $r\in \R_{\geq 0}$.
\end{lemma}
\proof For any $r\in \R_{\geq 0}$, let $S_r=\{a\in\R_{\geq 0}:\psi(a)\le r\}$. This set is bounded since $\psi$ is proper.
Let $\bar{\psi}(r)=\sup \,S_r$. Then it follows that $\bar{\psi}$ is a proper map.
\qed

\smallskip
The lemma above has the following immediate corollary.
\begin{cor}\label{cor-inverse of proper map}
	Suppose $X,Y$ are metric spaces, $\psi:\R_{\geq 0}\map \R_{\geq 0}$ is a proper map and
	$f:Y\map X$ is a $\psi$-proper map.
	Then $\fa~r\in\RR$ and $\fa~y,y\pr\in Y,~d_Y(y,y\pr)>r$ implies $d_X(f(y),f(y\pr))> \bar{\psi}(r)$ where
	$\bar{\psi}$ is defined as in Lemma \ref{rev-of-pro-map}.
\end{cor}
%\begin{proof}
%Let $S_n=\{r\in\N:f(r)\le n\}$. Define a map $g:\N\ri\N$ such that $g(n)=\textrm{max } S_n=r_n$ (say). Thus $g$ is a proper map. Note that $f(r_n)\le n$. Therefore, $y,y\pr\in Y,~d_Y(y,y\pr)>n$ implies $d_X(y,y\pr)>g(n)$; otherwise, if $d_X(y,y\pr)\le g(n)$ implies $d_Y(y,y\pr)\le f(g(n))=f(r_n)\le n$ $-$ which is a contradiction.
%$(ii)$ On contrary, we assume that $g$ is not proper i.e. there is $s\in\N$ and infinite set $\{n_i:i\in\N\}$ such that $g(n_i)\le s$. Note that as $f$ is proper, $r_n\rightarrow\infty$ as $n\rightarrow\infty$.
%We show that $S_i\sse S_j,\fa i\le j$. Let $l\in S_i$. Then $f(l)\le i\le j\Rightarrow l\in S_j$. Which implies if $i\le j$ implies $g(i)=r_i\le r_j=g(j)$. Note also that $\cup_{i\in\N}S_i=\N$ and so $lim_{i\ri\infty}r_i=\infty$.
%Again, $g(n_i)\le s\Rightarrow f(r_{n_i})\le s$ and so, $lim_{i\ri\infty}f(r_{n_i})\le s$ -which is a contradiction as $i\ri\infty\Rightarrow n_i\ri\infty\Rightarrow r_{n_i}\ri\infty\Rightarrow f(r_{n_i})\ri\infty$.
%\end{proof}

Lemma \ref{qi11} $(1)$ and $(2)$ follow easily from the definition of a qi embedding, whereas $(3)$ follows from this definition together with $(2)$.
\begin{lemma}\label{qi11}
	Let $k\ge1$ and $D,D'\ge0$. Suppose $f:X\map Y $ is a $k$-qi embedding between the metric spaces $X,Y$
	and $A,A'\subset X$, $B\subset Y$. Then the following holds:
	
	(1) If $Hd_X(A,A')\leq D$ then $Hd_Y(f(A), f(A'))\leq D_{\ref{qi11}}(k,D)$.
	
	(2)$Hd_X(A, f^{-1}(f(A)))\leq R_{\ref{qi11}}(k)$.
	
	(3) If $Hd_Y(B, f(A))\leq D'$ then $Hd_X(f^{-1}(B), A)\leq D'_{\ref{qi11}}(k,D')$.
\end{lemma}

\subsection{Hyperbolic metric spaces}
We assume that the reader is familiar with (Gromov) hyperbolic metric spaces.
We recall some basic definitions and results from the theory of Gromov hyperbolic spaces and groups
which are essential for our purpose. For more details
on (Gromov) hyperbolic metric spaces and (Gromov) hyperbolic groups we refer the reader to standard
references like \cite{bridson-haefliger}, \cite{abc}, \cite{GhH}, or  Gromov's original article
\cite{gromov-hypgps}. {\em We shall always assume that the hyperbolic spaces dealt with in this paper
are geodesic metric spaces; and we will adopt Rips' definition of hyperbolicity; so a $\delta$-hyperbolic space
(for some $\delta\geq 0$) will mean a geodesic metric space in which geodesic triangles are all $\delta$-slim.}

\subsubsection{\bf Quasigeodesics in hyperbolic spaces}
The following result is known as Morse Lemma or Stability of Quasi-Geodesic (see \cite[Theorem $1.7$, III. H]{bridson-haefliger}).
\begin{lemma}[{Stability of quasi-geodesic}]\label{ml}
	Given $\dl\ge0, \ep\ge0$ and $k\ge1$, we have a constant $D_{\ref{ml}}=D_{\ref{ml}}(\dl,k,\ep)\ge0$ such that the following holds.

	Suppose $X$ is a $\dl$-hyperbolic metric space. Then for any geodesic $\alpha$ in $X$ and a $(k,\epsilon)$-quasi-geodesic $\beta$ with the same end points as $\alpha$, the Hausdorff distance between $\alpha$ and $\beta$ is bounded by $D_{\ref{ml}}$.
\end{lemma}

We note that quasigeodesics are proper paths with linear properness. However, in the presence of a Lipschitz embedding and arbitrary properness,
we have a sort of converse to Lemma \ref{ml} as follows.
\begin{lemma}\textup{(\cite[Lemma $2.5$]{ps-krishna})}\label{lem-giving qg}
	Suppose $X$ is a metric space, $x,y\in X$, and that $\gm$ is a $k$-quasigeodesic in $X$. Let $\al$ be a $1$-Lipschitz and $\eta$-proper path in $X$ such that $Hd(\al,\bt)\le D$. Then $\al$ is a $K_{\ref{lem-giving qg}}$-quasigeodesic in $X$ where $K_{\ref{lem-giving qg}}=K_{\ref{lem-giving qg}}(k,D,\eta)\ge1$.
\end{lemma}

The following lemma follows easily by slimness of triangles.
\begin{lemma}\label{lem-qg}
	Given $\dl\ge0$, $R\ge0$ there is a constant $K_{\ref{lem-qg}}=K_{\ref{lem-qg}}(\dl,R)\ge1$ such that the following holds.
	
	Suppose $X$ is a $\dl$-hyperbolic metric space. Let $p,q,r,s\in X$ be such that $d(p,q)\le R$ and $d(r,s)>R+2\dl$. Let $z\in[p,r]$ and $w\in[q,s]$ be such that $R+2\dl\le d(z,w)\le R+5\dl$. Then the arc length parametrization of $[r,z]\cup[z,w]\cup[w,s]$ is a $K_{\ref{lem-qg}}$-quasigeodesic.
\end{lemma}

\begin{comment}

The lemma below is easy to verify using Lemma \ref{ml}. Hence, we omit its proof.

\begin{lemma}\label{replace-by-qg}
Given $\dl\ge0$, $k\ge1$ there is $k_{\ref{replace-by-qg}}=k_{\ref{replace-by-qg}}(\dl,k)$ such that the following holds.

Suppose $X$ is $\dl$-hyperbolic geodesic metric space and $\al:[0,\infty)\ri X$ is a geodesic ray in $X$. Let $\{r_i\}$ be a strictly increasing, unbounded sequence in $[0,\infty)$. Let $\al_i$ be a continuous $k$-quasi-geodesic joining $\al(r_i)$ and $\al(r_{i+1})$. Suppose $\bt$ is the concatenation of the various $\al_i$'s. Then the arc length parameterization on $\bt$ is a $k_{\ref{replace-by-qg}}$-quasi-geodesic in $X$.
\end{lemma}
\end{comment}

\subsubsection{Quasiconvex subsets}

	%\RED{Parts of this section should be moved to before boundaries}
	\begin{defn}
Suppose $K \ge 0$ and $X$ is a geodesic metric space. A subset $Y \subseteq X$ is said to be
$K$-\emph{quasiconvex} if every geodesic in $X$ joining any pair of points of $Y$ lies in the
$K$-neighbourhood of $Y$ in $X$. A subset of $X$ is said to be \emph{quasiconvex} if it is
$K$-quasiconvex for some $K \ge 0$.
	\end{defn}

%\RED{Group starts}

%\RED{new lemma, needed just for the sake of completeness of the discussion.}

Following is a natural way to construct examples of  quasiconvex subsets.
\begin{defn}{\em ({\bf Quasiconvex hull})}
Given a subset $A$ of a hyperbolic metric space $X$, the quasiconvex hull, denoted by $Hull_X(A)$, is the union of all geodesic
segments joining pairs of points from $A$.
\end{defn}
Following lemma is an easy consequence of slimness of triangles (and the stability of quasigeodesics).
\begin{lemma}\label{hull of qc}
Let $\dl\ge0$, $k\ge0$ and $L\ge1$. Suppose $X$ is a $\delta$-hyperbolic metric space and $A\subset X$. Then the following hold.

\begin{enumerate}[label=(\arabic*)]
\item \label{hull of qc:1} $Hull_X(A)$ is $K_{\ref{hull of qc}\ref{hull of qc:1}}(\dl)$-quasiconvex in $X$.

\item \label{hull of qc:2} Suppose, moreover, $A\subset X$ is $k$-quasiconvex.
Then $Hd(A, Hull_X(A))\leq D_{\ref{hull of qc}\ref{hull of qc:2}}(\delta, k)$.
\end{enumerate}
\end{lemma}

%\RED{Group ends}

The following three lemmas show other ways to construct quasiconvex subsets. Lemmas \ref{lem-union qc} and \ref{lem-union qc qg} follow easily from the stability of quasigeodesics (Lemma \ref{ml}) and the slimness of triangles, whereas Lemma \ref{qc morph} follows from the stability of quasigeodesics together with the definition of a qi embedding.
\begin{lemma}\label{qc morph}
Let $\dl,k,k_1,k_2\ge0$.	Suppose $X,Y$ are $\delta$-hyperbolic metric spaces and $f:X\map Y$ is a $k$-qi embedding.
	Suppose $A\subset X$ and $B\subset f(A)$ where $A$ is $k_1$-quasiconvex in $X$
	and $B$ is $k_2$-quasiconvex in $Y$. Then the following holds:

	(1) $f(A)$ is $K_{\ref{qc morph}}(\delta, k, k_1)$-quasiconvex in $Y$.

	(2) $f^{-1}(B)$ is $K'_{\ref{qc morph}}(\delta, k, k_2)$-quasiconvex in $X$.
\end{lemma}

\begin{lemma}\label{lem-union qc}
	Let $R\ge0$, $\dl\ge0$ and $k\ge0$. Suppose $X$ is a $\dl$-hyperbolic metric space. Suppose that $\{U_{\lm}:\lm\in\Lambda\}$ is a collection of $k$-quasiconvex subsets of $X$. Further, assume that $V$ is also a $k$-quasiconvex subset of $X$ such that $d(V, U_{\lm})\le R$ for all $\lm\in\Lambda$. Then $V\bigcup\bigcup_{\lm\in\Lambda}U_{\lm}$ is $(4\dl+k+R)$-quasiconvex in $X$.
\end{lemma}

\begin{lemma}\label{lem-union qc qg}
	Let $\dl\ge0$, $k\ge0$ and $R\ge0$. Suppose $X$ is a $\dl$-hyperbolic metric space. Let $\{U_{\lm}:\lm\in\Lm\}$ be a collection of $k$-quasiconvex subset of $X$. Further assume that $V$ is also a $k$-quasiconvex subsets of $X$ such that $V\sse N_R(\bigcup_{\lm\in\Lm}U_{\lm})$ and $V\cap U_{\lm}\ne\emptyset$ for all $\lm\in\Lm$. 
	Then $\bigcup_{\lm\in\Lm}U_{\lm}$ is $K_{\ref{lem-union qc qg}}$-quasiconvex where $K_{\ref{lem-union qc qg}}=R+k+2\dl$.
\end{lemma}

Note that in a hyperbolic metric space $X$, any qi embedded subspace is quasiconvex (see, for instance, \cite[Lemma $1.90$]{ps-kap}). The following lemma shows, in particular, that  the converse is true if the subspace is properly embedded. For a proof of the following lemma one may look at \cite[Lemma $1.91$]{ps-kap}.
\begin{lemma}\label{qi-emb-in-Y-X}{\em ({\bf Quasiconvex vs QI embedded})}
Given a map $\phi:\R_{>0}\ri\R_{>0}$ and constants $k\ge0$ and $R\ge k+1$ there are constants $L_{\ref{qi-emb-in-Y-X}}=L_{\ref{qi-emb-in-Y-X}}(k,R)\ge1$ and $L\pr_{\ref{qi-emb-in-Y-X}}=L\pr_{\ref{qi-emb-in-Y-X}}(k,\phi)\ge1$ such that we have the following.

$(1)$ Suppose $X$ is a geodesic metric space and $A$ is a $k$-quasiconvex subset of $X$. Then $N^X_R(A)$ is path connected and the inclusion $N^X_R(A)\to X$ is $L_{\ref{qi-emb-in-Y-X}}$-qi embedding with respect to the induced path metric on $N^X_R(A)$ from  $X$.

$(2)$ Suppose $Y\sse X$ is a path connected subspace such that the inclusion $Y\map X$ is $\phi$-proper embedding. Moreover, $Y$ is $k$-quasiconvex. Then the inclusion $Y\map X$ is $L'_{\ref{qi-emb-in-Y-X}}$-qi embedding.

In particular, if $G$ is a hyperbolic group and $H$ is a finitely generated subgroup then
$H$ is qi embedded in $G$ if and only if it is quasiconvex in (any Cayley graph of) $G$.
%\BLUE{This is not used, and yes it is wrong} \RED{This is wrong. Check where it is used:} {\tiny $(3)$ Suppose $Y\sse X$ is a $\dl$-hyperbolic geodesic subspace such that the inclusion $Y\hri X$ is $\phi$-proper embedding. Let $A\sse Y$ and $A$ be also $k$-quasiconvex in $Y$. Then $N^Y_R(A)$ is path connected and both the inclusions $N^Y_R(A)\hri X$ and $N^Y_R(A)\hri Y$ are $L\pr_{\ref{qi-emb-in-Y-X}}$-qi embedding.}
\end{lemma}

\noindent{\bf Nearest point projection maps}\\
In dealing with quasiconvex sets the existence of coarsely well-defined nearest point projection maps
turns out to be greatly useful. For a subset $A$ of a hyperbolic metric space $X$, we shall denote by $P^X_A$
a nearest point projection map $X\map A$. The following lemma gives the most useful properties of these maps
and it also show that this is a characterizing property of quasiconvex sets.
%\RED{Group starts}

\begin{lemma}\label{proj-on-qc}
Let $\dl\ge0$, $k\ge0$ and $r\ge1$. Suppose $X$ is a $\dl$-hyperbolic metric space and $U\sse X$. Then we have the following.
\begin{enumerate}[label=(\arabic*)] % makes items (1), (2), ...
\item \textup{(\cite[Lemma $7.3.D$]{gromov-hypgps})}\label{proj-on-qc:1} If  $U$ is $k$-quasiconvex in $X$ then $P^X_{U}:X\ri U$ is a $C_{\ref{proj-on-qc}\ref{proj-on-qc:1}}$-coarsely Lipschitz retraction where $C_{\ref{proj-on-qc}\ref{proj-on-qc:1}}=C_{\ref{proj-on-qc}\ref{proj-on-qc:1}}(\dl,k)>0$.

\item \textup{(\cite[Lemma $1.4.2$]{bowditch-stacks})}\label{proj-on-qc:lem-retraction imp qc} Suppose that $\rho:X\map U$ is a $r$-coarsely Lipschitz retraction onto $U$. Then $U$ is a $K_{\ref{proj-on-qc}\ref{proj-on-qc:lem-retraction imp qc}}$-quasiconvex in $X$ where $K_{\ref{proj-on-qc}\ref{proj-on-qc:lem-retraction imp qc}}=K_{\ref{proj-on-qc}\ref{proj-on-qc:lem-retraction imp qc}}(\dl,r)\ge0$. %$K_{\ref{proj-on-qc}~\ref{proj-on-qc:lem-retraction imp qc}}$
\end{enumerate}	
\end{lemma}

%\RED{Group ends}

	%We know that qi embedded subspaces are quasiconvex in hyperbolic geodesic metric spaces. But with the help of Lemma \ref{proj-on-qc} $(1)$ and Lemma \ref{qi-emb-in-Y-X}, we have the following for converse.
\begin{comment}
	\begin{lemma}\label{qc-is-qi-emb}
		Given $\dl\ge0,K\ge0$ and $R\ge K+1$, there exists $L_{\ref{qc-is-qi-emb}}=L_{\ref{qc-is-qi-emb}}(\dl,K,R)$ such that we have the following.

		Suppose $X$ is a $\dl$-hyperbolic geodesic metric space and $Y$ is a $K$-quasiconvex subset of $X$. Then the inclusion $i:N_R(Y)\hookrightarrow X$ is $L_{\ref{qc-is-qi-emb}}$-qi embedding, where $N_R(Y)$ is considered with its induced path metric from $X$.
	\end{lemma}%% 
\end{comment}

	%In hyperbolic geodesic metric space $X$, nearest  point projection of a point on a quasiconvex subset, say, $U$ is coarsely well-defined (see \cite[ Hyperbolic Groups, Lemma $7.3$.D]{gromov-hypgps}). So we get a map $P_{X,U}:X\ri U$, called \emph{a nearest point projection map} on $U$. Sometimes we denote this map by $P_U$ if $X$ is understood from the context. We collects some facts in the following lemmas; some are well known in hyperbolic geodesic metric space and some are very easy to prove.

%	\RED{Group starts:} 
	The following lemma describes various properties of the nearest point projection
	maps on quasiconvex sets.

\begin{lemma}\label{lem-proj on pair qc}{\em ({\bf Properties of nearest point projection maps})}
Let $\dl\ge0$, $k\ge0$ and $R\ge0$. Suppose $X$ is a $\dl$-hyperbolic metric space, and $U,~V$ are $k$-quasiconvex subsets of $X$. Then we have the following.
\begin{enumerate}[label=(\arabic*)]
\item\label{lem-proj on pair qc:new proj lemma3} Assume that $Hd_X(U,V)\leq R$. Then for all $x\in X$ we have
$$d_X(P^X_{U}(x), P^X_{V}(x))\leq D_{\ref{lem-proj on pair qc} \ref{lem-proj on pair qc:new proj lemma3}}$$for some constant $D_{\ref{lem-proj on pair qc}\ref{lem-proj on pair qc:new proj lemma3}}=D_{\ref{lem-proj on pair qc}\ref{lem-proj on pair qc:new proj lemma3}}(\dl,k,R)$. In particular, letting $U=V$, the map $P^X_{U}$ is coarsely well-defined (see Lemma \ref{proj-on-qc}~\ref{proj-on-qc:1}).

\item \label{lem-proj on pair qc:qc proj qc1}\textup{(\cite[Lemma $1.113$]{ps-kap})} For any $x,y\in X$, $Hd_X(P^X_{U}([x,y]),[P^X_{U}(x), P^X_{U}(y)])\leq D_{\ref{lem-proj on pair qc}\ref{lem-proj on pair qc:qc proj qc1}}$ for some constant $D_{\ref{lem-proj on pair qc}\ref{lem-proj on pair qc:qc proj qc1}}=D_{\ref{lem-proj on pair qc}\ref{lem-proj on pair qc:qc proj qc1}}(\dl,k)$.

\item \label{lem-proj on pair qc:concatenation qg} \textup{(\cite[Lemma $1.31$ $(2)$]{pranab-mahan})} For any $x\in X$ and $y\in U$, the arclength parametrization of $[x,P^X_{U}(x)]\cup[P^X_U(x),y]$ is a $(3+2k)$-quasigeodesic.

\item \label{lem-proj on pair qc:qc proj qc2}\textup{(\cite[Lemma $1.117$]{ps-kap})} $P^X_U(V)$ is $K_{\ref{lem-proj on pair qc}\ref{lem-proj on pair qc:qc proj qc2}}(\delta,k)$-quasiconvex in $X$.

\item \label{lem-proj on pair qc:small-imp-small} \textup{(\cite[Corollary $1.140$]{ps-kap})} If $diam\{P^X_{U}(V)\}\le R$ then $diam\{P^X_{V}(U)\}\le D\pr$ for some constant $D\pr$ depending only on $\dl,~k$ and $R$.

\item \label{lem-proj on pair qc:lem-nested proj} \textup{(\cite[Lemma $1.32$]{pranab-mahan})} Assume that $V\sse U$. Let $x\in X$ and $x'=P^X_U(x)$. Then $d_X(P^X_V(x),P^X_V(x'))\le D$ for some constant $D=D_{\ref{lem-proj on pair qc}\ref{lem-proj on pair qc:lem-nested proj}}(\dl,k)$.

\item \label{lem-proj on pair qc:new proj lemma} Suppose $U_1=P^X_{V}(U)$ and $U_1\sse N_R(U)$. Then for all $x\in V$ we have
$d_X(P^X_{U}(x), P^X_{U_1}(x))\leq D_{\ref{lem-proj on pair qc}\ref{lem-proj on pair qc:new proj lemma}}$ for some constant $D_{\ref{lem-proj on pair qc}\ref{lem-proj on pair qc:new proj lemma}}=D_{\ref{lem-proj on pair qc}\ref{lem-proj on pair qc:new proj lemma}}(\dl,k,R)$.
\end{enumerate}
\end{lemma}	

\begin{proof}
Note that $(1)$ follows from Lemma \ref{proj-on-qc}~\ref{proj-on-qc:1}. Hence we will prove $(7)$ only.

Let $x_1=P^X_{U_1}(x),~x_2=P^X_{U}(x)$, and let $x_3\in U$ be such that $d(x_1,x_3)\le R$. We need to show a uniform bound on $d(x_1,x_2)$. By \ref{lem-proj on pair qc:concatenation qg}, the arc-length parametrization of $[x,x_2]\cup[x_2,x_3]$ is a $(3+2k)$-quasigeodesic. Now the stability of quasigeodesics and the slimness of triangles conclude that $x_2$ is uniformly close to $[x,x_1]\cup[x_1,x_3]$. If $x_2$ is uniformly close to $[x_1,x_3]$ then $x_2$ is uniformly close to $x_1$ as $d(x_1,x_3)\le R$. Otherwise, $d(x_2,x_4)\le R_1$ for some $x_4\in[x,x_1]$ and uniform constant $R_1\ge0$. Thus $d(x_4,x_5)\le k$ for some $x_5\in V$ as $V$ is $k$-quasiconvex. Hence there is $x_6\in U_1$ such that $d(x_2,x_6)\le R_1+k$ as $d(x_2,x_5)\le R_1+k$. This implies that $d(x_4,x_6)\le d(x_4,x_2)+d(x_2,x_6)\le 2R_1+k$, and so $d(x_4,x_1)\le2R_1+k$ as $x_1=P^X_{U_1}(x)$, $x_4\in[x,x_1]$ and $x_6\in U_1$. Therefore, $d(x_2,x_1)\le d(x_2,x_4)+d(x_4,x_1)\le3R_1+k$. This completes the proof.
\end{proof}

\begin{comment}
\RED{Move the lemma below to the above list.}
\begin{lemma}\textup{(\cite[Lemma $1.31$ $(2)$]{pranab-mahan})}\label{lem-concatenation qg}
	Let $k\ge0$. Suppose $X$ is a geodesic metric space, and $U$ is a $k$-quasiconvex subset of $X$. Then for any $x\in X$ and $y\in U$, the arclength parametrization of $[x,P^X_{U}(x)]\cup[P^X_U(x),y]$ is a $(3+2k)$-quasigeodesic.
\end{lemma}
\end{comment}

The following result says that the nearest point projections and the qi embeddings among hyperbolic spaces `almost commute'. The reader is referred to \cite[Lemma $3.5$]{mitra-trees} for a proof. In the lemma below
we note that $A$ is quasiconvex in $X$ by Lemma \ref{qc morph}.

\begin{lemma}\label{proj-in-diff-are-close}
		Given $\dl\ge0,L\ge1$ and $K\ge0$, we have constants $K_{\ref{proj-in-diff-are-close}}=K_{\ref{proj-in-diff-are-close}}(\dl,L,K)$ and $D_{\ref{proj-in-diff-are-close}}=D_{\ref{proj-in-diff-are-close}}(\dl,L,K)$ such that the following holds.

		Suppose $X$ is a $\dl$-hyperbolic metric space, and $Y\sse X$ is a geodesic metric space
		with the induced path metric from $X$. Suppose the inclusion $i:Y\map X$ is $L$-qi embedding.
		Suppose $A\sse Y$ is $K$-quasiconvex set in $Y$. Then for all $y\in Y$, if  $y\pr \in A$
		and $y\prr \in A$ are nearest point projections of $y$ on $A$ in the metric $Y$ and $X$ respectively
		then $d_X(y\pr,y\prr)\le D_{\ref{proj-in-diff-are-close}}$.
	\end{lemma}
	
\begin{comment}	
	\begin{proof} \RED{Just cite Mahan's lemma in 'CT for trees of spaces: Projection commutes with
	qi embedding. No need to give a proof.}
		It is simple that $A$ is $K_{\ref{proj-in-diff-are-close}}$-quasiconvex for some constant $K_{\ref{proj-in-diff-are-close}}$ depending on $\dl$, $L$ and $K$.

		For the second part, by \cite[Lemma $1.31$ (2)]{pranab-mahan}, we see that $[y,y\pr]_Y\cup[y\pr,y\prr]_Y$ is a $(3+2K)$-quasi-geodesic in $Y$ and so is $L_1$-quasi-geodesic in $X$, where $L_1$ depends on $(3+2K)$ and $L$. Suppose $y_1\in[y,y\prr]_X$ such that $d_X(y\pr,y_1)\le D_{\ref{ml}}(\dl,L_1,L_1)$, and so $d_X(y_1,y\prr)\le D_{\ref{ml}}(\dl,L_1,L_1)$. Therefore, $d_X(y\pr,y\prr)\le d_X(y\pr,y_1)+d_X(y_1,y\prr)\le2D_{\ref{ml}}(\dl,L_1,L_1)=:D_{\ref{proj-in-diff-are-close}}$.
	\end{proof}
\end{comment}

\begin{comment}
\RED{Strange lemma, I dont follow what this is, find a better statement and make comments on the proof
or supply a proof outline.}

\begin{lemma}\label{consis-proj}
Let $\dl\ge0$ and $k\ge0$. Suppose $Y$ is a $k$-qi embedding subspace of a $\dl$-hyperbolic space $X$. Let $B$ and $A$ be two $k$-qi embedding subspaces of $Y$ and $X$ respectively such that $B\sse A$. Then there are constant $R$ and $R\pr$ depending on $\dl$ and $k$ such that hull$\{P^X_A(Y)\}\sse N^Y_R(B)$ if and only if $d_X(P^Y_B(y),P^X_A(y))\le R\pr$ for all $y\in Y$.
\end{lemma}
\end{comment}

\subsubsection{\bf Coboundedness of quasiconvex sets}

\begin{defn}\textup{\cite[Definition $1.25$]{ps-kap}}
	(1) Suppose $X$ is a metric space, $A,B\subset X$ and $L\geq 0,C\geq 0$ are constants.
	We will say that $A,B$ are {\em $(L,C)$-Lipschitz cobounded} if there exists two $L$-coarsely Lipschitz retraction
	maps $f_A:X\map A$ and $f_B: X\map B$ such that the diameters of both $f_A(B)$ and
	$f_B(A)$ are at most $C$.

	(2) Suppose $X$ is a hyperbolic metric space, $A, B$ are quasiconvex subsets of $X$.

	\noindent
	(i) We say that $A,B$ are {\em $R$-separated}, for some $R\geq 0$ if $d(A,B)\geq R$.\\
	(ii) We say that $A,B$ are {\em $D$-cobounded} if $diam\{P^X_{B}(A)\}\leq D$ and $diam\{P^X_{A}(B)\}\leq D$.
\end{defn}

	%\RED{See if the lemmas below are useful at all, write a lemma to the effect that cobounded qcshave shortest connectors.}
Below we recall several results in this connection that will be used later. Lemma \ref{lem-cobdd}~\ref{lem-cobdd:union-qc} $(i)$ follows from \ref{lem-cobdd:R-sep-D-cobdd} and Lemma \ref{lem-union qc}, whereas $(ii)$ directly follows from Lemma \ref{lem-union qc}.

%\RED{Group starts}

\begin{lemma}\label{lem-cobdd}
Let $\dl\ge0$, $k\ge0$. Suppose $X$ is a $\dl$-hyperbolic metric space, and $U,~V$ are $k$-quasiconvex subsets of $X$. Then we have the following.
\begin{enumerate}[label=(\arabic*)]
\item \label{lem-cobdd:close imp proj inclu}\textup{(\cite[ Lemma $1.127$]{ps-kap})} If $d(U,V)\le R$ for some $R\ge0$ then $P^X_U(V)\sse N_{R'_{\ref{lem-cobdd}\ref{lem-cobdd:close imp proj inclu}}}(U)$ where $R'_{\ref{lem-cobdd}\ref{lem-cobdd:close imp proj inclu}}=2k+3\dl+R$.
	
\item \label{lem-cobdd:R-sep-D-cobdd} \textup{(\cite[Lemma $1.139$]{ps-kap}, \cite[Lemma $1.35$]{pranab-mahan})} Let
$R_{\ref{lem-cobdd}\ref{lem-cobdd:R-sep-D-cobdd}}=2k+5\dl$ and $D_{\ref{lem-cobdd}\ref{lem-cobdd:R-sep-D-cobdd}}=2k+7\dl$. Then, $U$ and $V$ are $R_{\ref{lem-cobdd}\ref{lem-cobdd:R-sep-D-cobdd}}$-separated implies that
$U$ and $V$ are $D_{\ref{lem-cobdd}\ref{lem-cobdd:R-sep-D-cobdd}}$-cobounded.

\item \label{lem-cobdd:qc proj new} \textup{(\cite[Lemma $1.36$]{pranab-mahan})} Let $U'=P^X_{U}(V)$, and let $x\in U, y\in V$ be any
points. Then $$[x,y]_X\cap N_{D'}(U')\neq \emptyset$$for some constant $D'=D_{\ref{lem-cobdd}\ref{lem-cobdd:qc proj new}}(\delta, k)$.
In particular, for any $D_1\geq 0$ if $U$ and $V$ are $D_1$-cobounded, then there is a point $p\in U'\subset V$
such that $p\in N_{D_2}([x,y]_X)$ where $D_2=D_1+D'$.

%\item \label{lem-cobdd:geo-close-fix-pt} \textup{(\cite[Lemma $1.36$]{pranab-mahan})}

\item \label{lem-cobdd:union-qc} Let $D\ge D_{\ref{lem-cobdd}\ref{lem-cobdd:R-sep-D-cobdd}}(\dl,k)$.
There is a constant $K=K_{\ref{lem-cobdd}\ref{lem-cobdd:union-qc}}(\dl,k,D)$ such that the following holds:

(i) If $U$, $V$ are not $D$-cobounded then $U\cup V$ is $K$-quasiconvex in $X$.

(ii) If $U, V$ are $D$-cobounded then $U\cup V\cup[x,y]$ is $K$-quasiconvex for any
$x\in U, y\in V$.

\item \label{lem-cobdd:lips cobdd lemma}\textup{(\cite[ Lemma $1.137$]{ps-kap})} For all $L\ge0$ and $C\ge0$ there is a constant
$D_{\ref{lem-cobdd}\ref{lem-cobdd:lips cobdd lemma}}=D_{\ref{lem-cobdd}\ref{lem-cobdd:lips cobdd lemma}}(\dl,k,L,C)$ such that the following holds:

If $U$ and $V$ are $(L,C)$-Lipschitz cobounded then they are $D_{\ref{lem-cobdd}\ref{lem-cobdd:lips cobdd lemma}}$-cobounded.
\end{enumerate}
\end{lemma}

\begin{comment}
\begin{lemma}\label{proj-from-net}
Suppose $X$ is a geodesic metric space and $U$ is a subset of $X$. Let $Y$ is a $D$-net in $X$ and $P:X\ri U$ is a $L$-coarsely Lipschitz retraction. Then there is a constant $L\pr$ depending on $D$ and $L$ such that $P$ can be extended to a $L\pr$-coarsely Lipschitz retraction of $X$ onto $U$.
\end{lemma}

\BLUE{Lemma \ref{geo-close-fix-pt} was supposed to group with Lemma \ref{lem-cobdd}. However, Lemma \ref{geo-close-fix-pt} is not used in the paper, and it is a part of Lemma \ref{lem-cobdd}~\ref{lem-cobdd:qc proj new}. So we can remove Lemma \ref{geo-close-fix-pt}.}

\begin{lemma}\textup{(\cite[Lemma $1.36$]{pranab-mahan})}\label{geo-close-fix-pt}
Let $\dl\ge0,k\ge0$ and $D\ge0$. Suppose $X$ is a $\dl$-hyperbolic geodesic metric space and $U,V$ are $k$-quasiconvex subsets of $X$. Further, suppose the pair $(U,V)$ is $D$-cobounded. Then there is $p\in U$ and a constant $D\pr$ depending on $\dl,k$ and $D$ such that any geodesic joining $x$ and $y$ passes through $N_{D\pr}(p)$, where $x\in U$ and $y\in V$.
\end{lemma}

\begin{lemma}\label{new proj lemma2}
	Suppose $X$ is a $\delta$-hyperbolic space and $Y\subset X$ is a $\delta$-hyperbolic too
	with respect to the induced path metric from $X$. Further, suppose that $Y\map X$ is $k$-qi embedding. Let $A\subset Y$ be $k$-quasiconvex
	in both $X$ and $Y$. Then for any $x\in Y$,
	$d_X(P_{X,A}(x), P_{Y,A}(x))\leq D_{\ref{new proj lemma2}}(\delta, k)$. %\CYAN{This is not true in general. This is our projection hypothesis.}
\end{lemma} \end{comment}

\subsubsection{\bf Gromov Boundary and Cannon--Thurston Maps}\label{Gro-b-CT-maps}
$\newline$
In this section we briefly recall basic definitions and results concerning boundary of hyperbolic
metric space. Although most of these can be worked out for general hyperbolic spaces. We will
make the following convention as the main result in this paper is about proper hyperbolic spaces.

\begin{convention}
In this paper we shall always assume
that our spaces are proper hyperbolic metric spaces unless they are trees in which case they need
not be proper. However, in this section some of the results hold without this assumption. Therefore,
we explicitly mention where properness is used and absence of that hypothesis will indicate that
the result is valid even without properness.
\end{convention}

Suppose $X$ is a Gromov hyperbolic metric space.
\begin{itemize}
\item Let $\mathcal G(X)$ be the set of all geodesic rays in $X$.
One defines an equivalence relation on $\mathcal G(X)$ by setting $\alpha \sim \beta$ if
$Hd(\alpha, \beta)<\infty$ for all $\alpha,\beta \in \mathcal G(X)$.
The set of equivalence classes, denoted by $\partial X$, is called the (geodesic) boundary
of $X$.

\item Suppose $\alpha:[a,\infty)\map X$ (resp. $\alpha: (-\infty, a]\map X$) is a geodesic ray in $X$.
Then the equivalence class of $\alpha$ is denoted by $\alpha(\infty)$ (resp. $\alpha(-\infty)$).
If $\alpha(0)=x$ then we say that $\alpha$ joins $x$ to $\alpha(\infty)$ (resp. $\alpha(-\infty)$).

\item Suppose $\gamma$ is a geodesic line in $X$. Then the restrictions of $\gamma$ on
$[0,\infty)$ and $(-\infty, 0]$ give two geodesic rays in $X$. We shall denote their equivalence classes
by $\gamma(\infty)$ and $\gamma(-\infty)$ respectively. We will say that $\gamma$ joins
$\gamma(\pm \infty)$.
\end{itemize}

The following has been taken from \cite{bridson-haefliger}.
\begin{lemma}(\cite[Lemma 3.1, Lemma 3.2, Chapter II.H]{bridson-haefliger})
	Suppose $X$ is a proper $\delta$-hyperbolic metric space or a tree.

	1) If $x\in X$ and $\xi\in \partial X$ then there is a geodesic ray $\alpha$ in $X$ with
	$\alpha(0)=x$ and $\alpha(\infty)=\xi$. If $\alpha'$ any other geodesic joining $x$ to $\xi$
	then $Hd(\alpha, \alpha')\leq D=D(\delta)$.

	2) If $\xi_1\neq \xi_2$ are two points of $\partial X$ then there is a geodesic line
	$\gamma$ in $X$ joining $\xi_1$ to $\xi_2$. If $\gamma'$ any other geodesic joining $\xi_1$
	to $\xi_2$ then $Hd(\gamma, \gamma')\leq D'=D'(\delta)$.
\end{lemma}

Once the existence of geodesic rays and lines are guaranteed, one constructs ideal triangle. The following lemma
follows easily from \cite[Lemma 3.3, Chapter III.H]{bridson-haefliger}.

\begin{lemma}\label{ideal triangle}
Suppose $X$ is a $\delta$-hyperbolic space and $x_1,x_2\in X$, $x_3\in \partial X$. Let $\alpha_{ij}$ be
a geodesic (segment or ray) joining $x_i, x_j$, $i\neq j$. Then the geodesic triangle  formed
by the $\alpha_{ij}$'s is $2\delta$-slim (provided the sides of the triangle are defined).
\end{lemma}

\subsubsection{\bf Sequences converging to $\partial X$}
We note that there is a natural compact, Hausdorff topology on $\bar{X}=X\cup \partial X$ for a proper
hyperbolic space $X$; see e.g. \cite[Chapter III.H]{bridson-haefliger} or \cite{abc}. We shall skip any
detailed discussion on that as we do not need it in this paper. However, we will mention a few
results below about convergence of points (or subsets) of $X$ to  points of $\partial {X}$ which will be essential for the
discussion about Cannon--Thurston maps later.

For example, the following result (Lemma \ref{conv criteria}) provides a geometric criterion for convergence in $\bar{X}$. The
proof relies on the fact that in a hyperbolic metric space $X$, the Gromov inner product of $x$ and $y$ with respect to $z$ remains uniformly close to $d(z,[x,y])$ (see \cite[Chapter $2$, Lemma $17$]{GhH}, for instance). For a complete proof, see \cite[Lemma 2.44, Lemma 2.45]{ps-krishna}. %We state it without proof.
%the reader may use the following as one of the many equivalent definitions of convergence in $\bar{X}$.

\begin{lemma}\label{conv criteria}
Suppose $\{x_n\}$ is a sequence in $\bar{X}$, $x\in X$ and $\xi\in \partial X$.
Suppose $\alpha_n$ is a geodesic ray or line (according as $x_n\in X$ or $x_n\in \partial X$)
in $X$ joining $x_n$ to $\xi$.  Then $\{x_n\}$ converges to $\xi$ if and only if $\lim_{n\map \infty}d(x,\alpha_n) =\infty$.
\end{lemma}

The following lemma is a straightforward consequence of the version of Arzel\`a--Ascoli theorem
as in \cite[Lemma 3.10, Chapte I.3]{bridson-haefliger}; see the proof of
\cite[Lemma 3.1, Chapter III.H]{bridson-haefliger} for instance.

\begin{lemma}\label{subseq exists}
Any unbounded subsequence $\{x_n\}$ in a proper hyperbolic metric space $X$ has
		a subsequence $\{x_{n_k}\}$ with $\lim_{k\map \infty} x_{n_k}\in \partial X$.
\end{lemma}
An immediate consequence of the above lemma is the following.

\begin{lemma}\label{unbdd-pro-sp-ray}
	If $X$ is an unbounded proper hyperbolic metric space then $\partial X\neq \emptyset$.
\end{lemma}

%Proof of the following lemma is a carbon copy of the proof of Lemma \ref{unbdd-pro-sp-ray}.

\noindent{\bf Notation.} Suppose $\{x_n\}$ is a sequence in $X$ and $\xi\in \bar{X}$.
Then we write $\LMX x_n=\xi$ to mean that $\{x_n\}$ converges to $\xi$ in $\bar{X}$.
If $\LMX x_n=\xi$ for some $\xi\in \partial X$ then we say that $\LMX x_n$ exists.
Later in this paper we shall frequently encounter situations where there are two
hyperbolic spaces $Y\subset X$ and sequences $\{y_n\}$ in $Y$. To differentiate
between the limits of this sequences in $\bar{X}$ and $\bar{Y}$ we use superscript
as above.\smallskip

The following three lemmas follow easily from Lemma \ref{ideal triangle} and
Lemma \ref{conv criteria}.

\begin{lemma}\label{con-same-pt}
	Suppose $X$ is a proper hyperbolic metric space and $x\in X$. Then the following hold.

Suppose $\{x_n\}$ and $\{x'_n\}$ are two unbounded sequence in $X$ such that $\LMX x_n$ exists. Then $\LMX x'_n$ exists and $\LMX x_n=\LMX x'_n$ if and only if $\lim_{n\map \infty} d(x,[x_n,x'_n])=\infty$.

		Moreover, if $\LMX x_n=\LMX x'_n$ and $z_n\in [x_n,x'_n]$ then $\LMX z_n=\LMX x_n$.

\end{lemma} %\RED{Find a ref or if nothing is there then write a proof}

\begin{lemma}\label{lem-close imp same limit}
Suppose $X$ is a hyperbolic metric space. Let $R\ge0$, and let $\{x_n\},~\{y_n\}\sse X$ be such that $\lim_{n\map\infty}x_n$ exists in $\pa X$, and $d(x_n,y_n)\le R$. Then $\lim_{n\map \infty}y_n$ exists in $\pa X$ and $\lim_{n\map \infty}x_n=\lim_{n\map \infty}y_n$. 
\end{lemma}

\begin{lemma}\label{seq-on-con-geo}
Suppose $X$ is hyperbolic metric space and $\{x_n\}\sse X$ such that $\lim_{n\ri\infty}x_n$ exists. Let $x\in X$ and $x\pr_n\in[x,x_n]$ such that $d(x,x\pr_n)\ri\infty$ as $n\ri\infty$. Then $\lim_{n\ri\infty}x_n=\lim_{n\ri\infty}x\pr_n$.
\end{lemma}

\begin{lemma}\label{lem-same limits}
Let $\phi:\R_{\ge0}\map\R_{\ge0}$ be a proper map, and $k\ge0$, $\dl\ge0$. Suppose $Y\sse X$ are $\dl$-hyperbolic spaces where $Y$ is equipped with the path metric induced from $X$. Assume that the inclusion $Y\map X$ is $\phi$-proper embedding. If $Z\sse Y$ is $k$-quasiconvex in both $X$ and $Y$ then $N^Y_{k+1}(Z)$ (equipped with path metric induced from $Y$) is $L_{\ref{lem-same limits}}$-qi embedding in both $Y$ and $X$ where $L_{\ref{lem-same limits}}=L_{\ref{lem-same limits}}(\dl,k,\phi)\ge1$.

In particular, we have the following. For all $n\in\N$, let $Z_n\sse Y$ be $k$-quasiconvex in both $X$ and $Y$, and let $y_n,y'_n\in Z_n$. Further, suppose that $\LMX y_n$ and $\LMY y_n$ exist in $\pa X$ and $\pa Y$ respectively. Then $\LMX y'_n$ exists in $\pa X$ and $\LMX y_n=\LMX y'_n$ if and only if $\LMY y'_n$ in $\pa Y$ and $\LMY y_n=\LMY y'_n$.
\end{lemma}

\begin{proof}
By Lemma \ref{qi-emb-in-Y-X} $(1)$, $N^Y_D(Z)$ is $L_1$-qi embedded in $Y$ where $D=k+1$ and $L_1=L_{\ref{qi-emb-in-Y-X}}(k,D)$. Since $Y$ is $\phi$-properly embedded in $X$ and $N^Y_D(Z)$ is $L_1$-qi embedded in $Y$, it follows that $N^Y_D(Z)$ is $\psi$-properly embedded in $X$ where $\psi:\R_{\ge0}\map\R_{\ge0}$ such that $\psi(r)= L_1\phi(r)+L^2_1$. Note also that $N^Y_D(Z)$ is $k_1$-quasiconvex in $X$ where $k_1=D+2\dl+k$. Thus by Lemma \ref{qi-emb-in-Y-X} $(2)$, $N^Y_D(Z)$ is $L_2$-qi embedded in $X$ where $L_2=L_{\ref{qi-emb-in-Y-X}}(k_1,\psi)$. We can take $L_{\ref{lem-same limits}}=max\{L_1,L_2\}$. This completes the first part.

Now a geodesic, say $\alpha_n$, 
		joining $y_n,y'_n$ in $N^Y_D(Z_n)$ is $L$-quasigeodesic in both $X$ and $Y$. Note that $Y\map X$ is $\phi$-proper embedding. Then by the stability of quasigeodesic (Lemma \ref{ml}), for any $y\in Y$, we have $d_Y(y,[y_n, y'_n]_Y)\map \infty$ as $n\map \infty$ if and only if $d_X(y,[y_n,y'_n]_X)\map\infty$ as $n\map\infty$. Finally, the proof follows from Lemma \ref{con-same-pt}.
\end{proof}

Next we discuss when a sequence of subsets in a hyperbolic space converge to a point of the boundary.

\begin{defn}
	Suppose $X$ is hyperbolic metric space and $\{A_n\}$ is a sequence of
	subsets. Suppose $\xi\in \partial X$. We say that the sequence of subsets $\{A_n\}$
	converges to $\xi$ (in $\bar{X}$) if the following holds: Given $R>0$ there is $N\in \N$
	such that for all $n\geq N$ and $x_n\in A_n$ and any geodesic ray $\alpha$ joining
	$x_n$ to $\xi$, we have $d(x,\alpha)>R$.
\end{defn}
Informed reader would notice that the above definition is equivalent to the following
with the natural topology on $\bar{X}$:

{\em A sequence $\{A_n\}$ of subsets in a hyperbolic space $X$ converges to $\xi\in \pa X$ iff
for all neighbourhood $U$ of $\xi$ in $\bar{X}$, there is $N\in \N$ such that $A_n\subset U$
for all $n\geq N$.}

The following lemma gives a criteria for convergence of quasiconvex sets.
\begin{lemma}\label{qc conv criteria}
	1) Suppose $\{A_n\}$ is a sequence of subsets in a hyperbolic metric space
	$X$ that converges to $\xi\in \partial X$. Suppose for all $n\in \N$ we choose $z_n\in A_n$.
	Then $\LMX z_n =\xi$.

	2) Suppose $\{A_n\}$ is a sequence of uniformly quasiconvex subsets in a hyperbolic metric space
	$X$. Then $A_n$'s converge to $\xi\in \partial X$ iff $\lim_{n\map \infty} d(x, A_n)=\infty$
	and there is a sequence $\{z_n\}$ in $X$ where $z_n\in A_n$ for all $n\in \N$, such that
	$\LMX z_n=\xi$.
\end{lemma}
\proof (1) This follows immediately from Lemma \ref{conv criteria}.

(2)
If $A_n\map \xi$ then clearly $d(x,A_n)\map \infty$. Moreover, by (1) if $z_n\in A_n$ for all
$n\in \N$ then $\LMX z_n =\xi$. Conversely suppose $\lim_{n\map \infty} d(x, A_n)=\infty$
	and there is a sequence $\{z_n\}$ in $X$ where $z_n\in A_n$ for all $n\in \N$, such that
	$\LMX z_n=\xi$. We note that since $A_n$'s are uniformly quasiconvex, there is some $k\geq0$
	such that each $A_n$ is $k$-quasiconvex in $X$. Let $\alpha_n$ be a geodesic ray
joining $z_n$ to $\xi$. Let $w_n\in A_n$ be any point and let $\beta_n$ be any geodesic ray
joining $w_n$ to $\xi$. Given $R\geq 0$, let $N\in \N $ be such that  $d(x,A_n)\geq R+2\delta+k$ and
$d(x,\alpha_n)\geq R+k+2\delta$ for all $n\geq N$ where the latter is guaranteed by Lemma \ref{conv criteria}.
 Let $n\geq N$. Now the triangle formed by the three geodesics $\alpha_n$,
$\beta_n$ and $[z_n,w_n]$ is $2\delta$-slim by Lemma \ref{ideal triangle}.
Hence, $d(x, \beta_n)\geq \min\{d(x,\alpha_n), d(x, [z_n, w_n])\}-2\delta$. However, since $A_n$ is $k$-quasiconvex
and $d(x, A_n)\geq R+k+2\delta$, we have $d(x, [z_n, w_n])\geq R+2\delta$. It follows that
$d(x, \beta_n)\geq R$ for all $n\geq N$. \qed\smallskip

The following lemma is an analogue of Lemma \ref{subseq exists} for sets.
\begin{lemma}\label{seq-set-con}
	Suppose $X$ is a proper hyperbolic metric space and $\{A_n\}$ is sequence of uniformly
	quasiconvex subsets in $X$ such that the collection $\{A_n: n\in \N\}$ is locally finite.
	Then there is a subsequence $\{A_{n_k}\}$ of $\{A_n\}$ that converges to a point of
	$\partial X$.
\end{lemma}
We note that {\em a family of subsets $\{A_{\al}\}_{\al\in\Lambda}$ in a metric space $X$ is said
to be {\em locally finite} if any finite radius ball in $X$ intersects only finitely
many $A_{\al}$'s.}
\proof Since the collection $\{A_n\}$ is locally finite, we can assume, up to passing to
a subsequence if needed, that $\lim_{n\map \infty} d(x, A_n)= \infty$. Let $z_n\in A_n$ for all $n\in \N$. Now,
we can apply Lemma \ref{subseq exists} to find a subsequence $\{z_{n_k}\}$ which
converges to a point of $\partial X$. Then by Lemma \ref{qc conv criteria}(2)
$\{A_{n_k}\}$ is a required subsequence. \qed

%%%%%%%%%%%%%%%%%%%%%%%%%%%%%%%%%%%%%%%%%%%%%%%%
%%%%%%%%%%%%%%%%%%%%%%%%%%%%%%%%%%%%%%%%%

\subsubsection{\bf CT maps}
\begin{defn}\label{CT-map}\textup{({\bf Cannon--Thurston map})}
	Suppose $f:Y\ri X$ is a (proper) embedding between hyperbolic metric spaces.
	We say that $f$ admits the Cannon--Thurston (CT) map if there is a map
	$\partial f :\partial Y\map \partial X$ induced by $f$ in the following sense:

	For all $\xi\in \partial Y$ and for any sequence $\{y_n\}$ in $Y$ with
	$\LMY y_n=\xi$ one has $\LMX f(y_n)=\partial f(\xi)$.
\end{defn}

In this case $\partial f$ is called the CT map induced by $f$.
We note that in the Definition \ref{CT-map}, the existence of the CT map implies
that it is also continuous (see e.g. \cite[Lemma $2.50$]{ps-krishna}).
In \cite{mitra-trees}, Mitra gave the following criterion for the existence of CT maps.

\begin{lemma}\textup{({\bf Mitra's Criterion})}
	Suppose $f:Y\ri X$ is a map between hyperbolic metric spaces.
	Fix $y_0\in Y$ and let $x_0=f(y_0)$. Then $f$ admits the CT map if there is a map
	$\phi:\R_{\ge0}\ri\R_{\ge0}$ with $\phi(r)\ri\infty$ as $r\ri\infty$ such that the following
	holds:

	Suppose $y,y\pr\in Y$ are any points and $\alpha $, $\beta$
	are any geodesics joining $y,y'$ in $Y$ and $X$ respectively, and $R\geq 0$.
	Then $d_Y(y_0,\alpha)\geq R$ implies $d_X(x_0,\beta)\geq \phi(R)$.
\end{lemma}
In the situation of the above lemma we shall say that {\em $f$ satisfies Mitra's criterion
	with respect to the base point $y_0$} and we shall refer to the function $\phi$
to be a {\bf CT parameter} for this base point.
We note that Mitra's criterion implies that $f:Y\map X$ is a $\phi$-proper
embedding. On the other hand it is easy to check that if Mitra's criterion holds
for a map $f:Y\map X$ as above with respect to a base point $y_0\in Y$, then the
same will be true for any other base point in $Y$ although in that case the CT parameter
$\phi$ may be different. However if there is a group $G$ acting by isometries
on both $Y$ and $X$ such that $f$ is $G$-equivariant and the $G$-action on $Y$ is
transitive then the same function $\phi$ works for all base points in $Y$.
Typically this is the case in group theoretic situations, i.e. when we have hyperbolic groups
$H<G$ and $f$ is an inclusion map between their Cayley graphs.

\begin{defn}\label{mitra-irr}
	Suppose $f:Y\ri X$ is a proper embedding between hyperbolic metric spaces and that
	$f$ satisfies Mitra's criterion with respect to a base point. We say that
	$f$ satisfies a {\bf uniform Mitra's criterion} if
	there is a function $\phi$ which works as a CT parameter for all base points in $Y$.
\end{defn}
The function $\phi$ in Definition \ref{mitra-irr} will be referred to as a {\em uniform
CT parameter}.
We note that although Mitra's criterion is not necessary for the existence of CT maps,
%\RED{ref} \CYAN{\cite[Example and remarks $3$ after Lemma $.249$]{ps-krishna}},
it is a very reasonable sufficient condition for the
existence of CT maps as the following lemma shows. Since this is quite standard we
skip its proof.

\begin{lemma}\label{necessity}\textup{({\bf Necessity of Mitra's criterion})}
	Suppose $X,Y$ are two proper hyperbolic metric spaces and $f:Y\map X$ is a proper embedding.
	If $f$ admits the CT map, then $f:Y\ri X$ satisfies Mitra's criterion.
\end{lemma}

The following lemmas describe some basic properties of CT maps.

\begin{lemma}\label{qi-im-ct-map}
Suppose $Y$ is a qi embedding (hyperbolic) subspace of a hyperbolic space $X$. Then the inclusion $Y\to X$ admits an (injective) CT map.
\end{lemma}

\begin{lemma}\textup{(\cite[Lemma $2.1$]{mitra-pams})}\label{lem-qc and inj ct}
Suppose $H$ is a hyperbolic group of a hyperbolic group $G$. Then $H$ is quasiconvex in $G$ if and only if the inclusion $H\map G$ admits an injective CT map.
\end{lemma}

\begin{lemma}\textup{(\cite[Lemma $8.6$]{ps-kap})}\label{functo-ct-map}
Suppose $f:Z\ri Y$ and $g:Y\ri X$ are maps between hyperbolic spaces both admitting the CT-maps. Then the composition $g\circ f:Z\ri X$ admits the CT map and $\partial (g\circ f)=\pa g \circ \pa f$.
\end{lemma}

We note that all the spaces in consideration in this paper, for which CT maps are to be
discussed, are proper. For proofs, the following lemma will be very useful. Since the proof is immediate using Lemma \ref{subseq exists}
we skip it.

\begin{lemma}\label{not-ct-not-comp} {\em ({\bf Nonexistence criteria})}
	Suppose $X,Y$ are two proper hyperbolic metric spaces, and $f:Y\map X$ is a proper
	embedding which does not admit the CT map. Then there are two unbounded sequences $\{y_n\}$,
	$\{y'_n\}$ in $Y$ such that $\LMY y_n=\LMY y'_n\in \partial Y$ but
	$\LMX y_n\neq \LMX y'_n$ in $\pa X$.
\end{lemma}

\subsubsection{\bf Coarse separation and coarse intersection}\label{subsubsec-coarse sep}

\begin{defn}{\em ({\bf Coarse separation})}
Suppose $X$ is a geodesic metric space. Let $Y_1,Y_2, Z\sse X$ be nonempty subsets. We say that $Z$ coarsely separates $Y_1$ and $Y_2$ if for all $k\ge1$ there is $R\ge0$ such that for all $k$-quasigeodesic $\gm$ joining $y_1$ and $y_2$ where $y_i\in Y_i$, we have $\gm\cap N_R(Z)\ne\emptyset$.
\end{defn}

A bit general form of coarse separation is as follows.

\begin{defn}\label{coarse-inter}{\em ({\bf Coarse intersection})}
	Suppose $X$ is a metric space and $Y_1, Y_2, Z\subset X$. We say that $Z$ is a {\em coarse intersection}
	of $Y_1,Y_2$ in $X$ if the following hold. (1) There is a constant $r\geq 0$ such that
	$Z\subset N_r(Y_1)\cap N_r(Y_2)$. (2) For any $C\geq 0$ there is $D\geq 0$ such that
	$N_C(Y_1)\cap N_C(Y_2)\subset N_D(Z)$.
\end{defn}
It is clear that coarse separation implies coarse intersection provided $Z\sse N_r(Y_1)\cap N_r(Y_2)$ for some $r\ge0$. However, we will show in Lemma \ref{lem-coarse int iff coarse sep} that, in a hyperbolic space, the two notions are equivalent when both $Y_1$ and $Y_2$ are quasiconvex.
\begin{example}\label{exp-coarse sep}
Suppose $G$ is a $\delta$-hyperbolic group with the word metric $d_G$ for some finite generating set $S$. Assume that $H, K<G$ are $k$-quasiconvex subgroups.  Then for all $g\in G$, $gKg^{-1}$ is also $k'$-quasiconvex for some $k'\ge0$ depending on $\dl$, $k$ and $d_G(1,g)$. Thus $H\cap gKg^{-1}$ is $k''$-quasiconvex for some $k''\ge0$ depending on $\dl$, $k$ and $k'$ \textup{(\cite[Proposition $3$]{short})}. On the other hand, given $L\ge0$ there is a constant $L'\ge0$ depending on $H$, $gK$ and $S$ such that $N_L(H)\cap N_{L}(gK)\sse N_{L'}(H\cap gKg^{-1})$ by \textup{\cite[Lemma $4.5$]{HW-bdd-pck}}. Note that $Hd_G(gK,gKg^{-1})\le d_G(1,g)$. This concludes that $H\cap gKg^{-1}$ is a coarse intersection of $H$ and $gKg^{-1}$.
	
Moreover, it follows that for all $k\in K$, $k(H\cap K)$ is a coarse separation of $kH$ and $K$.
\end{example}

\begin{lemma}\label{lem-coarse int iff coarse sep}
Suppose $X$ is a $\dl$-hyperbolic metric space, and $Y_1,Y_2$ are $k$-quasiconvex subsets of $X$. Let $Z\sse X$ be (non-empty) such that $Z\sse N_r(Y_1)\cap N_r(Y_2)$ for some $r\ge0$. Then $Z$ is a coarse separation of $Y_1$ and $Y_2$ if and only if $Z$ is a coarse intersection of $Y_1$ and $Y_2$.	
\end{lemma}

\begin{proof}
We will only prove that coarse intersection implies coarse separation as the other direction follows easily. Let $y_i\in Y_i$. We show that $[y_1,y_2]$ passes through a uniform ball of $Z$, and this will completes the proof by the stability of quasigeodesic (Lemma \ref{ml}) as $X$ is hyperbolic. Let $z_i\in Y_i$ and $z\in Z$ be such that $d(z,z_i)\le r$. Let $x_i\in[y_i,z]$ be such that $d(x_1,x_2)\le \dl$ and $d(x_i,[y_1,y_2])\le \dl$. Since $Y_i$ is $k$-quasiconvex, $z_i\in Y_i$ and $d(z,z_i)\le r$, there is $w_i\in Y_i$ such that $d(x_i,w_i)\le\dl+r+k$. Thus $x_1\in N_C(Y_1)\cap N_C(Y_2)$ where $C=2\dl+r+k$. Hence by coarse intersection property, there is $D\ge0$ such that $d(x_1,w)\le D$ for some $w\in Z$. On the other hand, the arc length parametrization of $[y_1,x_1]\cup[x_1,x_2]\cup[x_2,y_2]$ is a uniform quasigeodesic; see for instance Lemma \ref{lem-qg}. Since $X$ is hyperbolic, by the stability of quasigeodesic, for any $K$-quasigeodesic joining $y_1$ and $y_2$ passes through a uniform ball centered at $x_1$; hence passes through a uniform ball of $Z$. Since $y_i\in Y_i$ was arbitrary, this completes the proof.
\end{proof}	

%Condition $(2)$ of the following lemma will be used throughout the rest of this section. 

%\BLUE{I made changes from here.}

The following definition will also appear in Sections \ref{sec-subtree of spaces} and \ref{sec-main thm} for the {\em fibers of trees of metric spaces}.

\begin{defn}[\bf Compatible projection condition]\label{defn-compatible proj}
	Suppose we have a hyperbolic space $X$ and a properly embedded hyperbolic subspace
	$Y$. Suppose $A$ is a quasiconvex subset of $X$ such that $B= Y\cap A$ is quasiconvex
	in both $X$ and $Y$. We will say that the quadruple $(X,Y,A,B)$ satisfies
	the compatible projection condition if the following holds: \\
	There is a constant $R\geq 0$ such that for all $y\in Y$  we have
	$$d_X(P^{X}_A(y), P^{Y}_{B}(y))\leq R.$$
	
	Sometimes, we say that the quadruple $(X,Y,A,B)$ satisfies the compatible projection condition with a constant $R$.
\end{defn}

\begin{example}\label{exp-compatible projection}
Suppose we have the assumption of Example \ref{exp-coarse sep}. Then by Lemma \ref{lem-proj iff coarse intersec} below it is easy to see that the quadruple $(G,K,H,H\cap K)$ satisfies the compatible projection condition.
\end{example}

Before going into Lemma \ref{lem-proj iff coarse intersec}, in the following lemma, we will characterize the compatible projection condition in terms of a limit set intersection property, extending Example \ref{exp-compatible projection} in a more general setting.

\begin{lemma}\label{lem-char of proj cond}
Suppose $H<G$ are hyperbolic groups, and $K, H\cap K<G$ are quasiconvex subgroups. Assume that the inclusion $H\to G$ admits a CT map. Then the quadruple $(G,H,K,H\cap K)$ satisfies the compatible projection condition if and only if $\Lambda_G(H)\cap\Lambda_G(K)=\Lambda_G(H\cap K)$.
\end{lemma}

\begin{proof}
Let $L=H\cap K$. Suppose $Y\sse X$ are Cayley graphs of $H$ and $G$ respectively with respect to some finite generating sets. Suppose the inclusion $Y\to X$ is $\eta$-proper embedding for some proper map $\eta:\N\to\N$.

 $(\Longrightarrow)$  Note that $\Lambda_G(L)\sse\Lambda_G(H)\cap \Lambda_G(K)$. Let $\{h_n\}\sse H$ and $\{k_n\}\sse K$ be such that $\LMX h_n=\LMX k_n=\xi\in\Lambda_G(H)\cap\Lambda_G(K)$. Suppose $P^X_L(h_n)=h'_n$ and $P^X_K(h_n)=h''_n$. Then by the given condition, $d_X(h'_n,h''_n)\le D$ for some fixed $D\ge0$. Note that $\al_n=[h_n,h''_n]_X\cup[h''_n,k_n]_X$ is a uniform quasigeodesic in $X$ (Lemma \ref{lem-proj on pair qc}~\ref{lem-proj on pair qc:concatenation qg}), and $d_X(\al_n,h'_n)\le D$. This shows that $\xi\in\Lambda_G(L)$ (see Lemma \ref{con-same-pt}).

 $(\Longleftarrow)$ On contrary, suppose not. After translating, we assume that $\{h_n\}\sse H$ such that $P^Y_{L}(h_n)=1$, $P^X_K(h_n)=k_n$ and $d_X(1,k_n)\ge n, d_X(1,h_n)\ge n$. 
 
  Let $P^X_L(k_n)=l_n$. Note that $\al_n=[h_n,k_n]_X\cup[k_n,1]_X$ and $[k_n,l_n]_X\cup[l_n,1]_X$ are uniform quasigeodesics in $X$ (Lemma \ref{lem-proj on pair qc}~\ref{lem-proj on pair qc:concatenation qg}). Then $d_X(\al_n,l_n)\le D$ for some fixed $D\ge0$. Thus $d_X(l_n,[h_n,1]_X)\le D'$ for some fixed $D'\ge0$ (Lemma \ref{ml}).
 
 Now $Y\to X$ satisfies a uniform Mitra's criterion with the parameter function $\phi:\N\to\N$. Let $R\in\N$ be such that $\phi(R)>D'$. Therefore, $d_X(l_n,[h_n,1]_Y)\le R$; otherwise, $d_X(l_n,[h_n,1]_X)\ge\phi(R)>D'$. That is, we have $l'_n\in[h_n,1]_Y$ such that $d_X(l_n,l'_n)\le R$, and hence $d_Y(l_n,l'_n)\le\eta(R)$. Since $1$ is a nearest point projection of $h_n$ in $Y$ on $L$, $l'_n\in[h_n,1]_Y$ and $d_X(l'_n,l_n)\le \eta(R)$ for some $l_n\in L$, so $d_Y(l'_n,1)\le\eta(R)$. Thus, by triangle inequality, $d_Y(1,l_n)\le2\eta(R)$.
 
 Hence $\{P^X_L(k_n):n\in\N\}\sse B(1;2\eta(R))$ (ball of radius $2\eta(R)$ centered at $1$). Note that $d_X(L,[h_n,k_n]_X)\to\infty$ as $n\to\infty$. Otherwise, let $p_n\in L$ be such that $d_X(p_n,[h_n,k_n]_X)\le D_0$ for some constant $D_0\ge0$. This shows that $d_X(k_n,p_n)\le2D_0$, and so $d_X(k_n,P^X_L(k_n))=d_X(k_n,l_n)\le2D_0$. Hence $\{k_n\}$ is a bounded sequence $-$ which is a contradiction.
 
 Therefore, in particular, $d_X(1,[h_n,k_n]_X)\to\infty$ as $n\to\infty$. After passing to a sequence, if necessary, we assume that $\LMX h_n,\LMX k_n$ exist. Hence, $\LMX h_n=\LMX k_n\in\Lambda_G(H)\cap\Lambda_G(K)=\Lambda_G(L)$. This gives a contradiction to the fact that $\{P^X_L(h_n):n\in\N\}\sse B(1;2\eta(R))$ is bounded. This completes the proof.
\end{proof}

A characterization of compatible projection condition in terms of coarse separation and coarse intersection.

\begin{lemma}\label{lem-proj iff coarse intersec}
Let $\dl\ge0, ~k\ge0,~L\ge1$. Suppose $Y\sse X$ are $\dl$-hyperbolic metric spaces where $Y$ is equipped with the induced path metric from $X$. Assume that the inclusion $Y\map X$ is $L$-qi embedding. Suppose $A$ is a $k$-quasiconvex subset of $X$ such that $B= Y\cap A$ is $k$-quasiconvex
in both $X$ and $Y$.
%Let $A\sse Y$ and $A\sse B\sse X$ be such that $A$ is $k$-quasiconvex in $Y$ as well as in $X$, and $B$ is $k$-quasiconvex in $X$. 
Then the following are equivalent.
\begin{enumerate}
\item $B$ coarsely separates $Y$ and $A$ in $X$.

\item The quadruple $(X,Y,A,B)$ satisfies the compatible projection condition

%There is a constant $R\ge0$ %depending on $\dl$, $k$ and $L$ such that for all $y\in Y$, we have $d_X(P^Y_A(y),P^X_B(y))\le R$.

\item For any $C\ge0$ there is a constant $D\ge0$ such that $N^X_C(Y)\cap N^X_C(A)\sse N^X_D(B)$, i.e., $B$ is a uniform coarse intersection of $Y$ and $A$ in $X$.
\end{enumerate}

Moreover, the constants in the three conditions are mutually determined by $\dl$, $L$ and $k$.
\end{lemma}

\begin{proof}
Note that $Y$ and $A$ are quasiconvex in $X$ (see Lemma \ref{qc morph}). Moreover, $B\sse Y\cap A$. Hence $(1)\Longleftrightarrow(3)$ follows from Lemma \ref{lem-coarse int iff coarse sep}. We will prove that $(1)$ is equivalent to $(2)$.

\noindent$(2)\Longrightarrow(1)$: Let $K\ge1$ and $\gm$ is a $K$-quasigeodesic joining $y\in Y$ and $a\in A$. Let $y_1=P^X_{A}(y)$ and $y_2=P^Y_{B}(y)$.
Then by $(2)$, we have $d_{X}(y_1, y_2)\leq R$ for some uniform $R\ge0$. Now $[y,y_1]_X\cup[y_1,a]_X$ is a $(3+2k)$-quasigeodesic in $X$ by Lemma \ref{lem-proj on pair qc}~\ref{lem-proj on pair qc:concatenation qg}. By the stability of quasigeodesic (Lemma \ref{ml}), $d_X(y_1,\gm)\le D'$ for some constant $D'\ge0$ depending on $\dl$, $K$ and $k$. Thus $d_X(y_2,\gm)\le R+D'$.

\noindent$(1)\Longrightarrow(2)$: Let $y\in Y$ and let $y'= P^X_A(y)$ and $y''=P^Y_B(y)$.
Now, by Lemma \ref{lem-proj on pair qc}~\ref{lem-proj on pair qc:concatenation qg} the arc length parametrization of
$[y,y']_X\cup[y',y'']_X$ is $(3+2k)$-quasigeodesic in $X$ joining $y,y''\in Y$.
Since $Y$ is $L$-qi embedded in $X$, it follows that $y'$ is $D_1$-close
to a point $y'''\in [y,y'']_Y$ for some constant $D_1\ge0$ depending only on $\dl$, $(3+2k)$ and $L$ (see Lemma \ref{ml}). Note that $d_X(y',y''')\le D_1$ and $y'\in A,~y'''\in Y$. If $B$ is coarse separation of $Y$ and $A$, there is a constant $R'\ge0$ depending on $D_1$, $\dl$ such that $y'''$ is $R'$ close to $B$ in $X$. Thus $d_Y(y''',b)\le LR'+L$ for some $b\in B$. Then $d_Y(y''',y'')\le LR'+L$ as $P^Y_B(y)=y''$ and $y'''\in [y,y'']_Y$.

Therefore, by triangle inequality, we have $d_{X}(y', y'')\le L.R'+L+D_1$. This completes the proof.
\end{proof}

Note that the forward direction of Lemma \ref{lem-char of proj cond} holds in general hyperbolic metric spaces and does not require a group structure. However, the converse direction relies essentially on group properties. In the following result, we see an instance where both the limit set intersection property and the compatible projection condition hold without a group structure. %(It is easy to see $\Lambda_X(Y)\cap\Lambda_X(A)=\Lambda_X(B)$ in the following result.)

%In the following result, we present another variant in which the compatible projection condition holds, without requiring $Y$ to be qi embedded.
\begin{lemma}\label{uni-mitra-proj}
Given $\dl\ge0$, $k\ge0$, $L\ge0$ and a proper function $\psi:\R_{\ge0}\ri\R_{\ge0}$,
	there is a constant $R=R(\dl, k, L, \psi)\ge0$ such that the following holds:\\
	Suppose $Y\sse X$ are $\dl$-hyperbolic metric spaces where $Y$ is equipped with the induced path metric from $X$. %Suppose $Y$ is $\delta$-hyperbolic subspace of a $\delta$-hyperbolic space $X$ such that 
	Assume that the
	inclusion $Y\to X$ satisfies uniform Mitra's criterion with uniform CT parameter $\psi$.
	Suppose $A$ is a quasiconvex subset of $X$ such that $B= Y\cap A$ is quasiconvex
	in both $X$ and $Y$. Moreover, $Hd_X(A,B)\le L$. Then the quadruple $(X,Y,A,B)$ satisfies the compatible projection condition.
\end{lemma}

\begin{proof}
Suppose $y\in Y$. Let $y'=P^X_A(y)$ and $y''=P^Y_B(y)$. Let $b\in B$ be such that $d_X(y',b)\le L$. Now by Lemma \ref{lem-proj on pair qc}~\ref{lem-proj on pair qc:concatenation qg}, the arc length parametrization of $[y,y']_X\cup[y',y'']_X$ is a $(3+2k)$-quasigeodesic in $X$. Note that the inclusion $Y\map X$ satisfies the Mitra's criterion with the base point $b\in Y$ and the proper map $\psi$. Now by the stability of quasigeodesic, $[y,y'']_X$ passes through $D$-radius ball centered at $y'$ in $X$ for some constant $D\ge0$ depending on $\dl$ and $(3+2k)$ (see Lemma \ref{ml}). Thus by triangle inequality, $d_X(b,[y,y'']_X)\le D+L$. %Let $\bar{\psi}$ is the proper map obtained in Lemma \ref{rev-of-pro-map} depending on $\psi$. 
We claim that $d_Y(b,[y,y'']_Y)\le r$ for some $r$ such that $\psi(r)> D+L$. (Note that such $r$ exists as $\psi$ is a proper map.) Indeed, if not, then $d_Y(b,[y,y'']_Y)>r$ implies $D+L\ge d_X(b,[y,y'']_X)\ge\psi(r)> D+L$.

 Let $y'''\in[y,y'']_Y$ such that $d_Y(b,y''')\le r$. Since $y''$ is a nearest point projection of $y$ on $B$ and $y'''\in[y,y'']_Y$, so $d_Y(y''',y'')\le r$. Hence by triangle inequality, we have $d_X(y',y'')\le d_X(y',b)+d_X(b,y''')+d_X(y''',y'')\le L+2r$. We take $R=L+2r$. This completes the proof.
\end{proof}

%In Lemma \ref{lem-proj iff coarse intersec}, we have seen that the projection condition is equivalent to the statement that $A$ coarsely separates $Y$ and $B$ in $X$, provided $Y$ is a qi embedded subspace of $X$. Equivalently, under the projection condition, any geodesic in $X$ joining points of $Y$ and $B$ must pass through a uniformly bounded neighborhood of $A$. 

%Suppose we have the assumption of Lemma \ref{lem-proj iff coarse intersec}. Under the compatible projection condition for any unbounded subset $W\subset Y$, one intuitively expect that $\Lambda_Y(W)\cap\Lambda_Y(B)=\Lambda_X(W)\cap\Lambda_X(A)$; think of $\pa Y$ sits in $\pa X$. The following two results make this idea precise in a more general setting, namely, when $Y$ is not required to be qi embedded in $X$.

In the forward direction of Lemma \ref{lem-char of proj cond}, we have $\lim^G_{n\to\infty}h_n=\lim^G_{n\to\infty}k_n\in\Lambda_G(L)$ for some sequences $\{h_n\}\sse H$, $\{k_n\}\sse K$ provided $\Lambda_G(L)\ne\emptyset$. By restricting the sequence $\{h_n\}$ to  geodesic ray, we can get something more as obtained in the following result.
\begin{lemma}\label{qua-im-qua}
	Let $\dl\ge0$, $R\ge0$, $k\ge0$, and $\eta:\R_{\ge0}\map\R_{\ge0}$ be a proper map.	
	
	Suppose $Y\sse X$ are proper $\dl$-hyperbolic spaces where $Y$ is equipped with the induced path metric from $X$. Further, assume that the inclusion $i:Y\map X$ is an $\eta$-proper embedding. %and admits a CT-map. 
	Suppose $A$ is a $k$-quasiconvex subset of $X$ such that $B=Y\cap A$ is $k$-quasiconvex in both $X$ and $Y$. Moreover, assume that the quadruple $(X,Y,A,B)$ satisfies the compatible projection condition with the constant $R$.

	%Let $A\sse Y$ and $A\sse B\sse X$ be such that $A$ is $k$-quasiconvex in $Y$ as well in $X$, and $B$ is $k$-quasiconvex in $X$. Moreover, assume that for all $y\in Y$, $$d_X(P^Y_A(y),P^X_B(y))\le R.$$
	
Furthermore, let $\al$ be a geodesic ray in $Y$ such that $\LMX\al(n)\in\Lm_X(A)\sse\pa X$.
	
	Then for any geodesic ray $\bt$ in $X$ with $\bt(\infty)=\LMX\al(n)$, we have $$Hd_X(\al,\bt)<\infty\text{ and }\textrm{lim}^Y_{n\map\infty}\al(n)\in\Lm_Y(B)\sse\pa Y.$$
	
	In particular, $\al$ is a quasigeodesic ray in $X$. 
\end{lemma}

\begin{proof}
Fix a point $y_0\in B$. Without loss of generality, we  may assume that $\al$ starts at $y_0$, i.e., $\al(0)=y_0$. Let $P^Y_B(\al(n))=y_n$ and $P^X_A(\al(n))=x_n$. Note that $[y_0,y_n]_Y\cup[y_n,\al(n)]_Y$ and $[y_0,x_n]_X\cup[x_n,\al(n)]_X$ are $(3+2k)$-quasigeodesics in $Y$ and in $X$ respectively (see Lemma \ref{lem-proj on pair qc}~\ref{lem-proj on pair qc:concatenation qg}). Note also that for any $x\in A$, $[\al(n),x_n]_X\cup[x_n,x]_X$ is also $(3+2k)$-quasigeodesic in $X$ (by Lemma \ref{lem-proj on pair qc}~\ref{lem-proj on pair qc:concatenation qg}). %Thus $\{x_n\}$ has to be unbounded as $\LMX\al(n)\in\Lm_X(B)$ (by Lemma \RED{ref}). 
	Hence, by the stability of quasigeodesic (Lemma \ref{ml}), and Lemmas \ref{con-same-pt} and \ref{lem-close imp same limit}, one concludes that $\LMX x_n=\LMX\al(n)$. %Note that $\{y_n\}$ is unbounded in $Y$ as $\{x_n\}$ is so and $d_X(x_n,y_n)\le R$ (by the assumption). 
Suppose $\bt_n$ is a geodesic in $X$ joining $y_0$ and $x_n$, and that $\bt_n$ converges to a (quasi)geodesic ray, $\bt$ say, in $X$.
	
	%Then $\LMX y_n=\LMX x_n$ (see Lemma \RED{ref}).	We also have $\LMY y_n=\al(\infty)$.

Since $B$ is $k$-quasiconvex in both $X$ and $Y$, by Lemma \ref{lem-same limits}, $N^Y_{k+1}(B)$ (equipped with the path metric induced from $Y$) is $L$-qi embedded in both $X$ and $Y$ for some constant $L=L_{\ref{lem-same limits}}(\dl,k,\eta)\ge1$. Again $d_X(x_n,y_n)\le R$ for all $n$ (by the assumption). Hence by the stability of quasigeodesic (Lemma \ref{ml}), we have $Hd_X(\bt_n,[y_0,y_n]_Y)\le D$ for some constant $D\ge0$ depending only on $\dl$, $L$ and $R$.
	
	By the stability of quasigeodesic, we also have $t_n\in[0,n]$ such that $d_Y(y_n,\al(t_n))\le D_1$ where $D_1=D_{\ref{ml}}(\dl,3+2k)$. Hence $Hd_Y([y_0,y_n]_Y,\al|_{[0,t_n]})\le D_1+\dl$. %Thus by triangle inequality, $d_X(x_n,\al(t_n))\le R+D_1$. 
	Thus by triangle inequality, we conclude that $Hd_X(\bt_n,\al|_{[0,t_n]})\le D+D_1+\dl$. Finally, the result follows from the fact that $\bt_n$ converges to $\bt$.
	
	Since $Y$ is properly embedded and $Hd_X(\al,\bt)<\infty$, $\al$ is a quasigeodesic in $X$ by Lemma \ref{lem-giving qg}.
\end{proof}

As a simple corollary of Lemma \ref{qua-im-qua}, we have the following. Second statement of Corollary \ref{cor-line-imp-line} follows from Lemma \ref{lem-giving qg}. %We leave the proof for the reader.

\begin{cor}\label{cor-line-imp-line}
Suppose we have the assumption of Lemma \ref{qua-im-qua}. Assume that $\al$ is a geodesic line in $Y$ such that $$\textup{lim}^X_{n\map\infty}\al(\pm n)\in\Lm_X(A)\sse\pa X.$$ Then there are constants $D\ge0$ and $K\ge1$ depending only on $\dl$, $R$, $k$ and $\eta$ satisfying the following.
	
	For any geodesic line $\bt$ in $X$ with $\bt(\pm\infty)=\LMX\al(\pm n)$, we have $Hd_X(\al,\bt)<D$. In particular, $\al$ is a $K$-quasigeodesic line in $X$.
\end{cor}

\section{Geometry of trees of metric spaces}\label{trees-of-metric-sps}
The notion of {\em trees of metric spaces} was introduced by Bestvina and Feighn in \cite{BF}.
Inspired by that, later, an equivalent definition was used by Mitra in \cite{mitra-trees}.
We are going to adopt Mitra's definition. However, we restate some part of it for better
clarity.
\subsection{Basics of trees of hyperbolic spaces}
\begin{defn}\label{tree-of-sps}
	Suppose $T$ is a tree and $(X, d)$ is a geodesic metric space. Suppose that there is a surjective,
	$1$-Lipschitz map
	$\pi:X\ri T$. For any point $b\in T$, let $X_b=\pi^{-1}(b)$ and let $d_b$ denote the length metric
	on $X_b$ induced from $X$. Also for any unoriented edge $e$ of $T$ let $m_e$ denote the midpoint of $e$.
	Let $X_e$ denote $X_{m_e}$ and let $d_e$ denote the metric $d_{m_e}$.
	%(We always work with the metric $d_b$ on $X_b$ in what follows. )

	Then we say that $\pi:X\map T$ is a {\bf tree of hyperbolic metric spaces with qi embedding condition}
	%(or that $X$ is a tree of hyperbolic metric spaces over $T$ with qi embedding condition when $\pi$ is 	understood)
	if there is a (proper) function $\eta_0: \RR \ri\RR$ and constants $\delta_0 \geq 0$ and $L_0\geq 1$ with the following properties:
	\begin{enumerate}
		\item  For all $b\in T$, $X_b$ is a $\delta_0$-hyperbolic, geodesic metric space.
		\item The inclusion map $X_b\map X$ is an $\eta_0$-proper embedding for all $b\in T$.
		\item For any (unoriented) edge $e$ of $T$ joining two vertices $v,w$, there is a map
		$f_e: X_e \times[0,1]\ri\pi^{-1}(e)\sse X$ such that the following hold.
		\begin{enumerate}
			\item $\pi\circ f_e$ is the projection map onto $[v,w]$.

			\item $f_e$ restricted to $X_e\times(0,1)$ is an isometry onto the pre-image (under $\pi$) of the interior of $e$ equipped with the
			path metric.

			\item $f_e$ restricted to $X_e \times \{0\}$ and $X_e \times \{1\}$ are $L_0$-qi embeddings into $X_v$ and $X_w$ respectively.

		\end{enumerate}
	\end{enumerate}
\end{defn}

\begin{convention}\label{con-trees of metric spaces}
Following important terminologies and conventions will be used in the rest of the paper:
\begin{enumerate}
\item {\bf Vertex spaces:} In the rest of the paper, the spaces $X_v$, $v\in V(T)$ will be referred to as the {\bf vertex spaces} of the tree of spaces $\pi:X\map T$.

\item {\bf Edge spaces:} $X_e$'s where $e\in E(T)$ are referred to as the {\bf edge spaces}.

\item {\bf Incidence maps:}  The natural composition maps $X_e\map X_e\times \{v\}\map X_v$ and
$X_e\map X_e\times \{w\}\map X_w$ will be denoted by $f_{e,v}$ and $f_{e,w}$ respectively and
they will be referred to as the {\bf incidence maps}.

\item We shall denote $Im(f_{e,v})$ by $X_{ev}$ and $Im(f_{e,w})$ by $X_{ew}$.
By Lemma \ref{qc morph}, $X_{ev}$ and $X_{ew}$ are uniformly quasiconvex in $X_v$ and $X_w$
respectively. We shall assume that they are $\lambda_0$-quasiconvex.
\item The constants $\eta_0, \delta_0, L_0, \lambda_0$ will be referred to as the {\bf parameters}
of the tree of hyperbolic metric spaces.
\item For any subtree $T_1\subset T$ we shall denote $\pi^{-1}(T_1)$ by $X_{T_1}$ which will
be endowed with the induced path metric from $X$.

\item {\bf QI sections and QI lifts:}\label{lift}
	Suppose $T\pr$ is a subtree of $T$ and $k\geq 1$ is an arbitrary constant.
	A $k$-{\bf qi section} over $T\pr$ will mean $k$-qi embedding $s:T\pr\ri X$ such that
	$\pi\circ s=Id_{T\pr}$ on $V(T')$.

	It is easy to verify that a set theoretic section on $V(T')$ can be made to a qi section if for all edge $e=[u,v]\in E(T\pr)$, $d_{X_{uv}}(s(u),s(v))\le k$, where $X_{uv}=\pi^{-1}([u,v])$.

	If $T\pr=[v,w]$ is a geodesic segment in $T$ joining $v,w\in V(T)$ then $s:[v,w]\ri X$
	will be called a $k$-{\bf qi lift} of the geodesic $[v,w]$. %\RED{repetition??::}

%	By a $k$-qi lift of a geodesic segment $\gm:[0,l]\ri[u,v]\sse T$ joining $u,v\in V(T)$ is a $k$-qi embedding $\bar\gm:[0,l]\ri X$ such that $\pi\circ\bar\gm(r)=\gm(r)$ for all $r\in[0,l]\cap\Z$.

{\em We shall assume that for a qi section (respectively, a qi lift) $s:T'\map X$ we have $\pi(Im(s))=V(T')$.}

\item \label{defn-flaring condition}{\bf The Bestvina--Feighn flaring condition:}
Suppose $k\ge1$. A tree of metric spaces $\pi:X\ri T$ is said to satisfy $k$-flaring if there are constants $M_k,~n_k\in\N$ and $\lm_k>1$ such that the following holds.

For every pair $(\gm_0,\gm_1)$ of $k$-qi lifts of a geodesic $\gm:[-n_k,n_k]\map T$ with $$d_{\gm(0)}(\gm_0(0),\gm_1(0))\ge M_k$$we have $$\lm_kd_{\gm(0)}(\gm_0(0),\gm_1(0))\le max\{d_{\gm(-n_k)}(\gm_0(-n_k),\gm_1(-n_k)),~d_{\gm(n_k)}(\gm_0(n_k),\gm_1(n_k))\}.$$

We say that the tree of metric spaces satisfies (Bestvina--Feighn) flaring
condition if it satisfies $k$-flaring for all $k\geq 1$.

\end{enumerate}
\end{convention}
Now, we collect some basic results on trees of hyperbolic spaces for later use. Lemma \ref{lem-edge sps loc finite} follows easily from the structure of trees of spaces.

\begin{lemma}\label{lem-edge sps loc finite}
	Suppose $\pi:X\map T$ is a tree of metric spaces such that $X$ is a proper metric spaces. Let $u$ be a vertex of $T$ and $\Lambda$ be the set containing all the edges incident on $u$. Then $\{X_{e_{\lm}u}:e_{\lm}\in\Lambda\}$ is a locally finite collection subsets of $X_u$.
\end{lemma}

\begin{prop}\label{pro-emb}\textup{(\cite[Proposition $2.17$]{ps-kap})}
	Suppose $X\map T$ is a tree of metric spaces and $T'\subset T$ is a subtree.
	Then the inclusion $X_{T\pr}\map X$ is $\eta_{\ref{pro-emb}}$-proper embedding,
	where $\eta_{\ref{pro-emb}}$ depends only on $\eta_0$ in the Definition \ref{tree-of-sps}.
\end{prop}

\begin{comment}
\BLUE{I think Lemma \ref{com-two-hyp-sps} is not required with its specified constants since the definition of flow is different. It is required in the section flow spaces as a lemma.} Let us fix some secondary constants $\bm{\dl\pr_0,L\pr_0,\lm\pr_0,L\pr_1}$ in the following Lemma \ref{com-two-hyp-sps} and the notations are fixed through out the paper unless otherwise stated. %\RED{This may be needed in the section on flow spaces:}
\end{comment}

\begin{lemma}\textup{(\cite[Corollary $2.62$, Lemma $2.27$]{ps-kap})}\label{com-two-hyp-sps}
	Suppose $\pi:X\map T$ is a tree of hyperbolic spaces with the qi embedded condition with the parameters $\eta_0,\dl_0,L_0$ as in Definition \ref{tree-of-sps}. Let $e=[u,v]\sse T$ be an edge. Then we have constants $L'_0\ge 1$, $\lm'_0\ge0$ and $L'_1\ge0$ depending on $\eta_0$, $\dl_0$ and $L_0$ satisfying the following.
	
	$(1)$ $X_{uv}:=\pi^{-1}([u,v])$ is $\dl\pr_0$-hyperbolic and $X_u,X_v$ are $L\pr_0$-qi embedded in $X_{uv}$.
	
	We also assume that $X_{eu}$ (resp., $X_{ev}$) is $\lm_0$-quasiconvex in $X_u$ (resp., $X_v$) (see Convention \ref{con-trees of metric spaces} $(4)$).
	
	$(2)$ Suppose $U$ is a $2\dl_0$-quasiconvex subset in the fiber $X_u$ or $X_v$. Then $U$ is $\lambda\pr_0$-quasiconvex in $X_{uv}$. In particular, $X_u,X_v$ are $\lambda\pr_0$-quasiconvex in $X_{uv}$. Thus a nearest point projection map $P^{X_{uv}}_{X_u}:X_{uv}\ri X_v$ or $P^{X_{uv}}_{X_v}:X_{uv}\ri X_v$ in the path metric $X_{uv}$ is $L\pr_1$-coarsely Lipschitz.
\end{lemma}

%%%%%%%%%%%%%%%%%%%%%%%%%%%%%%%%%%%%%

\begin{theorem} \textup{(\cite{BF} and \cite[Lemma $2.56$]{ps-kap})}\label{thm-nec suff of flaring}
Suppose $\pi:X\map T$ is a tree of hyperbolic metric spaces with qi embedding condition.
Then $X$ is hyperbolic if and only if the tree of spaces satisfies flaring condition.

In particular, if $S\subset T$ is a subtree and $X$ is hyperbolic then $X_S$ is also hyperbolic.
\end{theorem}

The purpose of our paper was to extend the following theorem.

\begin{theorem}\label{thm-CT-ps-kap} \textup{(\cite[Theorem $8.11$]{ps-kap})}
Suppose $\pi:X\map T$ is a tree of hyperbolic metric spaces with qi embedding condition.
Suppose $X$ is also hyperbolic. Then for any subtree $S\subset T$ the inclusion $X_S\map X$
admits a CT map.
\end{theorem}

This paper makes substantial use of the results and proof techniques of
the book \cite{ps-kap}. In this books three general constructions are used
repeatedly, namely `flow spaces', `ladders' and `boundary flow spaces'.
Therefore, we shall briefly recall them here for the ease of reference.
However, we modify these definitions to make them more readable and intuitive.
On the other hand, this makes this paper more self contained.

\begin{convention}\label{con-1}
$(1)$ For the rest of this section, we shall assume that $\pi:X\map T$ is a tree of
hyperbolic metric spaces satisfying the qi embedded condition, with the parameters
$\eta_0,\dl_0,L_0,\lm_0$ as in Definition \ref{tree-of-sps}. However, we shall not
assume that $X$ is hyperbolic. We shall explicit mention when we need hyperbolicity.

%$(2)$ \RED{revise:} Given non-negative constants $a,b,c$, when we define a constant (in our statements) such as $k_{i,j}=k_{i.j}(a,b,c)$ where $i,j$ in suffix refer the statement number, we mean that $k_{i.j}$ depends on the constants $a,b,c$. Note also that $k_{i.j}$ will depend on the parameters of the tree of spaces $\pi:X\map T$; however, we suppress them in our notation for simplicity.

$(2)$ In our statements, the constants obtained from the given constants also depend on the parameters of the tree of spaces $\pi:X\map T$. For simplicity of notation, we sometimes omit their dependence on parameters of $\pi:X\map T$.
\end{convention}

%%%%%%%%% New stuff %%%%%%%%%%%%%%%%%%%
\subsection{Flow spaces}\label{subsec-flow space}
Flow spaces are some special subsets of a tree of hyperbolic spaces defined
in \cite{ps-kap}. Here, we shall give a slightly different but equivalent definition
for the ease and better clarity of exposition.
\subsubsection{Flowing across edge spaces}
In this subsection we shall describe what it means to flow
a subset of a vertex space to an adjacent vertex space, and a few other
related notions. In what follows $k\geq 0$ will denote an arbitrary constant.

\begin{enumerate}
	\item Suppose $u,v$ are two vertices of $T$ connected by an edge $e$.
	If $A\subset X_{eu}$ then we define the {\bf transport} of $A$ into $X_v$
	to be the set $Tr_{uv}(A)=f_{e,v}(f^{-1}_{e,u}(A))$.

	\item
	Suppose $A\subset X_u$ is a $k$-quasiconvex subset of $X_u$ and $R\geq 0$.
	We say that $A$ is $R$-{\bf flowable} into $X_v$ if
	$P^{X_u}_{X_{eu}}(A)\subset N^{X_u}_R(A)$.

	In this case, we define the $R$-{\bf flow} of $A$
	into $X_v$ to be the set $$\F l_{uv,R}(A)=Hull_{X_v}(Tr_{uv}(P^{X_u}_{X_{eu}}(A))).$$

	We note that the set $\F l_{uv,R}(A)$ does not dependent on $R$ as long as
	$A$ is $R$-flowable, and in that case, we sometimes write $\F l_{uv}(A)$ instead $\F l_{uv,R}(A)$. %\BLUE{When we write $\F l_{uv,R}(A)$, it always means that $A$ is $R$-flowable.}
	\item
	Suppose $A_u\subset X_u$ and $A_v\subset X_v$ are $k$-quasiconvex subsets of $X_u$ and $X_v$ respectively.
	We will say that they are $R$-{\bf flow stable} for some $R\geq 0$ if
	$\F l_{uv,R}(A_u)\subset N^{X_v}_R(A_v)$ and $\F l_{vu,R}(A_v)\subset N^{X_u}_R(A_u)$. %\BLUE{Notation issue, $\F l_{uv,R}$ or $\F l_{uv}$}

	\item
	Suppose $S\subset T$ is a subtree and for all $v\in V(S)$ we have a $k$-quasiconvex
	subset $A_v$ of $X_v$. Let $\mathcal A_S$ be the collection $\{A_v:v\in V(S)\}$.

	$(i)$ We will say that $\mathcal A_S$
	is $R$-flow stable for some $R\geq 0$ if for any two adjacent vertices $u,v\in V(S)$,
	the sets $A_u\subset X_u$ and $A_v\subset X_v$ are $R$-flow stable.

	$(ii)$ $\mathcal A_S$ will be called $D$-{\bf saturated} if for any pair of adjacent
	vertices $v\in V(S)$, $w\in V(T)\setminus V(S)$ connected by an edge $e$, say, we
	have $diam(P^{X_v}_{X_{ev}}(A_v))\leq D$ in $X_v$.
\end{enumerate}

We record some of the easy conclusions of the definitions as a lemma.

\begin{lemma}\label{flow1}
	Suppose $u,v\in V(T)$ are connected by an edge $e$ and $A\subset X_u$ is a $k$-quasiconvex
	subset. Then the following hold:
	\begin{enumerate}
		\item $P^{X_u}_{X_{eu}}(A)$ is $\bar{k}=K_{\ref{lem-proj on pair qc}\ref{lem-proj on pair qc:qc proj qc2}}(\delta_0,k+\lambda_0)$-quasiconvex in $X_u$.
		\item $Tr_{uv}(P^{X_u}_{X_{eu}}(A))$ is $k'$-quasiconvex in $X_v$ where
		$$k'=K_{\ref{qc morph}}(\delta_0, L_0, K'_{\ref{qc morph}}(\delta_0, L_0, \bar{k})).$$
		\item If $A$ is $R$-flowable into $X_v$ then
		$$Hd_{X_v}(Tr_{uv}(P^{X_u}_{X_{eu}}(A)), \F l_{uv,R}(A))\leq k'; \text{ see Lemma \ref{hull of qc}}.$$
		\item If $A_1\subset X_{eu}$ is $k_1$-quasiconvex then $Hull_{X_u}(A_1)$ is
		a $2\delta_0$-quasiconvex subset of $X_u$ which is $\lambda_0$-flowable
		into $X_v$.
		\item If $A$ is $R$-flowable into $X_v$ and $R\geq \lambda_0$ then it follows from $(4)$ above that $\F l_{uv,R}(A)$ is
		$R$-flowable into $X_u$.

	\end{enumerate}
\end{lemma}

\begin{lemma}\label{transport lemma}
	Suppose $u,v\in V(T)$ are connected by an edge $e$. Then the following hold:

	$(1)$ Suppose $A_1, A_2\subset X_u$ are any two subsets.
	If $Hd(A_1,A_2)\leq D$ then $$Hd_{X_v}(Tr_{uv}(P^{X_u}_{X_{eu}}(A_1)), Tr_{uv}(P^{X_u}_{X_{eu}}(A_2)))\leq D'$$ where $D'$
	depends only on $D$ and the parameters of the tree of spaces.

	$(2)$ Suppose $A\subset X_{eu}$. Then

	$(a)$ $Hd_{X_{uv}}(A, Tr_{uv}(A)) \leq1$ and

	$(b)$ $Hd_{X_u}(A,Tr_{vu}(Tr_{uv}(A)))\leq \eta_0(2)$

	where fibers are $\eta_0$-proper embedding as in Definition \ref{tree-of-sps}.
\end{lemma}
\proof $(1)$ Since $X_{eu}$ is $\lm_0$-quasiconvex in $X_u$, the projection map
$P^{X_u}_{X_{eu}}$ is $C_{\ref{proj-on-qc}\ref{proj-on-qc:1}}(\dl_0,\lm_0)$-coarsely Lipschitz (Lemma \ref{proj-on-qc}~\ref{proj-on-qc:1}). Hence,
we have $D_1=(D+1)C_{\ref{proj-on-qc}\ref{proj-on-qc:1}}(\dl_0,\lm_0)$
such that $Hd_{X_u}(P^{X_u}_{X_{eu}}(A_1), P^{X_u}_{X_{eu}}(A_2))\leq D_1$. Let $x\in P^{X_u}_{X_{eu}}(A_1)$ and $y\in P^{X_u}_{X_{eu}}(A_2)$ be such that $d_{X_u}(x,y)\le D_1$.
Since the maps $f_{e,v}:X_e\map X_v$ and $f_{e,u}:X_e\map X_u$
are $L_0$-qi embeddings, we have $Hd_{X_v}(f_{e,v}(f^{-1}_{e,u}(x)),f_{e,v}(f^{-1}_{e,u}(y)))\le L_0(D_1+L^2_0)+L_0$. Since $x$ and $y$ were arbitrary with their choices, we have $Hd_{X_v}(Tr_{uv}(P^{X_u}_{X_{eu}}(A_1)), Tr_{uv}(P^{X_u}_{X_{eu}}(A_2)))\leq D'$ where $D'=L_0(D_1+L^2_0)+L_0$.

$(2)$ Note that $(a)$ is clear from definition whereas $(b)$ follows from $(a)$.
\qed

\begin{lemma} \label{flow stable lemma1}
Let $k\ge0$ and $R\ge0$.	Suppose $A\subset X_{eu}$ is $k$-quasiconvex in $X_u$ then
	$Hd_{X_{uv}}(A,\F l_{uv,R}(A))\leq D_{\ref{flow stable lemma1}}(k)$.
\end{lemma}
\proof
Let $A_1=Tr_{uv}(A)$. Then $Hd_{X_{uv}}(A,A_1)\leq 1$ (Lemma \ref{transport lemma} $(2)$ $(a)$).
Now, by Lemma \ref{flow1} $(2)$, $A_1$ is $k'$-quasiconvex in $X_v$ where $k'$ as in Lemma \ref{flow1}.
Hence, $Hd_{X_v}(A_1, \F l_{uv,R}(A))\leq k'$ and so $Hd_{X_{uv}}(A_1,\F l_{uv,R}(A))\leq k'$.
Hence, $Hd_{X_{uv}}(A,\F l_{uv,R}(A))\leq Hd_{X_{uv}}(A,A_1)+Hd_{X_{uv}}(A_1,\F l_{uv,R}(A))\leq 1+k'$. We set $D_{\ref{flow stable lemma1}}(k)=1+k'$.
\qed

\begin{lemma}\label{flow stable lemma}
Suppose $[u,v]$ is an edge in $T$ and $A\subset X_u$ is a $k$-quasiconvex subset of $X_u$. Let $A$ be $R$-flowable into $X_v$ where $R\geq \lambda_0$
	and $B=\F l_{uv,R}(A)$. Then $A,B$ are $D_{\ref{flow stable lemma}}(k)$-flow stable.
\end{lemma}

\proof
Let $A_1=P^{X_u}_{X_{eu}}(A)$. Note that $B=\F l_{uv,R}(A)=Hull_{X_v}(Tr_{uv}(A_1))$ by definition. Thus we only will prove that $\F l_{vu,R}(B)$ lies in uniform neighbourhood of $A$ in $X_u$.

Since $A$ is $k$-quasiconvex and $X_{eu}$ is $\lambda_0$-quasiconvex in $X_u$,
$A_1$ is $K_{\ref{lem-proj on pair qc}\ref{lem-proj on pair qc:qc proj qc2}}(\delta_0,k+\lambda_0)$-quasiconvex in $X_u$.
Let $k_1=K_{\ref{lem-proj on pair qc}\ref{lem-proj on pair qc:qc proj qc2}}(\delta_0,k+\lambda_0)$. Then from Lemma \ref{flow stable lemma1},
we get $Hd_{X_{uv}}(A_1, B)\leq D_{\ref{flow stable lemma1}}(k_1)$.

Let $B_1=P^{X_v}_{X_{ev}}(B)$. 
Since $X_{ev}$ is $\lambda_0$-quasiconvex in $X_v$ we see that
$B\subset N^{X_v}_{\lambda_0}(X_{ev})$. It follows then that $B$ is $R$-flowable
into $X_u$ as $R\geq \lambda_0$ and also that
$Hd_{X_{v}}(B,B_1)\leq \lambda_0$,
whence, $Hd_{X_{uv}}(B,B_1) \leq \lambda_0$. On the other hand, since $B$ is
$2\delta_0$-quasiconvex in $X_v$, $B_1$ is $K_{\ref{lem-proj on pair qc}\ref{lem-proj on pair qc:qc proj qc2}}(\delta_0,2\delta_0+\lm_0)$-quasiconvex
in $X_v$. Let $k_2=K_{\ref{lem-proj on pair qc}\ref{lem-proj on pair qc:qc proj qc2}}(\delta_0,2\delta_0+\lm_0)$. Now by Lemma
\ref{flow stable lemma1}, we have
$Hd_{X_{uv}}(B_1, \F l_{vu,R}(B))\leq D_{\ref{flow stable lemma1}}(k_2)$.

Hence, $Hd_{X_{uv}}(B, \F l_{vu,R}(B))\leq Hd_{X_{uv}}(B_1, \F l_{vu,R}(B))+ Hd_{X_{uv}}(B_1,B)
\leq \lambda_0+D_{\ref{flow stable lemma1}}(k_2)$. Finally we have
$Hd_{X_{uv}}(A_1,\F l_{vu,R}(B))\leq Hd_{X_{uv}}(A_1,B)+Hd_{X_{uv}}(B,\F l_{vu,R}(B))
\leq D_{\ref{flow stable lemma1}}(k_1)+\lambda_0+D_{\ref{flow stable lemma1}}(k_2)=D'$, say.
Hence, $Hd_{X_u}(A_1, \F l_{vu,R}(B))\leq \eta_0(D')$ as $X_u$ is $\eta_0$-properly embedded in $X$.
Since, by hypothesis $A_1\subset N^{X_u}_R(A)$, we have
$\F l_{vu,R}(B)\subset N^{X_u}_{R+\eta_0(D')}(A)$. Thus we can take
$D_{\ref{flow stable lemma}}=R+\eta_0(D')$. \qed

%\begin{lemma}\label{flow stable lemma3}
%Suppose $A\sse X_u$ and $B\sse X_v$. If $A$ and $B$ are $R$-flow stable then $A$ is $R'$-flowable in $X_v$ and $B$ is $R'$-flowable in $X_u$.
%\end{lemma}
%
%\begin{proof}
%Suppose $A_1=P^{X_u}_{X_{eu}}(A)$ and $B_1=P^{X_v}_{X_{ev}}(B)$.
%\end{proof}

\begin{lemma}\label{flow stable lemma2}
Given non-negative numbers $k,~D,~D',~R$ we have constants $D_{\ref{flow stable lemma2}}=D_{\ref{flow stable lemma2}}(D)$, $R'_{\ref{flow stable lemma2}}=R'_{\ref{flow stable lemma2}}(D)$ and $D'_{\ref{flow stable lemma2}}=D'_{\ref{flow stable lemma2}}(k,D',R)$ satisfying the following. Suppose $A\sse X_u$ and $B\sse X_v$ are $k$-quasiconvex in $X_u$ and $X_v$ respectively.
	Let $A_1=P^{X_u}_{X_{eu}}(A)$ and $B_1=P^{X_v}_{X_{ev}}(B)$. Then:

	$(1)$ If $A$, $B$ are $D$-flow stable then $Hd_{X_{uv}}(A_1,B_1)\leq D_{\ref{flow stable lemma2}}$. Moreover, $A$ is $R'_{\ref{flow stable lemma2}}$-flowable in $X_v$ and $B$ is $R'_{\ref{flow stable lemma2}}$-flowable in $X_u$

	$(2)$ Conversely, if $A$ is $R$-flowable in $X_v$, $B$ is $R$-flowable in $X_u$ and $Hd_{X_{uv}}(A_1,B_1)\leq D'$ then $A$, $B$ are $D'_{\ref{flow stable lemma2}}$-flow
	stable
\end{lemma}

\begin{proof}
$(1)$ Suppose $x\in A_1$ and $x_1\in Tr_{uv}(A_1)$ such that $d_{X_{uv}}(x,x_1)\le 1$ (see Lemma \ref{transport lemma} $(2)$). Since $\F l_{uv,D}(A)=Hull_{X_v}(Tr_{uv}(A_1))\sse N^{X_v}_D(B)$, we have $x_2\in B$ such that $d_{X_v}(x_1,x_2)\le D$. Thus $d_{X_v}(x_1,B_1)\le d_{X_v}(x_1,x_2)+d_{X_v}(x_2,P^{X_v}_{X_{ev}}(x_2))\le 2D$. Hence by triangle inequality, $x\in N^{X_{uv}}_{1+2D}(B_1)$. Since $x$ was arbitrary, $A_1\sse N^{X_{uv}}_{1+2D}(B_1)$. Similarly, we have $B_1\sse N^{X_{uv}}_{1+2D}(A_1)$. Therefore, we can take $D_{\ref{flow stable lemma2}}(D)=1+2D$.

Let $x_3=P^{X_v}_{X_{ev}}(x_2)$ and $x_4\in Tr_{vu}(B_1)$ be such that $d_{X_{uv}}(x_3,x_4)\le 1$. Since $A$ and $B$ are $D$-flow stable, we have $x_5\in A$ such that $d_{X_u}(x_4,x_5)\le D$. By triangle inequality, $d_{X_{uv}}(x,x_5)\le 2+3D$. Since the fibers are $\eta_0$-proper embedding, $d_{X_u}(x,A)\le \eta_0(2+3D)$. Since $x$ was arbitrary, we have $A_1\sse N^{X_u}_{\eta_0(2+3D)}(A)$. Similarly, $B_1\sse N^{X_v}_{\eta_0(2+3D)}(B)$. Therefore, we can take $R'_{\ref{flow stable lemma2}}=\eta_0(2+3D)$.

$(2)$ By Lemma \ref{flow1} $(1)$, $A_1\sse X_{eu}$ is $\bar{k}$-quasiconvex in $X_u$ where $\bar{k}=K_{\ref{lem-proj on pair qc}\ref{lem-proj on pair qc:qc proj qc2}}(\delta_0,k+\lambda_0)$. Then by Lemma \ref{flow stable lemma1}, $Hd_{X_{uv}}(A_1,\F l_{uv,R}(A))=Hd_{X_{uv}}(A_1,\F l_{uv,R}(A_1))\le D_{\ref{flow stable lemma1}}(\bar{k})$. Hence by triangle inequality, $Hd_{X_{uv}}(\F l_{uv,R}(A),B_1)\le D'+D_{\ref{flow stable lemma1}}(\bar{k})$. Since $B$ is $R$-flowable in $X_u$, i.e., $B_1\sse N^{X_v}_R(B)$, so $\F l_{uv,R}(A)\sse N^{X_{uv}}_{D''}(B)$ where $D''=R+D'+D_{\ref{flow stable lemma1}}(\bar{k})$. Since fibers are $\eta_0$-proper embedding in $X$, $\F l_{uv,R}(A)\sse N^{X_v}_{\eta_0(D'')}(B)$. Similarly, $\F l_{vu,R}(B)\sse N^{X_u}_{\eta_0(D'')}(A)$. Therefore, we can take $D'_{\ref{flow stable lemma2}}(k,D')=\eta_0(D'')\ge R$.
\end{proof}

\subsubsection{\bf Definition of flow spaces}\label{subsub-flow space}

In this section we state the definition of flow of a subset of $X$ which will turn out to be
flow stable and saturated subset of $X$. We will also see below that these sets are uniformly
quasiconvex in $X$. Given $k\ge0$, fix two constants $\bm{ R_{fc}}(k)=2k+7\dl_0+R_{\ref{lem-cobdd}\ref{lem-cobdd:R-sep-D-cobdd}}(\dl_0,k+2\dl_0+\lm_0)$ and $\bm {D_{fc}}(k)=D_{\ref{lem-cobdd}\ref{lem-cobdd:R-sep-D-cobdd}}(\dl_0,k+2\dl_0+\lm_0)$; see Remark \ref{rmk-about flow constants} for the significance of these constants. These will refer as flow constants; here suffix `$fc$' stands for `flow constant'.

\begin{defn}\label{defn-flow space}{\em ({\bf Flow of fiber subspaces})} {\em
		(I) Suppose $u\in V(T)$ and $A\subset X_u$ is $k$-quasiconvex in $X_u$ for some $k\geq 0$. Assume that $R\geq \bm{R_{fc}}(k)$, $D \geq 0$. Then the $D$-saturated, $R$-flow space
		determined by $A$, or simply the $(R,D)$-flow of $A$, denoted by $\F l^X_{R,D}(A)$,
		is defined inductively as follows: $\F l^X_{R,D}(A)$ consists of a collection
		$\{A_v:v\in V(S)\}$ where
		\begin{itemize}
			\item $S$ is a subtree of $T$ containing $u$,
			\item $A_u=A$ and
			\item each $A_v$, $v\in V(S)$, $v\neq u$, is a $2\delta_0$-quasiconvex subset of $X_v$.
		\end{itemize}
		The induction is on distance from $u$ in $T$. Moreover, $S$ and the sets $A_v$'s are
		simultaneously constructed in the induction process.\smallskip

\noindent{\em Base of induction:} For each $v\in V(T)$ which is connected by an edge $e$ to $u$ one checks
		(1) if $A$ is $R$-flowable into $X_v$, and (2) if $diam(P^{X_u}_{X_{eu}}(A))\geq D$ in $X_u$; and if
		both the conditions are verified then we include the segment $[u,v]$
		in $S$ and we let $A_v=\F l_{uv,R}(A)$. Thus by the first step of induction
		we get a subtree of $T$ contained in the ball $B(u,1)$.\smallskip

\noindent{\em Induction step:} Suppose $v\in V(S)$ with $d(u,v)=n$. Then for each
		$w\in V(T)$ which is connected to $v$ by an edge $e'$, say, such that $d(u,w)=n+1$
		we check if (1) $A_v$ is $R$-flowable into $X_w$ and (2) $diam(P^{X_v}_{X_{e'v}}(A_v))\geq D$ in $X_v$.
		If both the conditions are verified then we include the segment $[v,w]$
		in $S$ and define $A_w=\F l_{vw,R}(A_v)$.

		\smallskip
		(II) ({\bf Flow of flow stable subspaces.}) More generally, suppose $S$ is a subtree of $T$ containing at least two vertices and $k\geq 0, D\geq 0$ are
		constants, and $\A=\{A_v: v\in V(S)\}$ is a $R$-flow stable family of subsets
		in $X$ where each $A_v\in \A$ is $k$-quasiconvex in $X_v$.
		Then we define the $(R,D)$-flow $\F l^X_{R,D}(\A)$ or $\F l_R(\A)$ of $\A$ where
		$R\geq \bm{R_{fc}}(k)$ as follows.

For each $v\in V(S)$, let $T'_v$ be the union of all the
		components $T\setminus \{v\}$ which are disjoint from $S$ and let $T_v=T'_v\cup \{v\}$.
		Let $Z_v=X_{T_v}$ be the subtree of spaces in $X$ defined over $T_v$.
		Let $\V$ be the set of vertices of $S$ such
		that $v\in \V$ implies $T'_v\neq \emptyset$.

		Then we define
		$$\F l^X_{R,D}(\A)=(\bigcup_{v\in V(S)} A_v)\bigcup (\bigcup_{v\in \V} \F l^{Z_v}_R(A_v)).$$}
\end{defn}

\begin{remark}\label{rmk-about flow constants}
Suppose we consider the $(R,D)$-flow space in Definition \ref{defn-flow space} (I)  with the constant $D=\bm{D_{fc}}(k)=D_{\ref{lem-cobdd}\ref{lem-cobdd:R-sep-D-cobdd}}(\dl_0,k+2\dl_0+\lm_0)$ and $R\geq\bm{R_{fc}}(k)=2k+7\dl_0+R_{\ref{lem-cobdd}\ref{lem-cobdd:R-sep-D-cobdd}}(\dl_0,k+2\dl_0+\lm_0)$. Here is the reason for choosing such constants. Recall that for an edge $e=[u,v]$ joining $u$ and $v$, $X_{eu}$ is $\lm_0$-quasiconvex in $X_u$ by Lemma \ref{com-two-hyp-sps}. In the definition of flow spaces, one notes that when we do induction, we take quasiconvex hull of a transport. We mention that quasiconvex hull of any subset in a $\dl_0$-hyperbolic space is $2\dl_0$-quasiconvex. So we take sum of all these quasiconvexity constants as $k+2\dl_0+\lm_0$ so that all the mentioned subsets are $(k+2\dl_0+\lm_0)$-quasiconvex in their respective ambient spaces. 

Suppose we have flowed $A_u$ in $X_v$, i.e., $A_v\ne\emptyset$. Note that $diam(P^{X_u}_{X_{eu}}(A_u))\ge \bm{D_{fc}}(k)$. Thus $d_{X_u}(A_u,X_{eu})\le R_{\ref{lem-cobdd}\ref{lem-cobdd:R-sep-D-cobdd}}(\dl_0,k+2\dl_0+\lm_0)$ by Lemma \ref{lem-cobdd}~\ref{lem-cobdd:R-sep-D-cobdd}. Hence by Lemma \ref{lem-cobdd}~\ref{lem-cobdd:close imp proj inclu}, $P^{X_u}_{X_{eu}}(A_u)\sse N^{X_u}_{\bm{R_{fc}}(k)}(A_u)$. In other words, if $P^{X_u}_{X_{eu}}(A_u)$ has large diameter in $X_u$ then $A_u$ is $\bm{R_{fc}}(k)$-flowable in $X_v$.

Flow space is being introduced to capture a uniform quasiconvex subset in $X$ containing $A_u$. This can be achieved by constructing a retraction of the entire space $X$ onto it provided $X$ is hyperbolic (see Theorem \ref{thm-mitra proj on stable-satura}). Thus the computation above shows that if the projection $P^{X_u}_{X_{eu}}(A_u)$ has small diameter in $X_u$, then the direction containing $e$ away from $u$ is not an effective direction for the flow space to be quasiconvex, so we do not flow in that direction.
\end{remark}

A direct consequence of the definition, we have the following.

\begin{lemma}\label{lem-lift in flow space}
Suppose $R\ge\bm{R_{fc}}(k)$ and $D\ge0$. Let $u\in V(T)$, and let $e$ be an edge in $T$ incident on $u$. Suppose $\F l^X_{R,D}(X_{eu})$ is the flow space of $X_{eu}$ in $X$. Let $x\in\F l^X_{R,D}(X_{eu})$ be any point such that $\pi(x)=w$. Then there is a $K_{\ref{lem-lift in flow space}}$-qi lift of $[u,w]$ whose image lies in $\F l^X_{R,D}(X_{eu})$ where $K_{\ref{lem-lift in flow space}}=R+k'+1$ and $k'$ as in Lemma \ref{flow1} $(2)$.
	
	Moreover, we may assume that if $\al$ is such a qi lift then $\al(v)\in X_{ev}$ where $e$ is the edge in $[u,v]\sse [u,w]$ incident on $v$.
\end{lemma}

\begin{prop}\label{consistent qc}
Suppose $k\ge0$, $R\ge\bm{R_{fc}}(k)$ and $D\ge0$. Suppose $S$ is a subtree of $T$ and $\A=\{A_v: v\in V(S)\}$ is a $R$-flow stable family of subsets
	in $X$ where each $A_v\in \A$ is $k$-quasiconvex in $X_v$.
	Then $(R,D)$-flow $\F l^X_{R,D}(\A)$ of $\A$ is a $R'$-flow stable
	collection of subsets in $X$ for some $R'$ depending on $R$. %\BLUE{this is also $D$-saturated, no need to talk about the last sentence of this proposition.}

	In particular the same conclusion holds for $\F l_R(A)$ where $A$ is
	a $k$-quasiconvex subset of $X_u$ for some $u\in V(T)$.

	Lastly, $\F l^X_{R,D}(\A)$ is $D'$-saturated where $D'=max\{D,\bm{D_{fc}}(k)\}$.
\end{prop}

\proof The saturation condition is immediate from the construction of the flow
space. The flow stability is a direct consequence of Lemma \ref{flow stable lemma}.
\qed\smallskip

%\RED{See what properties of flow spaces are used in the paper- coarse path conn, flow lines, qc of course etc } \BLUE{have mentioned in blue before lemmas}

\subsubsection{{\bf Lipschitz retraction on flow spaces}}\label{subsubsec-Lipschitz retraction}
The following theorem makes the flow spaces extremely useful. However, both the definition of the retraction
and the proof that it is coarsely Lipschitz are adaptation of the work of M. Mitra from \cite{mitra-trees}.

\smallskip
\noindent{\bf Definition of Mitra's retraction map.}

Suppose $S$ is a subtree of $T$, and for some $k\geq 0$, $A_v$ is a $k$-quasiconvex subset of
$X_v$ for all $v\in V(S)$. Suppose $\A=\{A_v:v\in V(S)\}$ is an $R$-flow stable and $D$-saturated
collection of subsets of $X$ for some $R\ge\bm{R_{fc}}(k)$, $D\geq 0$.
Let $A'=\bigcup_{v\in V(S)}A_v$ and $X'_S=\bigcup_{v\in V(S)}X_v$.
Let $\rho': X'_S\map A'$ be the map defined by setting $\rho'(x)=P^{X_v}_{A_v}(x)$ for all
$x\in X_v$, $v\in V(S)$. Now $\rho'$ is extended to $\rho':X'=\bigcup_{w(\in V(T)} X_w\map A'$ as follows.
Suppose $w\in V(T)\setminus V(S)$. Let $u$ be the nearest point projection of $w$ on $S$
and let $e$ be the edge incident on $u$ on the geodesic $[u,w]$.
Let $x_u\in P^{X_u}_{A_u}(X_{eu})$ be an arbitrary point chosen once and for all.
Then for any $x\in X_w$ we define $\rho'(x)=x_u$. With this definition of $\rho'$ we have the following
result. Thus for any connected component $T_1$ of $T\setminus S$, if $X'_1=\bigcup_{v\in V(T_1)} X_v$
then $\rho'(X'_1)$ is a point.

\begin{theorem}\label{thm-mitra proj on stable-satura}
There are constants $L_{\ref{thm-mitra proj on stable-satura}}=L_{\ref{thm-mitra proj on stable-satura}}(R,D,k)\geq 0$ such that the following hold:

The map $\rho'$ can be extended to a
$L_{\ref{thm-mitra proj on stable-satura}}$-coarsely Lipschitz retraction map $\rho:X\map A'$.
%where on $A'$ we take the metric restricted from $X$.

%=\rho'(X_{eu})$ and $diam(\rho'(X_{T_1}))\leq D'_{\ref{thm-mitra proj on stable-satura}}$.
\end{theorem}

\proof
%Since $X'$ is a $1$-dense subset of $X$, it is now enough, by Lemma \ref{proj-from-net},
%to check that for all $x,y\in X'$ with $d_X(x,y)\leq 1$, \RED{1 may not be enough: formulate. Note:
%Here any path is approximated by concatenation of vertical and horizontal path}
%$d_X(\rho(x),\rho(y))$ is uniformly small. However, by Lemma \ref{proj-on-qc} and the definition of $\rho$,
%this is easy if $x,y$ are in the same vertex space of $X$. Hence, it only remains to consider the
%case where $x\in X_v, y\in X_w$ where $v,w$ are connected by an edge $e$ and
%$x=f_{e,v}(z), y=f_{e,w}(z)$, $d_X(x,y)=1$. This is done in three different cases.
We first define an extension map $\rho:X\map A'$ of $\rho':X'\map A'$. Note that $\rho$ is already defined on $X'$. Now let $x\in X\setminus X'$ so that $\pi(x)\in[u,v]=e\in E(T)$. If $\pi(x)\ne m(e)$ (the middle point of $e$) and $\pi(x)\in[u,m(e)]$ (respectively, $\pi(x)\in[m(e),v]$), we take $f_{e,u}(x)=x'\in X_{eu}\sse X'$ (respectively, $f_{e,v}(x)=x'\in X_{ev}\sse X'$), and define $\rho(x):=\rho(x')$. If $\pi(x)=m(e)$, we choose either $f_{e,u}$ or $f_{e,v}$ to define $\rho(x)$ as above, and this choice is fixed once and for all.

Since $X$ is a geodesic metric space and $X=N_1(X')$, it is enough prove that for all $x,y\in X'$ with $d_X(x,y)\leq 1$,
$d_X(\rho(x),\rho(y))$ is uniformly small. However, by Lemma \ref{proj-on-qc}~\ref{proj-on-qc:1} and the definition of $\rho$,
this is easy if $x,y$ are in the same vertex space of $X$. Hence, it only remains to consider the
case where $x\in X_v, y\in X_w$ where $v,w$ are connected by an edge $e$ and $z\in X_e$,
$x=f_{e,v}(z),~ y=f_{e,w}(z)$, $d_X(x,y)=1$. This is done in three different cases.

{\bf Case 1.} Suppose $v\not\in S, w\not\in S$.  In this case, it follows from the
definition of $\rho$ that $\rho(x)=\rho(y)$. So there is nothing to prove.

{\bf Case 2.} Suppose $v\in S$ but $w\not\in S$. In this case,
$\rho(x), \rho(y)\in P^{X_v}_{A_v}(X_{ev})$. However, by the saturation condition on $\A$, we know that $diam(P^{X_v}_{ X_{ev}}(A_v))< D$. Hence, by Lemma \ref{lem-proj on pair qc}~\ref{lem-proj on pair qc:small-imp-small},
$diam(P^{X_v}_{A_v}(X_{ev}))\leq D'$ for some constant $D'$ depending only on $\dl_0$, $max\{k,\lm_0\}$ and $D$.

{\bf Case 3.} Suppose $v,w\in S$.
Let $B_v=P^{X_v}_{X_{ev}}(A_v)$ and $B_w=P^{X_w}_{X_{ew}}(A_w)$. Since $\mathcal A$ is a $R$-flow stable collection, we have $B_v\sse N_{R_1}(A_v)$ and $B_w\sse N_{R_1}(A_w)$ for some constant $R_1$ depending on $R$ and $\eta_0$ (see Lemma \ref{flow stable lemma2} $(1)$). Then by Lemma
\ref{lem-proj on pair qc}~\ref{lem-proj on pair qc:new proj lemma} we have
$d_{X_v}(\rho(x), P^{X_v}_{B_v}(x))\leq D_{\ref{lem-proj on pair qc}\ref{lem-proj on pair qc:new proj lemma}}(\dl_0,k+\lm_0,R_1)$.

On the other hand, $X_{vw}$ is hyperbolic and $X_v,X_w$ are uniformly qi embedded
in $X_{vw}$ by Lemma \ref{com-two-hyp-sps}. Thus by Lemma {\ref{lem-proj on pair qc}~\ref{lem-proj on pair qc:qc proj qc2}} $B_v$ is uniformly
quasiconvex in $X_{vw}$. We note that $P^{X_v}_{B_v}(x)$ is uniformly close to
$P^{X_{vw}}_{B_v}(x)$ by Lemma \ref{proj-in-diff-are-close}. Hence, $\rho(x)$ is uniformly close to $P^{X_{vw}}_{B_v}(x)$.
In the same way, $P^{X_{vw}}_{B_w}(y)$ is uniformly close to $\rho(y)$.

Next since $\mathcal A$ is a
$R$-flow stable collection we have by Lemma \ref{flow stable lemma2} $(1)$ that
$Hd_{X_{vw}}(B_v,B_w)\leq R_2$ for a constant depending on $R$. Hence, by Lemma \ref{lem-proj on pair qc}~\ref{lem-proj on pair qc:new proj lemma3}, $d_{X_{vw}}(P^{X_{vw}}_{B_v}(x), P^{X_{vw}}_{B_w}(x))$ is uniformly
small. Since by Lemma \ref{proj-on-qc}~\ref{proj-on-qc:1}, $P^{X_{vw}}_{B_w}$ is uniformly coarsely Lipschitz,
we know that $d_{X_{vw}}(P^{X_{vw}}_{B_w}(x), P^{X_{vw}}_{B_w}(y))$ is uniformly small.
Case 3 follows from these.\smallskip

This concludes Theorem \ref{thm-mitra proj on stable-satura}.\qed

\smallskip

Theorem  \ref{thm-mitra proj on stable-satura} along with Lemma \ref{proj-on-qc}~\ref{proj-on-qc:lem-retraction imp qc}
immediately give the following.
\begin{cor}\label{flow-qc}
Suppose we have the hypotheses (and notation) of Theorem \ref{thm-mitra proj on stable-satura}.
Moreover, suppose that $X$ is hyperbolic. Then we have the following:

$(1)$ $A'$ is uniformly quasiconvex in $X$.

$(2)$ $A'$ is uniformly quasiconvex in $X_S$.
\end{cor}

\begin{remark}\label{rmk-properties used}
	The main fundamental properties of flow spaces that we use in this paper are as follows. Let $R=\bm{ R_{fc}}(\lm_0)$ and $D=\bm{D_{fc}}(\lm_0)$.
	\begin{enumerate}
		\item Flow spaces respect order, which follows directly from the definition. Namely, if
		$[u,v]\cup[v,w]$ is a geodesic segment in $T$ where $u,v,w\in V(T)$, then for every vertex $w'$ in $[v,w]$, we have $\F l_{R,D}(X_u)\cap X_{w'} \subseteq \F l_{R,D}(X_v)\cap X_{w'}$.
		
		\item The flow spaces $\F l_{R,D}(X_u)$ are coarsely path connected. Moreover, given any point
		$x \in \F l_{R,D}(X_u)$, there exists a uniform qi lift through $x$ till $X_u$
		(Lemma \ref{lem-lift in flow space}).
		
		\item The flow spaces, as well as the flow-stable spaces with the saturated condition, admit a uniform coarse Lipschitz retraction of the ambient space (Theorem \ref{thm-mitra proj on stable-satura}). This, in turn, implies that these spaces are quasiconvex (Corollary \ref{flow-qc}).
	\end{enumerate}
\end{remark}

Note that $X_u$ is $0$-quasiconvex in $X_u$ for any vertex $u\in V(T)$. Keeping this in mind, we choose the constants accordingly in the following lemma.

%\BLUE{a structural property of flow space is used in the following lemma}

\begin{lemma}\label{lem-flows are cobdd}
Let $D_1\ge0$, and let $R=\bm{R_{fc}}(0)$ and $D=\bm{D_{fc}}(0)$. Suppose $u,v,w\in V(T)$, and $w$ is a vertex in $[u,v]\setminus\{u,v\}$. Assume that $X$ is hyperbolic. If $\F l^X_{R,D}(X_u)$ and $\F l^X_{R,D}(X_w)$ are $D_1$-cobounded then $\F l^X_{R,D}(X_u)$ and $\F l^X_{R,D}(X_v)$ are $D'_1$-cobounded for some uniform constant $D'_1\ge0$ depending on $D_1$ and the other structural constants. 
\end{lemma}

\begin{proof}
	Suppose $w'$ is the vertex in $[w,v]$ adjacent to $w$. Let $T_1$ be the connected component of $T\setminus\{w'\}$ containing $w$. Then it follows from the definition of flow spaces (see also Remark \ref{rmk-properties used} $(1)$) that $\pi(\F l^X_{R,D}(X_v))\cap T_1\sse\pi(\F l^X_{R,D}(X_w))$. Note that $X_w$ is a separator of $X_u$ and $X_{T_2}$ where $T_2$ is the connected component of $T\setminus\{w\}$ containing $w'$. Hence it follows that the diameter of nearest point projections of $\F l^X_{R,D}(X_v)$ onto $\F l^X_{R,D}(X_u)$ is bounded by $D_1$. Since $X$ is hyperbolic and flow spaces are uniformly quasiconvex in $X$ (Corollary \ref{flow-qc}), the result follows from Lemma \ref{lem-proj on pair qc}~\ref{lem-proj on pair qc:small-imp-small}.
\end{proof}

%%%%%%%% To be moved to section 1 %%%%%%%%%%%%

%%%%%%%%%%%%%%%%%%%%%%%%%%%%%%%%%%%%%%%%%%%%%%%%%%%%%%%%%

\begin{comment}
\RED{Move the following:} {\tiny
Next we discuss yet another instance of consistent and saturated family
of subsets of $X$, namely {\em ladders}. They will be used in the last
section of the papers we discuss CT laminations.

{\bf Ladders and their properties}
Given a vertex space $X_v$ and a geodesic segment $\alpha \subset X_v$
one defines

Ladder is a special type of semicontinuous subgraph whose fibers are geodesic segments in the respective fibers.

\begin{defn}[Ladder]\label{ladder}
	Write new defn
\end{defn}

\RED{introduce semiinfinite and biinfinite ladders here.}
}

\end{comment}

%%%%%%%%%%%%%%%%%%%%%%%%%%%%%%%%%%%%%%%%
%%%%%%%%%%%%%%%%%%%%%%%%%%%%%%%%%%%%%%%%
%%%%%%%%%%%%%%%%%%%%%%%%%%%%%%%%%%%%%%%%
\subsection{Boundary of X}\label{subsec-boundary of trees of sps}
In general it is difficult to give a satisfactory description of
the geodesic rays in $X$ even when $X$ is hyperbolic. However, the following
theorem is somewhat informative in this connection. {\em For rest of the section, we will assume that $X$ is hyperbolic}

%\BLUE{Theorem \ref{bdry of X} is a by product of Lemma \ref{star vertex}, and is used only for motivational purpose of Lemma \ref{not-hv-pro-sub}. However, Theorem \ref{bdry of X} is a nice result to be here.}

\begin{theorem}\label{bdry of X}
	Suppose $\xi\in \partial X$. Then there is a sequence $\{x_n\}$ in $X$ with $\LMX x_n=\xi$ with
	one of the following additional properties:\\
	$(1)$ $\{\pi(x_n)\}$ is a constant sequence or \\
	$(2)$ there is a geodesic ray $\alpha\subset T$ such that
	$\pi(x_n)\in \alpha$ for all $n\in\N$ and $\lim^T_{n\map \infty} \pi(x_n)=\alpha(\infty)$.
\end{theorem}
%We postpone the proof of the theorem to collect a couple of lemmas needed for the proof.

{\em Proof of Theorem \ref{bdry of X}:}
Let $\{x'_n\}$ be any sequence in $X$ such that $\LMX x'_n=\xi$. Moreover, we may assume that $\pi(x'_n)$ are vertices of $T$.
Let $T_1$ be the convex hull in $T$ of the set $\{\pi(x'_n):n\in \N\}$. There are
two cases to consider.

\noindent{\bf Case 1}: Suppose $T_1$ is a locally finite tree, i.e., all its vertices are of finite degree.
Note that if $T_1$ is bounded then $T_1$ consists of finitely many edges (and vertices), and thus $\{x'_n\}$ has a subsequence $\{x'_{n_k}\}$ such that $\{\pi(x'_{n_k})\}$ is a single vertex.

\noindent Otherwise, suppose that $T_1$ is a unbounded. Then,
by Lemma \ref{unbdd-pro-sp-ray}, there is a geodesic ray $\alpha: [0,\infty)\map T_1$.
Let $u=\alpha(0)$ and let $\{x'_{n_k}\}$ be a subsequence of $\{x'_n\}$
such that $\lim^T_{k\map \infty} \pi(x'_{n_k})=\alpha(\infty)$. Let $x\in X_u$ be any point.
Let $v_k$ be the nearest point projection of $\pi(x'_{n_k})$ on $\alpha$. We note that
$[x,x'_{n_k}]_X\cap X_{v_k}\neq \emptyset$ for all $k\in \N$. Let $x_k\in [x,x'_{n_k}]_X\cap X_{v_k}$
for all $k\in \N$. Then by Lemma \ref{seq-on-con-geo}, $\LMX x'_n=\lim^X_{k\map \infty} x'_{n_k}=\lim^X_{k\map \infty} x_k$.

\noindent{\bf Case 2}: Suppose $T_1$ has a vertex of infinite degree. Then we are done by Lemma \ref{star vertex}
below.
\qed

\begin{lemma}\label{star vertex}
	Suppose $\{x_n\}$ is a sequence in $X$ such that $\LMX x_n\in \partial X$.
	Suppose $T_1$ is the convex hull in $T$ of the set $\{\pi(x_n):n\in\N\}$ and that there
	is a vertex $u$ of $T_1$ which has infinite degree in $T_1$.
	Let $\{x_{n_k}\}$ be any subsequence of $\{x_n\}$ such that
	$[u,\pi(x_{n_k})]\cap [u,\pi(x_{n_l})]=\{u\}$ for all $k\neq l$.
	Let $e_k$ be the edge on $[u,\pi(x_{n_k})]$ incident on $u$.
	Then (1) for all $k\in \N$, there is $x'_k\in X_{e_k u}$ such that $\LMX x_n=\LMX x'_n$.

	(2) Suppose that the subsequence $\{x_{n_k}\}$ is chosen
	in such a way that the sets $X_{e_ku}$ converges
	to a point of $\partial X_u$, and suppose that
	$x''_k$ is an arbitrary point of $X_{e_ku}$ for all $k\in \N$. Then $\LMX x_n=\LMX x''_n$.
\end{lemma}

%{\tiny\RED{Move this into the proof wherever is appropriate:} We note that in moreover part of Lemma \ref{star vertex}, $\{X_{e_ku}: k\in \N\}$ is an infinite, locally finite collection (see Lemma \ref{lem-edge sps loc finite}) of uniformly quasiconvex subsets of $X_u$. Hence, by Lemma \ref{seq-set-con} we can always extract a subsequence of $\{X_{e_ku}\}$ satisfying the condition of moreover part of Lemma \ref{star vertex}.}

\proof
Note that for all $k\in \N$, $[x,x_{n_k}]_X\cap X_{e_k u}\neq \emptyset$. Let $x'_k$
be any point of $[x,x_{n_k}]_X\cap X_{e_k u}$.
Since $\{X_{e_ku}\}$ is a locally finite collection of subsets in $X_u$ by Lemma \ref{lem-edge sps loc finite},
$d_{X_u}(x, X_{e_ku})\map \infty$ as $k\map \infty$. It follows that $d_{X_u}(x,x'_k)\map \infty$
as $k\map \infty$, and so $d_X(x,x'_k)\map\infty$ as $k\map\infty$. Then it follows from Lemma \ref{seq-on-con-geo} that $\LMX x'_n=\LMX x_n$.

Moreover if $X_{e_ku}\map \eta\in \partial X_u$ as $k\map \infty$ then for any choices of $x''_k\in X_{e_ku}$, $k\in \N$,
we have $\lim^{X_u}_{n\map \infty} x'_n =\lim^{X_u}_{n\map \infty} x''_n$ (see Lemma \ref{qc conv criteria} $(1)$). Since the inclusion $X_u\map X$
admits a CT map (\cite{mitra-trees}) we have $\LMX x''_n= \LMX x'_n=\LMX x_n$.
\qed\smallskip 

The following is a simple instance of Theorem \ref{bdry of X}, which can be easily verified.
\begin{lemma}\label{lem-not geo ray imp seq in vertex}
Suppose $\al$ is a geodesic ray in $X$. Assume that $\pi(\al)$ does not contain any geodesic ray. Then there is an unbounded sequence $\{r_n\}\sse\R_{\ge0}$ such that $\al(r_n)\in X_u$ and $\lim^{X_u}_{n\map\infty}\al(r_n)$ exists in $\pa X_u$. %and $\LMX x_n=\al(\infty)$.	
\end{lemma}

In the rest of the subsection we prove a few other related results which come to use in the later
part of the paper.

%\BLUE{qc of flow is used in the following lemma}

\begin{lemma}\label{reduction0}
	Suppose $\{x_n\}$ is an unbounded sequence in $X$ such that $\LMX x_n$ exists.
	Let $u_n=\pi(x_n)$ for all $n\in \N$, and let $d_T(u,u_n)\ri\infty$ as $n\ri\infty$.
	Let $e_n$ be the edge on $[u,u_n]$ incident on $u_n$ for all $n\in \N$ and
	let $x\pr_n$ be a nearest point projection of $x_n$ on $X_{e_nu_n}$ in $X_{u_n}$.
	Then $\LMX x'=\LMX x_n$.
\end{lemma}

\begin{proof}
%Let $x\pr_n$ be a nearest point projection of $x_n$ on $X_{e_nu_n}$ in the path metric $X_{u_n}$. 
Note that by Lemma {\ref{lem-proj on pair qc}~\ref{lem-proj on pair qc:qc proj qc1}}, diam$\{P^{X_{u_n}}_{X_{e_nu_n}}([x_n,x'_n]_{X_{u_n}})\}$ is uniformly bounded in $X_{u_n}$; suppose it is bounded by $D_1$ for some uniform constant $D_1\ge0$. Let $T_n$ be the subtree containing $u_n$ not containing interior of $e_n$. Note that $[x_n,x'_n]_{X_{u_n}}$ is $\dl_0$-quasiconvex in $X_{u_n}$. Let $D=max\{\bm{D_{fc}}(\dl_0),D_1\}$ and $R=\bm{R_{fc}}(\dl_0)$. Let $W_n=\pi^{-1}(T_n)$, and let $A_n=\F l^{W_n}_{R,D}([x_n,x'_n]_{X_{u_n}})$ be the $(R,D)$-flow space of $[x_n,x'_n]_{X_{u_n}}$ in $W_n$. Observe that $A_n$ is also $(R,D)$-flow space of $[x_n,x'_n]_{X_{u_n}}$ in $X$ containing $x_n$ and $x'_n$, and $d_X(x,A_n)\ge d_T(u,u_n)$ where $x\in X_u$ is fixed. On the other hand, $A_n$ is $K$-quasiconvex in $X$ for some uniform constant $K\ge0$ by Corollary \ref{flow-qc} $(1)$. Therefore, by the stability of quasigeodesic (Lemma \ref{ml}), there is a uniform constant $D_2\ge0$ such that $d_X(x,[x_n,x'_n]_X)\ge d_X(x,A_n)-D_2\ge d_T(u,u_n)-D_2$. Hence $d_X(x,[x_n,x'_n]_X)\map\infty$ as $n\map\infty$. Thus by Lemma \ref{con-same-pt}, we have $\LMX x'=\LMX x_n$.
%Let $T_n$  be the connected component of $T\setminus\{u_n\}$ containing $u$. We consider a ladder $\L^n_K([x_n,x\pr_n]^f)$ with parameters $K=K_{\ref{ladder}}(R),\ep=\ep_{\ref{ladder}}(R)$ and $C=max\{D,D_{\ref{R-sep-D-cobdd}}(\dl\pr_0,\lm\pr_0)\}$ such that $\L^n_K([x_n,x\pr_n]^f)\cap X_{T_n}=\emp$, where  $R=R_{\ref{R-sep-D-cobdd}}(\dl\pr_0,\lm\pr_0)$ and $\pi^{-1}(T_n)=X_{T_n}$. Fix a point $x\in X_u$. Since $d_T(u,u_n)\ri\infty$ as $n\ri\infty$ then $d_X(x,\L^n_K([x_n,x\pr_n]^f))\ri\infty$ as $n\ri\infty$. Hence $d_X(x,[x_n,x\pr_n]_X)\ri\infty$ as $n\ri\infty$. Therefore, $[\{x\pr_n\}]=[\{x_n\}]$ is $\pa_s X$.
\end{proof}

Given $\xi\in \partial X$, by Theorem \ref{bdry of X} there is a sequence $\{x_n\}$
such that either $\{\pi(x_n)\}$ is constant or $\lim^T_{n\map \infty} \pi(x_n)\in \partial T$.
However, these two possibilities are not mutually exclusive, i.e.,
we may have two different sequences $\{x_n\}$ and $\{x'_n\}$ such that
$\LMX x_n=\LMX x'_n=\xi$ where
$\{\pi(x_n)\}$ is constant but $\{\pi(x'_n)\}$ converges to a point of $\partial T$.
The following lemma records the implication of such an instance.

%\BLUE{qc of flow is used in the following lemma}

\begin{lemma}\label{not-hv-pro-sub}
Let $R=\bm{ R_{fc}}(\lm_0)\ge\bm{R_{fc}}(0)$ and $D=\bm{D_{fc}}(\lm_0)\ge\bm{D_{fc}}(0)$. Suppose $\{x_n\}$, $\{x'_n\}$ are two unbounded sequences in $X$ such that
	$\LMX x_n=\LMX x'_n\in \partial X$, and
	$\lim^T_{n\map \infty} \pi(x_n)=\xi\in \partial T$.
	Suppose that the nearest point projection of
	each $\pi(x'_n)$ on the geodesic ray $[u,\xi)$ is $u$ for some $u\in V(T)$.
	
	Then
$\F l^X_{R,D}(X_u)$ and $\F l^X_{R,D}(X_v)$ are not cobounded for any
	vertex $v\in [u,\xi)\setminus\{u\}$. Moreover, there is a uniform constant $D'\ge0$ such that $N_{D'}(\F l^X_{R,D}(X_u))\bigcap N_{D'}(\F l^X_{R,D}(X_v))\ne\emptyset$.
In particular, $$N_{D''}(\F l^X_{R,D}(X_{e_1u}))\bigcap N_{D''}(\F l^X_{R,D}(X_{e_2v}))\ne\emptyset$$ for some uniform constant $D''\ge0$ where $e_1$ and $e_2$ are the edges in $[u,\xi)$ incident on $u$ and $v$ respectively.
	
\end{lemma}

\proof %We note that (2) is an immediate from (1) and Corollary \ref{cobdd flows}.

Suppose that $X$ is $\dl'$-hyperbolic for some $\dl'\ge0$, and the $(R,D)$-flow spaces are $K$-quasiconvex in $X$ for some uniform constant $K\ge0$ (see Corollary \ref{flow-qc} $(1)$). Let $D_1=D_{\ref{lem-cobdd}\ref{lem-cobdd:R-sep-D-cobdd}}(\dl',K)$.
Suppose $v$ is a vertex of $[u,\xi)$ such that $\F l^X_{R,D}(X_u))$ and $\F l^X_{R,D}(X_v)$ are
$D_1$-cobounded. Let $v_n$ be the nearest point projection of $\pi(x_n)$ on $[u,\xi)$.
Since $\lim^T_{n\map \infty} \pi(x_n)=\xi$, there is $N\in \N$ such that for
$v_n\in [v,\xi)\setminus\{v\}$ for all $n\geq N$, $n\in \N$. Therefore, any geodesic segment
$[x'_n,x_n]_X$ has a subsegment, say $\alpha_n$, joining a point in $X_u$ to a
point in $X_v$. By Lemma \ref{lem-flows are cobdd}, $\F l^X_{R,D}(X_u)$, $\F l^X_{R,D}(X_{v_n})$ are $D'_1$-cobounded in $X$ for all $n\ge N$ and for some uniform constant $D'_1\ge0$ depending on $D_1$, $\dl'$ and $K$. Hence by Lemma \ref{lem-cobdd}~\ref{lem-cobdd:qc proj new},
there is a point of $\F l^X_{R,D}(X_u)$ which is $D_2$-close to $\alpha_n$ for all $n\geq N$ where the constant $D_2$ depends on $\dl'$, $K$ and $D'_1$.
In particular, for any $x\in X$, the distance $d_X(x, [x'_n,x_n]_X)$ is bounded.
This is a contradiction to Lemma \ref{con-same-pt} as $\LMX x_n=\LMX x'_n\in \partial X$.

For the {\em moreover} part, take $D'=R_{\ref{lem-cobdd}\ref{lem-cobdd:R-sep-D-cobdd}}(\dl',K)$.

The {\em last part} follows directly from the definition of flow spaces together with the {\em moreover} part (Remark \ref{rmk-properties used}).\qed

\subsubsection{\bf Vertical and horizontal sequences}

There are the following two types of sequences in $X$ converging to $\pa X$ that come
up repeatedly in our discussion.
\begin{defn}{\em ({\bf Vertical and horizontal sequences})}
Suppose $\{x_n\}$ is a sequence in $X$ such that $\LMX x_n \in \pa X$.
We shall refer to $\{x_n\}$ to be a {\em vertical sequence} if $\{\pi(x_n)\}$ is
a constant sequence in $T$. We shall call it a {\em horizontal sequence}
if $\pi(x_n)$'s lie on a geodesic ray of $T$ and
$\lim^T_{n \map \infty} \pi(x_n)\in \pa T$.
\end{defn}
A priori it is unclear if all points of $\pa X$ are limits of some
vertical or horizontal sequences. However, the following two lemmas proves
this point.

\begin{lemma}{\em ({\bf Vertical replacement lemma})}\label{vertical replacement}
Suppose $\{x_n\}$ is any sequence of points in $X$ such that $\LMX x_n$ exists.
Then there is a vertex $u$ of $T$ and a sequence $\{x'_n\}$ in $X_u$ such that
$\LMX x_n =\LMX x'_n$ provided one of the following conditions hold:

$(1)$ $\{\pi(x_n)\}$ is bounded.

$(2)$ There is a vertex of infinite degree in $T_1=hull\{\pi(x_n):n\in\N\}\sse T$.
\end{lemma}
\proof
We note that in case $\{\pi(x_n)\}$ is bounded, either there is a constant subsequence of
$\{\pi(x_n)\}$ or $T_1$ has a vertex of infinite degree. Hence we may
divide the proof into the following two cases:\smallskip

\noindent{\bf Case 1.} Suppose there is a constant subsequence $\{\pi(x_{n_k})\}$ of $\{\pi(x_n)\}$.
Let $u=\pi(x_{n_k})$ for all $k\in \N$. Then $x_{n_k}\in X_u$ for all $k\in \N$.
Let $x'_k=x_{n_k}$, $k\in \N$. Hence $\LMX x'_n=\LMX x_n$.\smallskip

\noindent{\bf Case 2.} Suppose $T_1$ has a vertex $u$ of infinite degree.
Suppose $\{x_{n_k}\}$ is a subsequence of $\{x_n\}$ such that
$[u,\pi(x_{n_k})]\bigcap [u,\pi(x_{n_l})]=\{u\}$ for all $k\neq l$. For all $k\in \N$, let $e_k$
be the edge on $[u,\pi(x_{n_k})]$ which is incident on $u$.  Fix $x\in X_u$, and fix $x'_k\in[x,x_{n_k}]_X\bigcap X_{e_ku}$. Since $X_u$ is proper embedding in $X$ and $\{X_{e_ku}:k\in\N\}$ is locally finite (Lemma \ref{lem-edge sps loc finite}), $d_X(x,x'_k)\to\infty$ as $k\to\infty$. Hence by Lemma \ref{seq-on-con-geo}, $\lim^X_{k\to\infty}x'_k=\lim^X_{k\to\infty}x_{n_k}=\LMX x_n$. This completes the proof.\qed\smallskip

In this case, we shall refer to $\{x_n\}$ as a sequence {\bf replaceable by
	a vertical sequence} and we shall call $\{x'_n\}$ to be a {\bf vertical replacement
	sequence of $\{x_n\}$ in $X$}.
\begin{comment}
\proof
We note that in case $\{\pi(y_n)\}$ is bounded, either there is a constant subsequence of
$\{\pi(x_n)\}$ or $T_1$ has a vertex of infinite degree. Hence we may
divide the proof into the following two cases:\smallskip

\noindent{\bf Case 1.} Suppose there is a constant subsequence $\{\pi(y_{n_k})\}$ of $\{\pi(y_n)\}$.
Let $u=\pi(y_{n_k})$ for all $k\in \N$. Then $y_{n_k}\in Y_u$ for all $k\in \N$.
Let $y'_k=y_{n_k}$, $k\in \N$. We note that $\LMX y'_n=\LMX y_n$ and $\LMY y'_n=\LMY y_n$.\smallskip

\noindent{\bf Case 2.} Suppose $T_1$ has a vertex $u$ of infinite degree.
Suppose $\{y_{n_k}\}$ is a subsequence of $\{y_n\}$ such that
$[u,\pi(y_{n_k})]\cap [u,\pi(y_{n_l})]=\{u\}$ for all $k\neq l$. For all $k\in \N$, let $e_k$
be the edge on $[u,\pi(y_{n_k})]$ which is incident on $u$. By passing
to a further subsequence we may assume that the sequence of quasiconvex
sets $Y_{e_ku}$ converges to a point of $\partial Y_u$ and the
sequence of quasiconvex sets $\{X_{e_k u}\}$ converges to a point of
$\partial X_u$ as $k\map \infty$ (Proposition \ref{seq-set-con} and Lemma \ref{lem-edge sps loc finite}). Therefore, if we take
$y'_k\in Y_{e_k u}\subset X_{e_k u}$ for all $k\in \N$ then, by the last part of Lemma
\ref{star vertex}, we get $\LMX y'_n=\LMX y_n$ and $\LMY y'_n=\LMY y_n$.
\qed\smallskip
\end{comment}

\begin{lemma}{\em ({\bf Horizontal replacement lemma})}\label{horiz replace}
Suppose $u$ is a vertex of $T$ and $x\in X_u$. Suppose $\{x_n\}$ is an unbounded
	sequence in $X$ such that $\lim^X_{n\map \infty}x_n\in \pa X$.
	Let $u_n=\pi(x_n)$ and suppose that $\lim^T_{n\map \infty} u_n=\xi \in \pa T$.
	Let $c_n$ be the nearest point projection of $u_n$ on $[u,\xi)$ and
	let $e_n$ be the edge on $[u,c_n]$ incident on $c_n$. Suppose
	$x'_n\in X_{e_nc_n}\bigcap [x, x_n]_Y$. Then
	$\lim^X_{n\map \infty}x_n=\lim^X_{n\map \infty}x'_n\in\pa X$.
\end{lemma}
\proof
We note that $d_X(x, x'_n)\geq d_T(u,c_n)=d_T(u, [u_n, \xi))$. Since $\lim^T_{n\to\infty}u_n= \xi$
we have $d_T(u, [u_n, \xi))\map \infty$, whence $d_X(x,x'_n)\map \infty$ as $n\map \infty$. Then by Lemma \ref{seq-on-con-geo}, $\LMY x'_n=\LMY x_n\in\pa X$ as $x'_n\in [x,x_n]_X$.
\qed\smallskip

In this case, we shall refer to $\{x_n\}$ as a sequence {\bf replaceable by
a horizontal sequence} and we shall call $\{x'_n\}$ to be a {\bf horizontal replacement
sequence of $\{x_n\}$ in $X$}. We note that Lemma \ref{vertical replacement} and Lemma
\ref{horiz replace} imply the following.

\begin{lemma}\label{vert or horiz}
Suppose $\{x_n\}$ is an unbounded sequence in $X$ such that $\LMX x_n \in \pa X$.
Then $\{x_n\}$ is either vertically replaceable or horizontally replaceable.
\end{lemma}

\subsection{Boundary flow}
As flow of quasiconvex subsets of vertex spaces were defined earlier, {\em boundary flow}
of points in the boundary of vertex spaces are defined in this section following \cite{ps-conical}.
\begin{defn}\textup{\cite[Definition $4.3$]{ps-conical}}
	(1) Suppose $u,v\in V(T)$ are connected by an edge $e$. If $\xi_u\in \partial X_u$ is in
	the image of $\partial f_{e,u}:\partial X_{e}\map\pa X_u$, then
	$\partial f_{e,v}( (\partial f_{e,u})^{-1}(\xi_u))=\xi_v$, say, is called the boundary
	flow of $\xi_u$ in $X_v$ (or more precisely, in $\partial X_v$).

	(2) Suppose $u,v\in V(T)$ are any two vertices and $u_0=u,u_1, \cdots, u_n=v$ are the
	consecutive vertices on $[u,v]$. Suppose $\xi_0\in \partial X_u$ and
	$\xi_n\in \partial X_v$. We say that $\xi_n$ is the boundary flow of $\xi_0$ if there
	are $\xi_i\in \partial X_{u_i}$, $1\leq i\leq n-1$ such that $\xi_i$ is the boundary flow
	of $\xi_{i-1}$ for all $1\leq i\leq n$.
\end{defn}

In this case we say
that $\xi_0$ can be flowed to $X_v$. Clearly boundary flow of a point of $\partial X_u$
to $\partial X_v$ is unique if it exists.

The following remark clarifies the relationship between the two types of flows: the one defined earlier and the one introduced in this section.
\begin{remark}\label{rmk-relation between flows}
Let $R=\bm{R_{fc}}(\dl_0)$ and $D=\bm{D_{fc}}(\dl_0)$. (Keep in mind that geodesics in vertex spaces $X_u$ are $\dl_0$-quasiconvex in $X_u$.) Suppose $\gamma$ is a geodesic in a vertex space $X_u$ and $v\in V(T)$ is any other vertex
such that $\F l^X_{R,D}(\gamma)\cap X_v \neq \emptyset$. Then $\F l^X_{R,D}(\gamma)\cap X_v $ is
quasiisometric to a segment of $\mathbb R$ where the parameters of the quasiisometry
depends on the parameters of the tree of spaces and $d_T(u,v)$.

In particular, if $\gamma$ is a geodesic ray then for all vertex $w$ in $[u,v]$, $\F l^X_{R,D}(\gamma)\cap X_w $
is qi to a segment of $[0,\infty)$. It is not difficult to check that $\gamma(\infty)$
can be boundary flowed to $X_v$ if and only if $\F l^X_{R,D}(\gamma)\cap X_v $ is qi to a geodesic ray
and in that case the limit point of $\F l^X_{R,D}(\gamma)\cap X_v $ is the boundary flow of
$\gamma(\infty)$.
\end{remark}

\begin{lemma}\label{bdry flow lemma}
	Suppose $u,v\in V(T)$ are any two vertices. Suppose $\xi_u\in \partial X_u$ and
	$\xi_v\in \partial X_v$. Suppose $\alpha_u\subset X_u$ and $\alpha_v\subset X_v$
	are geodesic rays in these vertex spaces such that $\alpha_u(\infty)=\xi_u$ and
	$\alpha_v(\infty)=\xi_v$. Then the following hold:

	$(1)$ $\xi_u$ can be boundary flowed to $X_v$ iff there is a constant $D>0$ such that
	$\alpha_u\subset N_D(X_v)$ in $X$.

	$(2)$ $\xi_v$ is the boundary flow of $\xi_u$ iff $\xi_u$ is the boundary flow of $\xi_v$.

	$(3)$ $\xi_v$ is the boundary flow of $\xi_u$
	iff $Hd_X(\alpha_u, \alpha_v)<\infty$.
\end{lemma}
\proof Clearly both (1), (2) follows from (3). For the proof of (3) see
\cite[Lemmas $4.4, 4.6$]{ps-conical} (or apply the remark above).
\qed

%The opposite implication is proved in the same way. We briefly sketch a proof for thesake of completion. \BLUE{It is enough to prove the statement when $u$ and $v$ are adjacent. Let $e=[u,v]\in E(T)$. First of all, note that $\xi_u\in\pa f_{e,u}(\pa X_e)$, otherwise, $Hd_X(\al_u,\al_v)$ will not be finite. Let $\bt$ be a geodesic ray in $X_v$ such that $\bt(\infty)=\pa f_{e,v}\circ(\pa f_{e,u})^{-1}(\xi_u)$. Note that $Hd_{X_{uv}}(\bt,\al_u)<\infty$ and so $Hd_{X}(\bt,\al_u)<\infty$. Thus (by triangle inequality) $Hd_{X}(\bt,\al_{v})<\infty$ and so $Hd_{X_v}(\bt,\al_{v})<\infty$ as $X_v$ is $\eta_0$-proper embedding in $X$. This shows that $\xi_v$ is the boudnary flow of $\xi_u$.}

The following result is the main property of boundary flow that will be useful
in the proof of our main theorem.
\begin{prop}\textup{(\cite[Proposition $2.3$]{ps-conical-c})}\label{bdry flow prop}
	Suppose $\xi_u\in \partial X_u$ and $\xi_v\in \partial X_v$ are mapped to the same
	point of $\partial X$ under the CT maps $\partial i_{X_u,X}:\partial X_u\map \partial X$ and
	$\partial i_{X_v,X}:\partial X_v\map \partial X$. Then there is a vertex $w$ in $[u,v]$
	such that both $\xi_u,\xi_v$ can be boundary flowed to $X_w$.
\end{prop}

We end this section with the following result which will be used in Section \ref{sec-CT lamination}.

\begin{lemma}\label{lem-equ-im-bdry-flow}
	Suppose $\bt$ is a geodesic ray in $X_u$ and $\al$ is geodesic ray in $X$ such that $\pi(\al)$ contains a geodesic ray, $[u,\xi)$ say, where $\xi\in\pa T$ and $u\in V(T)$. Further, we assume that $\pa i_{X_u,X}(\bt(\infty))=\al(\infty)$. Then $\bt$ has boundary flow over a geodesic ray in $\pi(\al)$.
	
	In particular, $[u,\xi)\sse\pi(\F l^X_{R,D}(X_u))$ where $R=\bm{ R_{fc}}(0)$ and $D=\bm{D_{fc}}(0)$.
\end{lemma}

\begin{proof}
	%	Note that by Lemma \ref{n-coni-im-ray} (2), $\pi(\al)$ contains a unique geodesic ray, say $[u,\xi)$, for some $\xi\in\pa T$. 
	For the sake of contradiction, let $v,w\in[u,\xi)$ be vertices such that $d_T(v,w)=1$ and $\bt(\infty)$ flows in $X_{v}$ but does not flow in $X_{w}$. In the following paragraph, we will find a $k$-quasiconvex subset, say $Z$, in $X$ containing $\bt$ such that $\pi(Z)\cap[u,\xi)=[u,v]$. Then we will be done as follows. Since $Z$ is quasiconvex in $X$ and $\pa i_{X_u,X}(\bt(\infty))=\al(\infty)$, we have $\al(\infty)\in\Lambda_X(Z)$. Thus $\al\sse N_{D_0}(Z)$ for some $D_0\ge0$. On the other hand, $[u,\xi)\sse\pi(\al)$. Thus we can choose large enough $m$ so that $d_X(Z,\al(m))>d_T(v,\al(m))>D_0$. This gives a contradiction.

	\underline{{\bf \emph{Finding $Z$}}}: Let $R=\bm{ R_{fc}}(0)$ and $D=\bm{D_{fc}}(0)$. Let $e=[v,w]$ and $W=\pi^{-1}([u,v])$. As mentioned in Remark \ref{rmk-relation between flows}, there is $K\ge0$ depending on the structural constants of $\pi:X\map T$ and $d_T(u,v)$ such that $\F l^X_{R,D}(\bt)\cap X_s$ is $K$-quasigeodesic ray in $X_s$ for all vertex $s\in[u,v]$. Moreover, $P^{X_v}_{X_{ev}}(\F l^X_{R,D}(\bt)\cap X_v)$ is of bounded diameter in $X_v$; suppose it is bounded by $D'\ge0$. 
	
Let $T'$ be the subtree $T\setminus\{w\}$ containing $v$. Suppose $Z=\F l^X_{R,D}(\bt)\cap\pi^{-1}(T')$. Let $D_1=max\{D,D'\}$. Then $Z$ is $R$-flow stable and $D_1$-saturated collection of subsets of $X$. Therefore, by Corollary \ref{flow-qc} $(2)$, $Z$ is $K$-quasiconvex in $X$ for some constant $K\ge0$ depending on $R$, $D_1$, and the other structural constants.
	
	This completes the proof.
\end{proof}

\section{Subtrees of hyperbolic spaces}\label{sec-subtree of spaces}
In this section we introduce the main objects of the paper, namely
subspaces of a tree of hyperbolic spaces with additional properties, e.g.
where the subspaces have an `induced' tree of hyperbolic space structure.
Here is the more precise definition.

\begin{defn}{\em ({\bf Subtrees of hyperbolic subspaces})}\label{defn subtree}
Suppose $\pi:X\map T$ is a tree of hyperbolic metric spaces
as defined in Definition \ref{tree-of-sps} and $Y\subset X$. Suppose $Y$ is a geodesic metric space with the induced length metric from
$X$, and that $Y$ is properly embedded in $X$. Let $S=\pi(Y)$ and let $\pi_1$ be the
restriction of $\pi$ to $Y$. We will say that $Y$ is a subtree of hyperbolic subspaces of $X$ if
$\pi_1=\pi|_S: Y\map S$ is a tree of hyperbolic metric spaces with qi embedded condition.
(In particular, in this case, for all $v\in V(S)$, $Y_v=X_v\cap Y$ and  for all $e\in E(S)$,
$Y_e=X_e\cap Y$  are the vertex spaces and edge spaces of $Y$ and the maps $f_{e,v}:Y_e\map Y_v$,
wherever applies, are the restrictions of the corresponding maps $f_{e,v}: X_e\map X_v$.)

If, in addition, there is a CT map for all the inclusions of vertex spaces and edge spaces
$Y_v\map X_v$, $v\in V(S)$ and $Y_e\map X_e$, $e\in E(S)$ respectively then we will say that $Y$ is a
{\bf subtree of hyperbolic subspaces with fiberwise CT maps}. Similarly if these inclusion maps are
all $k$-qi embeddings for some fixed $k\geq 1$ then we will say that $Y$ is a
{\bf subtree of hyperbolic subspaces with fiberwise qi embedding condition.}
\end{defn}

For the rest of this section and the next section we fix the following notation and convention
for the spaces $X$ and $Y$.
\begin{convention}\label{con-induced spaces}
	\begin{enumerate}
	\item $\pi:X\map T$ is a tree of hyperbolic metric spaces with parameters $\eta_0, \delta_0, L_0,\lm_0$
as defined in Definition \ref{tree-of-sps} (see Convention \ref{con-trees of metric spaces} $(5)$ for the constants and function $\eta_0$).

	\item $Y\subset X$ is a subtree of hyperbolic subspaces with fiberwise CT maps, $S=\pi(Y)$
	and $\pi_1=\pi|_S$.
	\item The tree of spaces $\pi_1:Y\map S$ has the same parameters
		$\eta_0, \delta_0, L_0$.

		\item The inclusions $Y_e\map X_e$, $e\in E(S)$, are $L$-qi embeddings for a constant $L\geq 1$.
		\item Both $X$ and $Y$ are proper hyperbolic geodesic metric spaces. We also assume that $Y\to X$ is $\eta_0$-proper embedding.
	\end{enumerate}	
\end{convention}

%We shall refer to $\pi_S:Y\map S$ where $Y\subset X$ as above as an {\bf {\em induced (sub)tree of subspaces}
	%satisfying property $\HC$.}
% \subsubsection{Induced subtrees of spaces with qi embedded condition}

We recall that the main result in this paper is about the existence of CT maps for the
inclusion $Y\map X$ where $Y$ is a subtree of subspaces of a tree of spaces $X$ where
these satisfy Convention \ref{con-induced spaces}, plus additional conditions to be introduced
in the next subsection. In this connection, in this section, we prove several results addressing
the following issue:  {\em Given two sequences $\{y_n\}$, $\{y'_n\}$ in $Y$ converging to the same
point of $\partial Y$, when do they converge to the same point of $\partial X$?}

\subsection{Superficial test sequence lemmas}
%\RED{May be find a better title of this subsection.}
We recall that if $\{y_n\}$, $\{y'_n\}$ are two sequences in $Y$ such that
$\LMY y_n =\LMY y'_n\in \pa Y$ and $\LMX y_n, \LMX y'_n \in \pa X$ which
may not be equal then we call $\{y_n\}$, $\{y'_n\}$ to be a {\em test sequence}
for the pair $(Y,X)$. If, moreover, $\LMX y_n= \LMX y'_n$ then we shall say that
$\{y_n\}$, $\{y'_n\}$ form a pair of {\bf superficial test sequence}.

\begin{lemma}\label{CT-from-a-ver}{\em ({\bf Vertical test sequences are superficial})}
Let $u$ be a vertex of $S$.	Suppose $\{y_n\},\{y\pr_n\}$ are two test sequences for the pair $(Y,X)$
	where $y_n, y'_n \in Y_u$ for all $n\in \N$. Moreover, assume that $\lim^{Y_u}_{n\map \infty} y_n,~\lim^{Y_u}_{n\map \infty} y'_n\in \partial Y_u.$
	Then $\lim^X_{n\map \infty} y_n=\lim^X_{n\map \infty} y'_n\in \partial X$;
	in other words, $\{y_n\},\{y\pr_n\}$ form a pair of superficial test sequences.
\end{lemma}

\proof Since the CT maps for the inclusions $Y_u\map X_u$ and $X_u\map X$ exist (\cite{mitra-trees}),
by the functoriality of CT maps (Lemma \ref{functo-ct-map}), we see that
$\lim^X_{n\map \infty} y_n$ and $\lim^X_{n\map \infty} y'_n$ exist;
and they are equal if $\lim^{Y_u}_{n\map \infty} y_n=\lim^{Y_u}_{n\map \infty} y'_n$.
Suppose $\lim^{Y_u}_{n\map \infty} y_n\neq \lim^{Y_u}_{n\map \infty} y'_n$.
Let $\alpha$ be a geodesic line in $Y_u$ such that
$\alpha(-\infty)=\lim^{Y_u}_{n\map \infty} y_n$ and 
$\alpha(\infty)=\lim^{Y_u}_{n\map \infty} y'_n$.

We note that $\alpha(-\infty)=\lim^{Y_u}_{n\map \infty} \alpha(-n)=\lim^{Y_u}_{n\map \infty} y_n$.
Thus as in the first paragraph of the proof we have
$\lim^X_{n\map \infty} \alpha(-n)=\lim^X_{n\map \infty} y_n$. Similarly, we have
$\lim^X_{n\map \infty} \alpha(n)=\lim^X_{n\map \infty} y'_n$. Thus we are reduced
to showing that $\LMX \alpha(n)=\LMX \alpha(-n)$.

Now we will use two results from Section \ref{sec-CT lamination} that describes uniform quasigeodesic in this case. Since $\pa i_{Y_u,Y}(\al(-\infty))=\pa i_{Y_u,Y}(\al(\infty))$, we will have description of uniform quasigeodesics joining $\al(-m)$ and $\al(n)$ (by Theorem \ref{thm-quasigeo-descrip}). However, the same description holds for $X$, which will conclude the result (by Proposition \ref{prop-con of geo des}).

We denote the threshold constants appearing in Theorem \ref{thm-quasigeo-descrip} by $M^Y$ and $M^X$ respectively to distinguish between $Y$ and $X$. For $k\ge1$, we have constants $M^Y$ and $M^X$ depending on $k$ as in Theorem \ref{thm-quasigeo-descrip}.

Note that the inclusions $Y_v\map Y\map X$ are $\eta_0$-proper embeddings. Thus $Y_v\map X_v$ is $\eta$-properly embedded where $\eta(r)=\eta_0\circ\eta_0(r)$ for all $r\ge0$. Let $\eta':\R_{\ge0}\to\R_{\ge0}$ be a proper map corresponding to $\eta$ as constructed in Lemma \ref{rev-of-pro-map}. Fix $R\ge M^Y$ such that for any $x,y\in Y_v$ with $d_{Y_v}(x,y)\ge R$ implies $d_{X_v}(x,y)\ge\eta'(R)\ge M^X$ (see Corollary \ref{cor-inverse of proper map}). Since $\LMY \alpha(n)=\LMY \alpha(-n)$, by Theorem \ref{thm-quasigeo-descrip}, we may assume the description of uniform quasigeodesics in $Y$ joining $\al(-m)$ and $\al(n)$ as follows $$\gm_{-m}*[\gm_{-m}(w),\gm_n(w)]_{X_w}*\gm_n$$ where $w$ is a vertex in $S$, and $\gm_{-m}$, $\gm_n$ are $k$-qi lifts of $[u,w]$ such that $$d_{Y_v}(\gm_{-m}(v),\gm_n(v))\ge R$$ for all vertex  $v$ in $[u,w]\setminus\{w\}$ and $$d_{Y_w}(\gm_{-m}(w),\gm_n(w))\le C$$ for some fixed $C\ge0$.

Observe that $\gm_{-m}$ and $\gm_n$ are $k$-qi lifts of $[u,w]$ in $X$ as well. Moreover, $$d_{X_v}(\gm_{-m}(v),\gm_n(v))\ge \eta'(R)\ge M^X\text{ and }d_{X_w}(\gm_{-m}(w),\gm_n(w))\le C.$$ Therefore, by Proposition \ref{prop-con of geo des}, we have $\LMX \alpha(n)=\LMX \alpha(-n)$.
\qed\smallskip

%A generalization of Lemma \ref{CT-from-a-ver} is the following.

%%%%%%%%%%%%%%%%%%%%%%%%%%%%%%%%
%%%%%%%%%%%%%%%%%%%%%%%%%%%%%%%%%%
%\RED{new lemma to be moved to earlier section, title may be changed later:}

%%%%%%%%%%%%%%%%%%%%%%%%%%%
%%%%%%%%%%%%%%%%%%%%%%%%%%%%

\begin{lemma}\label{bounded proj 2}{\em ({\bf Superficial vertical replacement lemma})}
	Suppose $\{y_n\}$ is an unbounded sequence in $Y$ such that both $\LMX y_n$, and 
	$\LMY y_n$ exist. Suppose $\{y_n\}$ is vertically replaceable in $Y$ (see Lemma
	\ref{vertical replacement}). Then there is vertical replacement sequence $\{y'_n\}$
	of $\{y_n\}$ in $Y$ such that 	$\LMX y_n=\LMX y'_n\in\pa X$, i.e. the sequences
	$\{y_n\},\{y\pr_n\}$ form a pair of superficial test sequences.
\end{lemma}
\proof 
With the help of Lemma \ref{CT-from-a-ver}, the proof of this lemma is a carbon copy of that of Lemma \ref{vertical replacement}.
The only notable point is that in Case $2$ of the proof one may, if required, pass to
a further subsequence to make sure that the quasiconvex sets $\{Y_{e_k u}\}$ converges to a point of
$\partial Y_u$ also.\qed

%\RED{ see if Lemma \ref{bounded projection lemma} is ever used.}

%%%%%%%%%%%%%%%

%%%%%%%%%%%%%%%%%%%%%%%%%%%%%%%%%%5

\begin{prop}\label{finite-proj-case}{\em ({\bf Vertically replaceable implies superficial})}
	Suppose $\{y_n\},\{z_n\}$ is a pair of test  sequences for the pair $(Y,X)$.
	If both of them are vertically replaceable then they are superficial.
\end{prop}

\proof Consider the sequence $\{y_n\}$. By Lemma \ref{bounded proj 2}
we can find a vertex $u\in V(T)$ and a sequence $\{y'_n\}$ in $Y_u$ such that
$\LMY y_n=\LMY y'_n$ and $\LMX y_n=\LMX y'_n$. Similarly, we can find a
vertex $v\in V(T)$ and a sequence $\{z'_n\}$ in $Y_v$ such that $\LMY z_n=\LMY z'_n$
and $\LMX z_n=\LMX z'_n$.
Passing to further subsequences, if necessary, we may assume that 
$\lim^{Y_u}_{n\map \infty}y'_n$ and $\lim^{Y_v}_{n\map \infty}z'_n$ exist.

Now, it is enough to show that $\LMX y'_n=\LMX z'_n$.
Let $\al_u$ and $\al_v$ be two geodesic rays in $Y_u$ and $Y_v$ respectively such that 
$\lim^{Y_u}_{n\map \infty} y'_n=\al_u(\infty)$  and 
$\lim^{Y_v}_{n\map \infty} z'_n=\al_v(\infty)$.
Since the inclusion map $Y_u\map Y$ admits a CT map (\cite{mitra-trees}) we have
$\LMY y'_n= \LMY \al_u(n)$ and $\LMY z\pr_n=\LMY \al_v(n)$. 
Similarly, since the inclusion maps $Y_u\map X_u$ and $X_u\map X$ admit CT maps we have  
$\LMX y'_n= \LMX \al_u(n)$ and $\LMX z\pr_n=\LMX \al_v(n)$. 
Hence, we are reduced to showing that $\LMX \alpha_u(n)=\LMX \alpha_v(n)$.
Note that $\LMY \alpha_u(n)=\LMY \alpha_v(n)$ as $\LMY y'_n=\LMY z'_n$.

However,
since $\pa i_{Y_u,Y}(\al_u(\infty))=\pa i_{Y_v,Y}(\al_v(\infty))$, by Proposition 
\ref{bdry flow prop}
there is a vertex $w$ in $[u,v]$ such that both $\al_u(\infty)$ and $\al_v(\infty)$ have
boundary flow to 
$\pa Y_w$. Let $\bt$ and $\bt\pr$ be geodesic rays in $Y_w$ such that the boundary flows of 
$\al_u(\infty)$ and $\al_v(\infty)$ in $\pa Y_w$ are respectively $\bt(\infty)$ and $\bt\pr(\infty)$.
Then by Lemma \ref{bdry flow lemma} we have $Hd_Y(\al_u,\bt)<\infty$, $Hd_Y(\al_v,\bt\pr)<\infty$.
This implies $Hd_X(\al_u,\bt)<\infty$, and $Hd_X(\al_v,\bt\pr)<\infty$ respectively since 
the inclusion $Y\map X$ is $1$-Lipschitz. However, $Hd_Y(\al_u,\bt)<\infty$ implies 
$\LMY \al_u(n)=\LMY \beta(n)$. Similarly, we have $\LMY \al_v(n)=\LMY \beta'(n)$,
$\LMX \alpha_u(n)=\LMX \beta(n)$ and $\LMX \alpha_v(n)=\LMX \beta'(n)$.
Thus it is enough to show that $\LMX \beta(n)=\LMX \beta'(n)$. Note that
$\LMY \beta(n)=\LMY \beta'(n)$ as we had $\LMY \alpha_u(n)=\LMY \alpha_v(n)$.
Therefore, we can apply Lemma \ref{CT-from-a-ver} to the sequences $\{\beta(n)\}$
and $\{\beta'(n)\}$ to finish the proof. \qed

%\begin{remark}
%	For the proof of Lemma \ref{CT-from-a-ver} and Proposition \ref{finite-proj-case}
%	we never needed the projection hypothesis.
%\end{remark}

%%%%%%%%%%%%%%%%%%%%%%%%%%%%%%%%%%%%%%%%%%%To be placed properly%%%%
\subsection{Compatible fiberwise projection condition}
The last piece of definition to state our main theorem is the following.

\begin{defn}\label{defn-com fib proj con} {\em  ({\bf Compatible fiberwise projection condition})}
 Suppose $Y$ is a subtree of hyperbolic spaces of a tree of hyperbolic spaces $\pi:X\map T$
 as in Convention \ref{con-induced spaces}. Then we say that the pair $(Y,X)$ satisfies the
compatible fiberwise projection condition with constant $R_0$ if the following holds.

For all $v\in V(S)$ and
	$e\in E(S)$ where $e$ is incident on $v$, and for all $x\in Y_v$ we have
	$$d_{X_v}(P^{X_v}_{X_{ev}}(x), P^{Y_v}_{Y_{ev}}(x))\leq R_0.$$
\end{defn}

Next we mention a few examples where the compatible fiberwise projection condition holds.

\begin{example}
Suppose $G$ is a hyperbolic group and $H$, $K$ are subgroups of $G$. Assume that either

\begin{enumerate}
	\item the inclusion $H\to G$ admits a CT map, $K$ is quasiconvex in $G$ and $K\sse H$, or
	
	\item both $H$ and $K$ are quasiconvex in $G$ (see Example \ref{exp-coarse sep}).
\end{enumerate}
Consider the doubles $G'=G*_KG$ and $H'=H*_KH$ in case $(1)$ or $G'=G*_{K}G$ and $H'=H*_{H\cap K}H$ in case $(2)$.

Note that in either case, $H'<G'$ (see Proposition \ref{prop-inter property}). Now the subgroup $H'$ corresponds to a subtree of subspaces $\pi|_Y:Y\to S$ in the tree of spaces $\pi:X\to T$ corresponding to $G'$ (see Section \ref{sec-application}). Then the fact that the pair $(Y,X)$ satisfies the compatible fiberwise projection condition follows from Lemma \ref{lem-uniform Mitra imp proj con} (below) in case $(1)$ and from Example \ref{exp-compatible projection} in case $(2)$.
\end{example}

%Following lemmas give instances where the compatible fiberwise projection condition holds.
\begin{lemma}\label{lem-uniform Mitra imp proj con}
Suppose that for all $e\in E(S)$, the qi embeddings $Y_e\map X_e$ are all uniform quasiisometries. Further assume that there is a proper map $\psi:\R_{\ge0}\map\R_{\geq 0}$ such that for all $u\in V(S)$, the inclusion $Y_u\map X_u$ satisfies uniform Mitra's criterion with the function $\psi$ (see Definition \ref{mitra-irr}). Then the compatible fiberwise projection condition holds by Lemma \ref{uni-mitra-proj}.
\end{lemma}

Another consequence of the projection condition is the following result which is not as straightforward
as the previous lemma and hence we write a proof for it.
\begin{lemma}\label{induced subtree consistent}
	Suppose $\pi:X\map T$ is a tree of hyperbolic spaces and $\pi_1:Y\map S$ is a
	subtree of spaces satisfying the hypotheses of Convention \ref{con-induced spaces}.
	Suppose $Y$ is fiberwise qi embedded too. Then the following are equivalent.

\begin{enumerate}
\item The pair $(Y,X)$ satisfies the compatible fiberwise
		projection condition.
		
\item For any $C\geq 0$ there is $D\geq 0$ such that the following holds:

If $u,v\in V(S)$ are connected by an edge $e\in E(S)$ then
$N^{X_u}_C(Y_u)\cap N^{X_u}_C(X_{eu})\subset N^{X_u}_D(Y_{eu}))$.
In other words, $Y_{eu}$ is a (uniform) coarse intersection of $X_{eu}$ and $Y_u$
in $X_u$. (See Definition \ref{coarse-inter}.)

\item The collection $\{Y_v: v\in V(S)\}$ is an $R$-flow stable family of subsets in $X$ for some $R\ge0$.

\item Any Mitra's retraction map $\{Y_v:v\in v(S)\}\to X_S$ induces a coarsely Lipschitz map $Y\to X_S$ (see Subsection \ref{subsubsec-Lipschitz retraction}).
\end{enumerate}
\end{lemma}
\proof %\RED{Proof seems too long and involved. Try to simplify:}
Note that $(1)\Longleftrightarrow(2)$ follows from $(2)\Longleftrightarrow(3)$ of Lemma \ref{lem-proj iff coarse intersec}.

%We will prove that $(1)\Longrightarrow(3)\Longrightarrow(2)$, $(4)\Longrightarrow(1)$ and $(3)\Longrightarrow(4)$.

We will prove that $(4)\Longrightarrow(1)\Longrightarrow(3)\Longrightarrow(4)$.

Clearly, $\{Y_v:v\in V(S)\}$ is a saturated collection of subsets in $X_S$. Hence $(3)\Longrightarrow(4)$ follows from Theorem \ref{thm-mitra proj on stable-satura}.\smallskip

\noindent$(4)\Longrightarrow(1)$: Suppose $e=[u,v]$ is an edge in $S$. Let $y\in Y_u$, and let $x'=P^{X_u}_{X_{eu}}(y)$ and $y'=P^{Y_u}_{Y_{eu}}(y)$. We will show that $d_{X_u}(x',y')$ is uniformly bounded.

Note that $[y,x']_{X_u}\cup[x',y']_{X_u}$ is a uniformly quasigeodesic in $X_u$ (see Lemma \ref{lem-proj on pair qc}~\ref{lem-proj on pair qc:concatenation qg}). Since $Y_u\to X_u$ is uniformly qi embedding, by the stability of quasigeodesic (Lemma \ref{ml}) there is $y_1\in[y,y']_{Y_u}$ such that $d_{X_u}(x',y_1)\le D_1$ for some uniform constant $D_1\ge0$. Let $y_2=P^{X_u}_{Y_u}(x')$. Then $d_{X_u}(x',y_2)\le D_1$.

Let $x_1\in X_{ev}$ be such that $d_X(x',x_1)=1$. Let $y_3=P^{X_v}_{Y_v}(x_1)$. Then by $(4)$, $d_X(y_2,y_3)\le C$ for some uniform constant $C\ge0$. 

Now $d_{X_{uv}}(y_2,y_3)\le\eta_{\ref{pro-emb}}(C)$ where $X_{uv}=\pi^{-1}([u,v])$. Since $Y\to X$ is $\eta_0$-properly embedded, there is a proper function $\eta:\R_{\ge0}\map\R_{\ge0}$ depending on $\eta_0$ such that $Y_{uv}$ is $\eta$-properly embedded in $X_{uv}$. Thus $d_{Y_{uv}}(y_2,y_3)\le\eta\circ\eta_{\ref{pro-emb}}(C)=D_2$ (say). Let $y_4\in Y_{eu}\cap[y_2,y_3]_{Y_{uv}}$ be such 
that $d_{Y_u}(y_2,y_4)\le D_2$. Hence by triangle inequality, $d_{X_u}(y_1,y_4)\le d_{X_u}(y_1,x')+d_{X_u}(x',y_2)+d_{X_u}(y_2,y_4)\le 2D_1+D_2$. Since $P^{Y_u}_{Y_{eu}}(y)=y'$, $y_1\in[y,y']_{Y_u}$ and $y_4\in Y_{eu}$, we have $d_{Y_u}(y_1,y')\le 2D_1+D_2$.

Finally, by triangle inequality, we have $d_{X_u}(x',y')\le d_{X_u}(x',y_1)+d_{X_u}(y_1,y')\le 3D_1+D_2$. This completes the proof.\smallskip 

\noindent$(1)\Longrightarrow(3)$: We will apply the characterization of flow stable as in Lemma \ref{flow stable lemma2} $(2)$. Let $R_0\ge0$ is the constant coming from the compatible fiberwise projection condition (see Definition \ref{defn-com fib proj con}). Let $e=[u,v]$ be an edge in $S$ joining vertices $u,v\in V(S)$. Let $A_1=P^{X_u}_{X_{eu}}(Y_u)$ and $B_1=P^{X_v}_{X_{ev}}(Y_v)$. Then the compatible fiberwise projection condition says that $Hd_{X_u}(A_1,Y_{eu})\le R_0$ and $Hd_{X_v}(B_1,Y_{ev})\le R_0$. Since $Hd_{Y_{uv}}(Y_{eu},Y_{ev})=1$, by triangle inequality, we have $Hd_{X_{uv}}(A_1,B_1)\le2R_0+1$. (Note that the inclusion $Y_{uv}\map X_{uv}$ is $1$-Lipschitz.) On the other hand, $A_1\sse N^{X_u}_{R_0}(Y_{eu})\sse N^{X_u}_{R_0}(Y_u)$, and $B_1\sse N^{X_v}_{R_0}(Y_v)$. That is, $Y_u$ is $R_0$-flowable in $Y_v$ and vice versa. Finally, assume that $Y_u$ and $Y_v$ are $k$-quasiconvex in $X_u$ and $X_v$ respectively for some uniform constant $k\ge1$ (see Lemma \ref{qc morph} $(2)$). Hence by the characterization of flow stable (Lemma \ref{flow stable lemma2} $(2)$), the collection $\{Y_v: v\in V(S)\}$ is $R$-flow stable family of subsets in $X$ where $R=D'_{\ref{flow stable lemma2}}(k,2R_0+1,R_0)$.\qed%\smallskip

%\noindent$(3)\Longrightarrow(2)$: We assume that $\{Y_u:u\in V(S)\}$ is a flow stable family of subsets of $X$. Let $[u,v]=e$ be an edge in $S$. Since $Y$ is $\eta_0$-properly embedded in $X$, there is a proper function $\eta:\R_{\ge0}\map\R_{\ge0}$ depending on $\eta_0$ such that $Y_{uv}$ is $\eta$-properly embedded in $X_{uv}$. It is enough to show that $Y_{eu}$ is a uniform coarse intersection of $X_{eu}$ and $Y_u$ in $X_u$.

%Let $x\in N^{X_u}_C(Y_u)\cap N^{X_u}_C(X_{eu})$ for $C\ge0$. We want to show that $d_{X_u}(x,Y_{eu})$ is bounded by a constant depending on $C$. Let $x_1\in Y_u$ and $x_2\in X_{eu}$ be such that $d_{X_u}(x,x_1)\le C$ and $d_{X_u}(x,x_2)\le C$. Let $x'=P^{X_u}_{X_{eu}}(x_1)$. Then $d_{X_u}(x_1,x')\le 2C$ as $d_{X_u}(x_1,x_2)\le 2C$ and $x_2\in X_{eu}$. Suppose $y'=f^{-1}_{e,v}\circ f_{e,u}(x')$. By the flow stable condition, we have $\F l^X_{uv}(Y_u)\sse N^{X_v}_R(Y_v)$. Let $y\in Y_v$ be such that $d_{X_v}(y',y)\le R$. Thus by triangle inequality, $d_{X_{uv}}(x_1,y)\le 2C+1+R$. Hence $d_{Y_{uv}}(x_1,y)\le\eta(2C+1+R)$. Note that any geodesic path in $Y_{uv}$ joining points in $Y_u$ and $Y_v$ passes through $Y_{eu}$. Therefore, $d_{Y_{uv}}(x_1,y'')\le \eta(2C+1+R)$ for some $y''\in Y_{eu}$. Hence $d_{Y_u}(x_1,y'')\le \eta_0\circ\eta(2C+1+R)=D$ (say). Thus by triangle inequality, $d_{Y_u}(x,y'')\le d_{X_u}(x,x_1)+d_{X_u}(x_1,y'')\le C+D$. This completes the proof.

\subsection{ Consequences of the projection condition}

\begin{convention}\label{subsubsec-induced trees}
	For the rest of this section, and all of the next section, we shall assume that the pair
of trees of spaces $Y\subset X$ satisfy both Convention \ref{con-induced spaces} and the
compatible fiberwise projection condition with constant $R_0$.
\end{convention}

The following lemma is a straightforward consequence of the definitions.

\begin{lemma}\label{rmk-coboundedness are same}
	Suppose $Y_{eu}$ and $Y_{e'u}$ are two distinct cobounded edge spaces in $Y_u$. Then $Y_{eu}$ and $Y_{e'u}$ are cobounded in $X_u$.
\end{lemma}

By Theorem \ref{flow-qc} we know that flow of a quasiconvex subset of a vertex space in $Y$ is quasiconvex
in $Y$. However, the following lemma says that under the compatible fiberwise projection condition
these flow spaces are quasiconvex in $X_S$ too. This is one of main technical tools needed for the results
in the later part of the paper.
\smallskip
\begin{lemma}\label{Y flow qc in X}{\em ({\bf Quasiconvexity of $\bm Y$-flows in $\bm{X_S}$})}
	Suppose the pair $(Y,X)$ of trees of hyperbolic spaces satisfies
the properties in Convention \ref{subsubsec-induced trees}. Then the following holds:

	Let $u\in V(S)$ and $e\in E(S)$ is an edge incident on $u$. Suppose
	$A\subset Y_{eu}$ is $k$-quasiconvex in $Y_u$ for some $k\geq 0$. Let $R\geq\bm{ R_{fc}}(k)$ and $D=\bm{D_{fc}}(k)$.
	Then $\F l^Y_{R,D}(A)$ is a $R$-flow stable and $D'$-saturated collection of
	subsets in $X_S$ where $D'\ge0$ is a constant depending on $D$ and $R_0$. Consequently, $\F l^Y_{R,D}(A)$
	is quasiconvex in $X_S$.
\end{lemma}
\proof Let $v\in \pi(\F l^Y_{R,D}(A))$. We first observe that $Y_v\cap \F l^Y_{R,D}(A)$
is a uniformly quasiconvex subset of $X_v$. This is clear for $v=u$ since the inclusions
$Y_{e}\to X_e\to X_v$ are uniformly qi embeddings. In general, since the edge spaces of $Y$
are uniformly qi embedded in the corresponding vertex spaces of $X$, and
$Y_v\cap \F l^Y_{R,D}(A)$ is the quasiconvex hull in $Y_v$, of a uniformly
quasiconvex subset of an edge space of $Y_v$ (see Lemma \ref{flow1}(2)(3)), it follows
that $Y_v\cap \F l^Y_{R,D}(A)$ is uniformly quasiconvex in $X_v$.
Next by Proposition \ref{consistent qc} $\F l^Y_{R,D}(A)$ is uniformly flow stable
in $Y$. One may then use the characterization of flow stability given in Lemma
\ref{flow stable lemma2} together with the compatible fiberwise projection condition to finish the proof.
\qed

\smallskip

Further applications of the compatible fiberwise projection condition will follow in the
next subsection.

\subsubsection{\bf A generalized flow space construction}

Suppose $u,v\in V(S)$. We construct a flow stable subset $Q_{uv}$ of $X_S$ containing
the edge spaces $Y_e$, $e\in E([u,v])$ through the following three steps. By construction
{\em most part} of $Q_{uv}$ is contained in $Y$, namely $(i)$ for all vertex $w\in \pi(Q_{uv})\setminus[u,v]$,
$Q_{uv}\cap X_w\subset Y_w$ and $(ii)$ for all vertex $w$ in $[u,v]$, $Q_{uv}\cap X_w$ is uniform Hausdorff close to a
quasiconvex hull in $X_w$ of the edge spaces $Y_e$ and $Y_{e'}$ where
$e, e'$ are the edges on $[u,v]$ incident on $w$. Here are the details:

\smallskip
{\bf Step 1. Construction of a flow stable subset of $\bm{X_{[u,v]}:=\pi^{-1}[u,v]}$.}
Suppose $u=u_1,u_2,\cdots, u_n=v$ is the sequence of vertices on the geodesic
$[u,v] \subset T$ with $d(u_1,u_i)=i-1$, $1\leq i\leq n$. Suppose $e_i$ is
the edge joining $u_i$ and $u_{i+1}$. Then for all $1\leq i\leq n$ we first define
a uniformly quasiconvex subset $A_i\subset X_{u_i}$ as follows.

\smallskip
{\bf Construction of $A_i$'s:}

{\bf Type 1.} For $1\leq i\leq n$ if $Y_{e_{i-1},u_i}$ and $Y_{e_{i}u_i}$ are two
cobounded edge spaces in $Y_{u_i}$
then we let $z_{i}\in P^{X_{u_i}}_{Y_{e_i u_i}}(Y_{e_{i-1}u_i})$ and
$z_{i-1}\in P^{X_{u_i}}_{Y_{e_{i-1}u_i}}(Y_{e_iu_i})$. We note that $Y_{e_i,u_i}$ and $Y_{e_{i-1}u_i}$ are also cobounded in $X_{u_i}$ by Lemma \ref{rmk-coboundedness are same}. In this case we define
$A_i= Y_{e_{i-1}u_i}\cup [z_{i-1},z_i]_{X_{u_i}}\cup Y_{e_i u_i}$.

{\bf Type 2.}
On the other hand if $Y_{e_i,u_i}$ and $Y_{e_{i-1}u_i}$ are not cobounded
then by Lemma {\ref{lem-cobdd}~\ref{lem-cobdd:R-sep-D-cobdd}}
$d_{Y_{u_i}}(Y_{e_{i-1},u_i},Y_{e_{i}u_i})\leq D_0$
where $D_0$ depends only on $\delta_0$ and $\lambda_0$. In this case we define
$A_i= Y_{e_{i-1}u_i}\cup Y_{e_i,u_i}$.

\smallskip
{\bf Step 2. Verification of the properties of $A_i$'s.}

\smallskip

{\bf Property 0.} First of all, clearly $A_i$ is {\em uniformly quasiconvex in $X_{u_i}$}.
This follows from the fact that the edge spaces of $Y$ are uniformly qi embedded
in the corresponding vertex spaces of $X$ and that finite union of quasiconvex sets
in a hyperbolic metric space is quasiconvex; see Lemma \ref{lem-cobdd}~\ref{lem-cobdd:union-qc}.

{\bf Property 1.} {\em $Hd_{X_{u_i}}(P^{X_{u_i}}_{X_{e_i u_i}}(A_i), Y_{e_i u_i})$
	and $Hd_{X_{u_i}}(P^{X_{u_i}}_{X_{e_{i-1} u_i}}(A_i), Y_{e_{i-1} u_i})$
	are uniformly small.} This is clear if $A_i$ is of type $2$ using the compatible fiberwise projection compatible.
When $A_i$ is of type $1$ one has to use Lemma {\ref{lem-proj on pair qc} \ref{lem-proj on pair qc:qc proj qc1}} in addition.

On the other hand, since
for $1\leq i\leq n-1$, $Hd_X(Y_{e_i u_i}, Y_{e_i u_{i+1}})\leq 1$,
it follows from the characterization of flow stable Lemma \ref{flow stable lemma2} that the collection {\em $\{A_i:1\le i\le n\}$ is uniformly flow stable family of subsets in $X$.}

{\bf Property 2.} {\em Suppose $e'$ is an edge incident on $u_i$ which is
	not on $[u,v]$. Then $P^{X_{u_i}}_{X_{e' u_i}}(A_i)=B^X_{i,e'}$, say,
	is a uniformly quasiconvex subset in $X_{u_i}$ $($see Lemma {\ref{lem-proj on pair qc}~\ref{lem-proj on pair qc:qc proj qc2}}, and $Hd_{X_{u_i}}(B^X_{i,e'},B^Y_{i,e'})$ is uniformly small}, where $P^{X_{u_i}}_{Y_{e'u_i}}(B^X_{i,e'})=B^Y_{i,e'}$, due to the compatible fiberwise projection condition and Lemma {\ref{lem-proj on pair qc} \ref{lem-proj on pair qc:qc proj qc1}} in addition.
\smallskip

$(*)$ Let $K$ be the maximum of all quasiconvexity constants we have above, and $R=\bm{ R_{fc}}(K)$ and $D=\bm{D_{fc}}(K)$. %Given an edge $e$ incident on a vertex $u$ and a $K$-quasiconvex subset $A\sse Y_{eu}$, we fix $\F l^Y_{R,D}(A)$ is flow space determined by $A$ in $Y$ as in Definition \ref{defn-flow space} with parameters $K$, $R$.

\smallskip
{\bf Step 3. Construction of the flow space.}\\
Now we construct (some modified) flow spaces of the various $A_i$'s in the
direction away from $[u,v]$.
Let $S_i$ be the maximal subtree of $S$ such that $S_i\cap [u,v]=u_i$
and let $Y_i=\pi_1^{-1}(S_i)$ for $1\leq i\leq n$. The modified flow space $\A_i$
of $A_i$ is defined as follows.

{\bf Case 1.} $i=1$ and $e_0=e_1$ or  $i=n$ and $e_n=e_{n-1}$: In this case we let $\A_i$
denote the $\F l^{Y_i}_{R,D}(A_i)$ (flow space of $A_i$ considered only in $Y_i$).

{\bf Case 2:} In all other situation we proceed as follows.
Suppose $e'$ is an edge connecting $u_i$ to $u'_i$, say, such that
$e'$ is not on $[u,v]$. Let $S'_{i}$ be the maximal subtree of $S$ containing
$u'_i$ and not containing $u_i$, and let $Y'_{i}=\pi^{-1}_1(S'_{i})$.
Let $P^{X_{u_i}}_{Y_{e'u_i}}(B^X_{i,e'})=B^Y_{i,e'}$.
%be the nearest point projection in $X_{u_i}$ of $B^X_{i,e}$ on $Y_{e,u_i}$.
We note that
$Hd_{X_{u_i}}(B^X_{i,e'},B^Y_{i,e'})$ is uniformly small by Property (2)
mentioned above. Now if $A_i$ and $X_{e'u_i}$ are $R$-separated (hence, $D$-cobounded)
in $X_{u_i}$, we define the `flow of $A_i$ in the direction of $u'_i$ to be
$\A_{i,e}=\emptyset$. Otherwise, we let $A_{i,e'}=\F l_{u_i u'_i}(B^Y_{i,e'})$ in $Y$ and
$\A_{i,e'}=\F l^{Y'_{i}}_{R,D} (A_{i,e'})$.
We let $\A_i =\bigcup \A_{i,e'}$ where the union is taken over all the edges $e'$
incident on $u_i$, other than $e_i$ and $e_{i-1}$.

\begin{theorem}{\em ({\bf Quasiconvexity of generalized flow spaces})}\label{Q-uv qc}
The set $$Q_{uv}=(\bigcup_{i=1}^n A_i)\bigcup(\bigcup_{i=1}^n \A_i)$$ is uniformly quasiconvex in $X_S$.
\end{theorem}
\proof
Since both the flow stability and saturation conditions are local properties,
i.e., they are to be checked for the edges in $\pi(Q_{uv})$, this follows from
Lemma \ref{Y flow qc in X}, Lemma \ref{flow stable lemma2} and the flow
stability of $\bigcup_{i=1}^n A_i$. One may then apply Theorem \ref{thm-mitra proj on stable-satura} and Lemma \ref{proj-on-qc}~\ref{proj-on-qc:lem-retraction imp qc} to show the quasiconvexity of $Q_{uv}$.
\qed

\smallskip
We recap the results of the construction of the generalized flow spaces in the form
of the following two results.
\begin{theorem}\label{tool 1}{\em ({\bf Geodesic comparison between $\bm Y$ and $\bm {X_S}$})}
There are uniform (non-negative) constants $D_{\ref{tool 1}}$ and $D'_{\ref{tool 1}}$ such that the following hold:

	Suppose $u,v$ are vertices in $S$ and that the $w$ is in $[u,v]$. Let $e$ be the edge on $[w,v]$ $($or $[u,w])$
	incident on $w$.

$(1)$ If $y\in Y_u$ and $y'\in Y_v$ belong to some edge spaces in the respective vertex
	spaces then $N_{D_{\ref{tool 1}}}(Y_{ew})\cap [y,y']_{X_S}\neq \emptyset$

	$(2)$ If $z\in Y_v$ belong to some edge space of $Y_v$ and $z'\in Y_{ew}$ then
	$d_T(u,\pi([z,z']_{X_S}))\geq d_T(u,\pi( \F l^Y_{R,D}(Y_{ew})))-D'_{\ref{tool 1}}$ where $R$ and $D$ are defined in $(*)$ above.
\end{theorem}
\proof  Let us assume that $e$ is on $[w,v]$ since the proof for the other
case is analogous. We will assume the notations of the proof of Theorem \ref{Q-uv qc}.

For (1) we note that $[y,y']_X\cap X_{ew}\neq \emptyset$. Let
$y_1\in [y,y']_X\cap X_{ew}$. By the quasiconvexity of $Q_{uv}$ there
is a point $y_2\in Q_{uv}$ uniformly close to $y_1$. Let $w=u_i$. Next, let $\rho:X_S\map Q_{uv}$
denote a Mitra's projection map restricted on $X_S$ (see Theorem \ref{thm-mitra proj on stable-satura}). Then $\rho(y_1)\in X_w$ is
uniformly close to $\rho(y_2)=y_2$ as $\rho$ is uniformly coarsely Lipschitz.
By triangle inequality, it follows that $d_{X_S}(y_1, \rho(y_1))$ is uniformly bounded. This means
$d_{X_w}(y_1, \rho(y_1))$ is uniformly bounded as $X_w\to X$ is $\eta_0$-proper embedding. We note that in this case
$\rho(y_1)\in A_i$ and $Y_{ew}=Y_{e_i u_i}$. Now, if $\rho(y_1)\in Y_{ew}$
then we are done. However, otherwise, by Lemma \ref{lem-cobdd}~\ref{lem-cobdd:qc proj new},
$[y_1,\rho(y_1)]_{X_{u_i}}$ goes through a uniformly small neighbourhood of
$Y_{ew}$ in $X_{u_i}$ whence (1) follows.

For (2) we appeal to the set $Q_{wv}$
instead of the whole collection. Since $Q_{wv}$ is uniformly quasiconvex in $X_S$,
$[z,z']_{X_S}$ is contained in a uniformly small neighbourhood of $Q_{wv}$.
Therefore, $\pi([z,z']_{X_S})$ is contained in a uniformly small neighbourhood of
$\pi(Q_{wv})$ since $\pi$ is $1$-Lipschitz. Hence, there is a uniform constant $D'\ge0$ such that
$d_T(u, \pi([z,z']_{X_S}))\geq d_T(u,\pi(Q_{wv}))-D'$. However, it is clear from the
construction of $Q_{wv}$ that $d_T(u, \pi(Q_{wv}))\geq d_T(u,\pi(\F l^Y_{R,D}(Y_{ew})))$
whence (2) follows. Hence, we take $D'=D'_{\ref{tool 1}}$. \qed\smallskip

%\begin{remark}
%We note that Lemma \ref{Y flow qc in X} and Theorem \ref{tool 1} are
%used in the proof of our main theorem to compare geodesics of $Y$ and those
%of $X_S$.
%\end{remark}

The following lemma is the final piece of tool for the proof of our main theorem
to be given in the next section of the paper.
\begin{lemma}\label{reduction1}{\em ({\bf A superficial horizontal replacement lemma})}
Suppose $u\in V(S)$, $y\in Y_u$ and $\{y_n\}$ is an unbounded
sequence in $Y$ such that $\lim^Y_{n\map \infty}y_n\in \pa Y$ and
$\lim^X_{n\map \infty}y_n \in \pa X$. Let $u_n=\pi(y_n)$, and suppose that
$\lim^T_{n\map \infty} u_n=\xi \in \pa T$. Let $c_n$ be the nearest point projection of
$u_n$ on $[u,\xi)$, and let
$e_n$ be the edge on $[u,c_n]$ incident on $c_n$ for all $n\in \mathbb N$. Suppose
$z_n\in Y_{e_nc_n}\cap [y, y_n]_Y$. Then
$\lim^Y_{n\map \infty}y_n=\lim^Y_{n\map \infty}z_n\in\pa Y$ and
$\lim^X_{n\map \infty}y_n=\lim^X_{n\map \infty}z_n\in\pa X$; in other words,
$\{y_n\}$ and $\{z_n\}$ are a pair of superficial test sequences for the pair $(Y,X)$.
\end{lemma}

\proof That $\LMY y_n =\LMY z_n$ follows from Lemma \ref{horiz replace}.
The only significant point here is to prove the second limit.

For all $n\in \N$ let $e'_n$ be the edge on $[u,u_n]$ incident on $u_n$, and
let $$e_n=e_{n,1},e_{n,2},\cdots, e_{n,t(n)}=e'_n$$ be all the successive edges on the
geodesic from $c_n$ to $u_n$ when $c_n\neq u_n$, i.e. $u_n\not \in [u,\xi)$.
First we do the following reduction.

{\bf Reduction step}:
Let $y'_n$ be a nearest point projection of $y_n$ on
$Y_{e'_n u_n}$ in $Y_{u_n}$. Then by Lemma \ref{reduction0},
$\LMY y_n =\LMY y'_n$. Moreover, by the same lemma and the compatible fiberwise projection condition,
we have $\lim^{X}_{n\to\infty} y_n =\lim^{X}_{n\to\infty} y'_n$ in $\pa X$. Therefore, we are reduced to proving
$\LMX y'_n=\LMX z_n$ in $\pa X$ and so far we have $\LMY y'_n =\LMY z_n$ in $\pa Y$. 

In the proof one will notice that the study is restricted in $Y\sse X_S$ due to Theorem \ref{thm-CT-ps-kap}.

The following recurring argument in the proof. First, we discuss three special cases
and then finally we prove the general case using them. In the proof we will consider $(R,D)$-flow spaces for the constants $R$ and $D$ as in Theorem \ref{tool 1} $(2)$.

{\bf Case 1.} Suppose $u_n\in [u,\xi)$ for all $n\in \N$. In this case $e_n=e'_n$ and the
reduction step yields $y'_n\in Y_{e_n c_n}$. Now note that $y'_n, z_n \in Y_{e_n c_n}$,
and $Z_n=\F l^Y_{R,D}(Y_{e_n c_n})$ is uniformly quasiconvex in both $X_S$ and $Y$ by Lemma \ref{Y flow qc in X} and Corollary \ref{flow-qc} respectively. 

Hence, for any fix $y\in Y_u$, $d_{X_S}(y,[y'_n,z_n]_{X_S})\to\infty$ as $n\to\infty$ by applying Lemma \ref{lem-same limits} in $Y\sse X_S$. Since $X_S\to X$ admits CT map (Theorem \ref{thm-CT-ps-kap}), by the necessity of Mitra's criterion (Lemma \ref{necessity}), we have $d_X(y,[y'_n,z_n]_X)\to\infty$ as $n\to\infty$. Hence $\LMX y'_n=\LMX z_n$ by Lemma \ref{con-same-pt}.

{\bf Case 2.} Suppose $u_n\not \in [u,\xi)$ but
$\F l^Y_{R,D}(Y_{e_{n,t(n)}u_n}) \cap Y_{e_nc_n}\neq \emptyset$ for all $n\in \N$. In this case
$\F l^Y_{R,D}(Y_{e_{n,t(n)}u_n}) \cup \F l^Y_{R,D}(Y_{e_nc_n})=Z_n$, say, is uniformly
quasiconvex in $Y$ as well in $X_S$ by Corollary \ref{flow-qc} and Lemma \ref{Y flow qc in X} and the fact that in a hyperbolic metric space, union of two intersecting quasiconvex subsets is also quasiconvex.
Note that $y'_n,z_n\in Z_n$. A carbon copy of the argument given in the last paragraph of Case $1$ concludes that $\LMX y'_n=\LMX z_n$ in $\pa X$.

{\bf Case 3.} Suppose $u_n\not \in [u,\xi)$ and $\F l^Y_{R,D}(Y_{e_{n,t(n)}u_n})\cap Y_{e_nc_n}=\emptyset $
for all $n\in\N$. Suppose $c_n=v_{n,1}, v_{n,2},\cdots$ are the consecutive vertices
on the geodesic joining $c_n$ to $u_n$ so that each edge $e_{n,j}$ joins $v_{n,j}$ and $v_{n,j+1}$.
Suppose $e_{n,i}$ is the closest edge from $c_n$ such that
$\F l^Y_{R,D}(Y_{e_{n,i}v_{n,i+1}})\cap Y_{e_nc_n}=\emptyset$. Let $y''_n\in Y_{e_{n,i}v_{n,i+1}}$ be any point.
Now, by Theorem \ref{tool 1} $(2)$, there is a uniform constant $D'=D'_{\ref{tool 1}}$ depending
	only on the parameters of the tree of spaces under consideration such that
$d_T(u, \pi([y'_n, y''_n]_{X_S})\geq d_T(u, \F l^Y_{R,D}(Y_{ e_{n,i}v_{n,i+1}}))-D'$. However,
since $\F l^Y_{R,D}(Y_{ e_{n,i}v_{n,i+1}})\cap Y_{e_{n} c_n}=\emptyset$, we have
$d_T(u, \F l^Y_{R,D}(Y_{ e_{n,i}v_{n,i+1}}))\geq d(u,c_n)$. Thus for any fix $y\in Y_u$,
$\lim_{n\map \infty}d_{X_S}(y, [y'_n, y''_n]_{X_S})=\infty$. On the other hand, $X_S\to X$ admits a CT map (Theorem \ref{thm-CT-ps-kap}) and so $X_S\to X$ satisfies Mitra's criterion (Lemma \ref{necessity}), and hence, $\lim_{n\to\infty}d_X(y,[y'_n,y''_n]_X)=\infty$. Therefore, by Lemma \ref{con-same-pt},
$\lim^{X}_{n\to\infty} y'_n=\lim^{X}_{n\to\infty} y''_n$ in $\pa X$. Applying Theorem \ref{tool 1} $(2)$ to the case $Y=X$
we similarly get $\LMY y'_n=\LMY y''_n$ in $\pa Y$. Let $Im_Y(y''_n)$ denote the image of
$y''_n$ in $Y_{ e_{n,i}v_{n,i}}$ (Subsection \ref{tree-of-sps}). Since $d_Y(y''_n, Im(y''_n))=1$, it follows that
$\lim^{X}_{n\to\infty} y'_n =\lim^{X}_{n\to\infty} y''_n=\lim^{X}_{n\to\infty} Im(y''_n)$ in $\pa X$ and $\LMY y'_n =\LMY y''_n=\LMY Im(y''_n)$ in $\pa Y$.

Next, we may apply the reduction step to $Im(y''_n)$ to find
$y'''_n\in Y_{e_{n,i-1}v_{n,i}}$ for each $n\in \N$ so that
$\lim^{X}_{n\to\infty} Im(y''_n)=\lim^{X}_{n\to\infty} y'''_n$ in $\pa X$ and $\LMY Im(y''_n)=\LMY y'''_n~(=\LMY z_n)$ in $\pa Y$.
Finally, since $\F l^Y_{R,D}(Y_{e_{n,i-1}v_{n,i}}) \cap Y_{e_n c_n}\neq \emptyset$
for each $n\in \N$ we have $\lim^{X}_{n\to\infty} z_n=\lim^{X}_{n\to\infty} y'''_n$ in $\pa X$ by Case $2$. Therefore, combining all the equalities, we have $\lim^{X}_{n\to\infty} z_n=\lim^{X}_{n\to\infty} y'_n$ in $\pa X$ as required.

{\bf Case $4$. The general case:} Let $S_1=\{n\in \N: u_n\in [u, \xi)\}$,
$S_2=\{n\in \N\setminus S_1: \F l^Y_{R,D}(Y_{e_{n,t(n)}u_n})\cap Y_{e_n c_n}\neq \emptyset \}$, and
let $S_3=\N \setminus (S_1\cup S_2)$. Now, if any $S_i$, $1\leq i\leq 3$, is infinite then we have a subsequence $\{n_{ik}\}_{k\in \N}$ of the sequence of natural numbers such that
$S_i=\{n_{ik}: k\in \N\}$. Then Case $i$ applies to the subsequence $\{y'_{n_{ik}}\}$ of $\{y'_n\}$ to
give $\lim^X_{k\map \infty} y'_{n_{ik}} =\lim^X_{k \map \infty} z_{n_{ik}}$ in $\pa X$.
Since $\LMX y'_n$ exists in $\pa X$, it follows that $\LMX y'_n =\LMX z_n$ in $\pa X$.
\qed\smallskip

The following corollary is a twin of Lemma \ref{vert or horiz} and it is immediate
from Lemma Lemma \ref{vert or horiz}, Lemma \ref{bounded proj 2} and Lemma \ref{reduction1}.
\begin{cor}\label{vert or horiz replace}
Suppose $\{y_n\}$ is an unbounded sequence in $Y$ such that
$\LMX y_n \in \pa X$ and $\LMY y_n\in \pa Y$. Then there is a sequence
$\{z_n\}$ in $Y$ such that

$(i)$ $\{y_n\}$ and $\{z_n\}$ is a pair of superficial test sequences.

$(ii)$ $\{z_n\}$ is either a vertical sequence or a horizontal sequence.
\end{cor}

%Therefore, we assume that $S=T$. Then  we use Lemmas/Propositions/Corollaries developed in Section \ref{prelim} where subtree of subspaces involved with the same base ($S=T$).

\section{Main theorem}\label{sec-main thm}
In this section we prove the following theorem which is the main result of this paper.
Although we have already mentioned Convention \ref{con-induced spaces} and
Convention \ref{subsubsec-induced trees} before, we list all the hypotheses here for
convenience of the reader.

\begin{theorem}\label{main thm}
	
	\begin{enumerate}
		\item $\pi:X\map T$ is a tree of hyperbolic metric spaces with parameters $\eta_0, \delta_0, L_0,\lm_0$
		as defined in Definition \ref{tree-of-sps}.
		
		\item $Y\subset X$ is a subtree of hyperbolic spaces with fiberwise CT maps.
		Suppose $\pi(Y)=S$ and $\pi_1=\pi|_S: Y\map S$ is a tree of hyperbolic metric spaces
		with qi embedded condition with parameters $\eta_0, \delta_0, L_0$ .
		
		\item Both $X$ and $Y$ are proper hyperbolic metric spaces and the inclusion $Y\map X$ is
		$\eta_0$-proper embedding.
		
		\item There is a constant $L\geq 1$ such that the inclusion $Y_e\map X_e$ is an $L$-qi embedding
		for all $e\in E(S)$.
		\item Compatible fiberwise projection condition, with constant $R_0$, holds for
		the pair of trees of spaces $(Y,X)$.
	\end{enumerate}
	Then the inclusion $i:Y \map X$ admits the CT map.
\end{theorem}

%Before we prove this theorem we start with some auxiliary results.
%\smallskip

%As the proof is a bit long some general helpful remark are in order.

\noindent{\bf Proof idea.} First of all, by Theorem \ref{thm-CT-ps-kap} (and Lemma \ref{functo-ct-map}), we may reduce to the case $S=T$ (see Step $1$ below). To prove the theorem, we show that any pair of test sequences for the pair $(Y,X)$ forms a pair of superficial test sequences. The conclusion then follows from Lemma \ref{not-ct-not-comp}. 

Let $\{y_n\}, \{y'_n\}$ be a pair of test sequences. By Lemma \ref{vert or horiz}, each of these sequences is either vertically replaceable or horizontally replaceable. Using Corollary \ref{vert or horiz replace} (see also Lemma \ref{reduction1}), we replace them by their vertical or horizontal counterparts, denoted $\{w_n\}$ and $\{w'_n\}$, respectively. This reduces to the following three cases.

\begin{enumerate}
\item Both $\{w'_n\}$ and $\{w_n\}$ are vertical sequences. We are done by Lemma \ref{finite-proj-case}. 

\item $\{w'_n\}$ is vertical and $\{w_n\}$ is horizontal.

\item Both $\{w'_n\}$ and $\{w_n\}$ are horizontals.
\end{enumerate}

In Case $2$, we replace $\{w_n\}$ by a vertical sequence $\{z_n\}$ such that $\{w_n\}$ and $\{z_n\}$ form a pair of superficial test sequences, reducing to Case $1$. The key input here is that the flow spaces of edge spaces of $Y$ are uniformly quasiconvex in both $Y$ and $X$ (Corollary \ref{flow-qc} and Lemma \ref{Y flow qc in X}). Finally, using the key input, in Case $3$, if $\lim_{n\to\infty}^T \pi(w_n) \neq \lim_{n\to\infty}^T \pi(w'_n)$, we again replace one of the sequences by a suitable vertical sequence, reducing to Case $2$; otherwise, we conclude by exhibiting a uniformly quasiconvex subset of both $Y$ and $X$ containing $w'_n$ and $w_n$ for all large $n\in\N$.

%If both $\{w_n\}$ and $\{w'_n\}$ are vertical sequences, then they form a pair of superficial test sequences by Lemma~\ref{finite-proj-case}.   If $\{w'_n\}$ is vertical and $\{w_n\}$ is horizontal, then we replace $\{w_n\}$ by a vertical sequence $\{z_n\}$ such that $\{w_n\}$ and $\{z_n\}$ form a pair of superficial test sequences, reducing to the previous case. The key input here is that the flow spaces of edge spaces of $Y$ are uniformly quasiconvex in both $Y$ and $X$ (Corollary \ref{flow-qc} and Lemma \ref{Y flow qc in X}).   Finally, suppose both $\{w_n\}$ and $\{w'_n\}$ are horizontal. If $\lim_{n\to\infty}^T \pi(w_n) \neq \lim_{n\to\infty}^T \pi(w'_n)$, then we again replace one of the sequences by a suitable vertical sequence, reducing to the earlier case. If instead $\lim_{n\to\infty}^T \pi(w_n) = \lim_{n\to\infty}^T \pi(w'_n)$, we conclude by exhibiting a uniformly quasiconvex subset of both $Y$ and $X$ containing $w'_n$ and $w_n$ for all large $n\in\N$.

%Sometimes we may need to do this a number of times. \RED{Is this correct?}

\subsection{Proof of Theorem \ref{main thm}}.

\noindent{\bf Step 1. Reduction to $\bm{S=T}$.} We first note that without loss
of generality one may assume that $X, Y$ are defined on the same tree, i.e. $S=T$.
By Theorem \ref{thm-CT-ps-kap}, there is a CT map for the inclusion $i:X_S\map X$.
Therefore, it is enough, by Lemma \ref{functo-ct-map}, to show the existence of a CT map for the
inclusion $i:Y\ri X$ where both $X$ and $Y$ have same base $T$. Hence, we are going to assume
that $S=T$ for the rest of the proof.

Suppose $\{y_n\}$ and $\{y'_n\}$ are two test sequences. We will show that these form a pair of superficial test sequences.\smallskip

\noindent{\bf Step 2. Reduction to vertical or horizontal sequences.}
By Corollary \ref{vert or horiz replace} we may assume that each of the sequences
$\{y_n\}$ and $\{y'_n\}$ is either vertical or horizontal. Rest of the proof
is divided into three cases based on nature of these sequences.\smallskip

\noindent{\bf Step 3. Case by case analysis based on types of  $\bm{\{y_n\}}$ and $\bm{\{y'_n\}}$.}
%Let $b_n=\pi(y_n)$ and $b\pr_n=\pi(y\pr_n)$ for all $n\in\N$. Let $T_1=Hull(\{b_n:n\in\N\})$ and $T_2=Hull(\{b\pr_n:n\in\N\})$. 
%We have the following three possibilities:
\smallskip

\noindent{\bf Case 1}: \underline{Both $\{y_n\}$ and $\{y\pr_n\}$
	are vertical.}
In this case the proof follows from Lemma \ref{finite-proj-case}.\smallskip

\noindent{\bf Case 2}: \underline{Exactly one of the sequences $\{y_n\}$ or $\{y\pr_n\}$
	is vertical and the other one is horizontal.}
Without loss of generality, suppose $\{y\pr_n\}$ is vertical
and $\{y_n\}$ is horizontal. Let $u,w,u_n\in V(T)$ and $\xi\in\pa T$ be such that $\{y'_n\}\sse Y_u$, $\pi(y_n)=u_n\in[w,\xi)$ and $\lim^T_{n\to\infty}u_n=\xi$. %Then nearest point projection of $\{\pi(y'_n)\}$ on $[w,\xi)$ is a single point (i.e., bounded). Then by Lemma \ref{reduction 0.1}, there is a vertical sequence $\{p_n\}\sse Y$ such that $\{y'_n\}$, $\{p_n\}$ is a pair of test sequences and for which we need to show that $\{y'_n\}$, $\{p_n\}$ is a pair of superficial test sequences. (Compare the sequences $\{y'_n\}$, $\{p_n\}$ and $\{y_n\}$ with $\{s'_n\}$, $\{t_n\}$ and $\{s_n\}$ respectively in Lemma \ref{reduction 0.1}.) Since both $\{y'_n\}$ are $\{p_n\}$ vertical sequences, we are done in this case by Lemma \ref{finite-proj-case}.
By discarding finitely many initial terms of the sequence $\{y_n\}$, we may assume that $u_n\in[u,\xi)$. Finally, using Lemma \ref{reduction1}, we may also assume that $y_n\in Y_{e_nu_n}$ where $e_n$ is the edge in $[u,u_n]$ incident on $u_n$. Let $e$ be the edge in $[u,\xi)$ incident on $u$. Fix $z_n\in Y_{eu}\cap[y'_n,y_n]_Y$. 

In the following claim, we will show that $\{z_n\}$, $\{y_n\}$ form a pair of superficial test sequences. First we complete the proof using the claim. From this claim, we also have $\LMX z_n\in\pa X$. Note that $\LMY y'_n=\LMY z_n=\LMY y_n$ by Lemma \ref{con-same-pt} as $\LMY y'_n=\LMY y_n$. Then $\{y'_n\},\{z_n\}$ form a pair of vertical test sequences, and so form a pair of superficial test sequences by Lemma \ref{CT-from-a-ver}.

This concludes that $\{y'_n\}$, $\{y_n\}$ form a pair of superficial test sequences. Hence we are done in this case.\smallskip

\noindent{\em Claim: $\{z_n\},~\{y_n\}$ form a pair of superficial test sequences.}\smallskip

\noindent{\em Proof of the claim:} Let us fix some constants to consider flow spaces. Note that $Y_{eu},Y_{e_nu_n}$ are $\lm_0$-quasiconvex in $Y_u$ and $Y_{u_n}$ respectively (see Convention \ref{con-trees of metric spaces} $(4)$). Let $R=\bm{R_{fc}}(\lm_0)$ and $D=\bm{D_{fc}}(\lm_0)$ as defined in Subsection \ref{subsub-flow space}.
Now, applying Lemma \ref{not-hv-pro-sub} to the sequences $\{z_n\},\{y_n\}$ in $Y$
we have $$N^Y_{D_1}(\F l^Y_{R,D}(Y_{eu}))\bigcap N^Y_{D_1}(\F l^Y_{R,D}(Y_{e_nu_n}))\ne\emptyset$$ for some uniform constant $D_1\ge0$.  Note that flow spaces under consideration are uniformly quasiconvex in both $Y$ and $X$ (see Corollary \ref{flow-qc} and Lemma \ref{Y flow qc in X}), and the union of two intersecting quasiconvex subsets is also quasiconvex in a hyperbolic space. Hence $N^Y_{D_1}(\F l^Y_{R,D}(Y_{eu}))\bigcup N^Y_{D_1}(\F l^Y_{R,D}(Y_{e_nu_n}))$ is uniformly quasiconvex in both $Y$ and $X$ containing $z_n$ and $y_n$. Now the conclusion follows from Lemma \ref{lem-same limits}.\smallskip

\noindent{\bf Case 3}: \underline{ Suppose $\{y\pr_n\}$ and  $\{y_n\}$ are both horizontal.} Let $u',u,b'_n,b_n\in V(T)$ and $\xi',\xi\in \partial T$ be such that $\pi(y'_n)=b'_n\in[u',\xi')$, $\pi(y_n)=b_n\in [u,\xi)$ and $\lim^T_{n\to\infty}b'_n=\xi'$, $\lim^T_{n\to\infty}b_n=\xi$. Let $e'_n$ (resp. $e_n$) be the edge in $[u',\xi')$ (resp. in $[u,\xi)$) incident on $b'_n$ (resp. on $b_n$). Finally, using Lemma \ref{reduction1}, we may assume that $y'_n\in Y_{e'_nb'_n}$ and $y_n\in Y_{e_nb_n}$.\smallskip

\noindent{\bf Subcase 3A}: Suppose $\xi' \ne\xi$. By discarding finitely many initial terms of the sequences $\{y'_n\}$ and $\{y_n\}$, we may assume that $(\xi',u]\cup[u,\xi)$ is a geodesic line. Let $e$ be the edge in $[u,\xi)$ incident on $u$. Fix $z_n\in Y_{eu}\cap[y'_n,y_n]_Y$. A carbon copy of the proof shown in the {\em Claim} of Case $2$ shows that $\{z_n\}$, $\{y_n\}$ form a pair of superficial test sequences. It therefore remains to show that the pair of test sequences $\{y'_n\}$, $\{z_n\}$ form a pair of superficial test sequences. This, however, directly follows from Case $2$ as $\{z_n\}$ is vertical.\smallskip

\noindent{\bf Subcase 3B}: Suppose $\xi=\xi'$. By discarding finitely many initial terms of the sequences $\{y'_n\}$ and $\{y_n\}$, we may assume that $b_n,b'_n$ are vertices in $[u,\xi)$.
Set $u_n=b_n$ and $\MF e_n=e_n$ if $d_T(u,b_n)\le d_T(u,b'_n)$; and $u_n=b'_n$ and $\MF e_n=e'_n$ if $d_T(u,b'_n)\le d_T(u,b_n)$.

{\em We need to show that $\{y_n\}$, $\{y'_n\}$ form a pair of superficial test sequences, i.e. $\LMX y_n=\LMX y'_n$.} On contrary, suppose $\LMX y_n\neq \LMX y'_n$. Then for any fixed $y\in Y_u$, there is a $D_1\in \N$ such that $d_X(y, [y_n,y'_n]_X)\leq D_1$ 
for all $n\in \N$. Fix constants $R=\bm{R_{fc}}(\lm_0)$ and $D=\bm{D_{fc}}(\lm_0)$ as defined in Case $2$. Consider the $(R,D)$ flow space $\F l^Y_{R,D}(Y_{\MF e_nu_n})$. Note that any geodesic $[y_n,y'_n]_X$ in $X$ passes through $Y_{\MF e_nu_n}$. Since the flow spaces $\F l^Y_{R,D}(Y_{\MF e_nu_n})$ are uniformly quasiconvex in $X$ (see Lemma \ref{Y flow qc in X}), there is a uniform constant $D'\ge0$ such that for all $n\in\N$, we have $d_X(y,\F l^Y_{R,D}(Y_{\MF e_nu_n}))\le D'$. Therefore, for all large $n\in\N$, we have $\F l^Y_{R,D}(Y_{ e_n u_n})\bigcap \F l^Y_{R,D}(Y_{ e'_n u'_n})\neq \emptyset$.

Note that the flow spaces under consideration are uniformly quasiconvex in both $Y$ and $X$ (see Corollary \ref{flow-qc} and Lemma \ref{Y flow qc in X}), and the union of two intersecting quasiconvex subsets is also quasiconvex in a hyperbolic space. Hence for all large $n$, $Z_n=\F l^Y_{R,D}(Y_{ e_n b_n})\bigcup \F l^Y_{R,D}(Y_{ e'_n b'_n})$ is uniformly quasiconvex in both $Y$ and $X$. Moreover, $y_n,y'_n\in Z_n$. Therefore, by Lemma \ref{lem-same limits}, $\LMX y_n=\LMX y'_n$ -- which contradicts our assumption. This completes the proof of Theorem \ref{main-theorem-ct}.\qed\smallskip

We will end by mentioning two particular cases of Theorem \ref{main thm}. As we saw in Lemma \ref{lem-uniform Mitra imp proj con}, one can obtain the compatible fiberwise projection condition when the inclusions between edge spaces are uniformly quasiisometries. Thus by a simple application of our main theorem (Theorem \ref{main thm}), we have the following.

\begin{cor}\label{thm-Xe=Ye}
	Suppose we have Convention \ref{con-induced spaces}. Further, assume that for all $u\in V(S)$ and for all $e\in E(S)$, the inclusion $Y_u\to X_u$ satisfies the uniform Mitra's criterion (see Definition \ref{mitra-irr}) and that $Y_e=X_e$. Then the inclusion $i:Y\to X$ admits a CT-map.
\end{cor}

%\RED{We will move to the introduction or somewhere earlier:} \\
The following theorem is essentially proved in \cite{ps-kap}
(see Proposition 8.70 and Theorem 8.71 there) but not mentioned explicitly.
We include a proof for the sake of completeness. %Projection hypothesis no needed???
%No idea where to keep this. Will worry later.
\begin{theorem}\label{thm-CT proj plus qi emb}
	Suppose $\pi:X\ri T$ is a tree of hyperbolic spaces with the qi embedded
	condition and $\pi|_Y:Y\ri S$ is an induced subtree of hyperbolic subspaces
	with the qi embedded condition. Further, assume that the vertex and edge spaces of $Y$
	are uniformly qi embedded in the corresponding vertex and edge spaces of $X$.
Moreover, assume that the compatible fiberwise projection condition holds for
the pair of trees of spaces $(Y,X)$. Then:
	\begin{enumerate}
		\item The inclusion $Y\ri X_S$ is (uniform) qi embedding.
		
		\item  The inclusion $Y\ri X$ admits the CT map.
	\end{enumerate}
\end{theorem}
\proof (1) By Proposition \ref{induced subtree consistent}, $\{Y_v: v\in V(S)\}=\Y$, say,
is a flow stable family of subsets in $X$. Then by Theorem \ref{thm-mitra proj on stable-satura},
$\bigcup_{v\in V(S)} Y_v=Y'$, say, is quasiconvex in $X_S$ as $\Y$ is clearly a saturated
collection of subsets in $X_S$. Since $Hd_X(Y,Y')\leq 1$, it follows that $Y$ is quasiconvex
in $X_S$ too. Now, since $Y$ is properly embedded it follows from Lemma \ref{qi-emb-in-Y-X} $(2)$
that $Y$ is qi embedded in $X_S$.

(2) By Theorem \ref{thm-CT-ps-kap}, the inclusion $X_S\map X$ admits a CT map. By (1) and Lemma \ref{qi-im-ct-map}, the inclusion $Y\map X_S$ admits a CT map. Hence,
by the functoriality of CT maps (see Lemma \ref{functo-ct-map}) the inclusion $Y\map X$
admits the CT map.
\qed

\begin{remark}
We note that, in the above theorem, the compatible fiberwise projection condition is not necessary for the conclusion, as this can be demonstrated by trivial examples -- for instance, by examples similar to Corollary \ref{new cor sec 6}.
\end{remark}

%Mitra presented a criterion, as described in \cite{mitra-trees}, which ensures the existence of a CT-map for a proper embedding between hyperbolic metric spaces (See Subsection \ref{Gro-b-CT-maps}). Furthermore, Lemma \ref{necessity} establishes that this criterion is also necessary for the existence of a CT-map.

\section{Cannon--Thurston lamination}\label{sec-CT lamination}
In this section we deal with the Cannon--Thurston lamination
for the pairs of trees of spaces $(Y,X)$ as in Theorem \ref{main thm}.

\begin{defn}[Cannon--Thurston lamination or CT lamination, for short]
	Suppose $f:X_1\map X_2$ is a proper embedding of hyperbolic metric spaces which admits the CT map $\pa i:\pa X_1\map\pa X_2$. Then the CT lamination for $f$ is defined to be $$\L_{CT}(X_1,X_2):=\{(p,q)\in\pa^2 X_1:\pa i(p)=\pa i(q)\}.$$
\end{defn}
In this connection, if $(p,q)\in \L_{CT}(X_1,X_2)$ and $\alpha$ is a geodesic line in $X_1$ joining $p,q$. Then we shall refer to $\alpha$ as a {\em leaf of the CT lamination} $\L_{CT}(X_1,X_2)$.

Suppose $H < G$ are hyperbolic groups and the CT map $\pa H \to \pa G$ exists.  Then the CT lamination is empty iff $H$ is quasiconvex in $G$ (see \cite[Lemma $2.1$]{mitra-pams}). Hence, in a specific context understanding the CT lamination for a group pair $H<G$ helps one to deduce quasiconvexity results. For instance, using Mitra's characterization of the CT lamination as {\em ending lamination} in
\cite{mitra-endlam}, in case $H$ is normal in $G$, this sort of quasiconvexity results are proved in \cite{mitra-pams}, \cite{mj-rafi}, \cite{mahan-pranab-qc}. Using the results of this section, we also prove a quasiconvexity result ( see Theorem \ref{thm-qc}) in the next section of the paper.

Coming back to the context, suppose we have a tree of spaces $\pi:X\map T$ as in Theorem \ref{main thm}. We note that CT lamination for $(X_u, X)$ for $u\in V(T)$ and
more generally for $(X_{S}, X)$ for a subtree $S\subset T$ have been
studied in \cite{ps-kap}. The results (see Theorems \ref{thm-at least one contain ray} and \ref{thm-both does not con geo ray}) in this section are motivated by these. 

\subsection{Recap of results on $\L_{CT}(X_S,X)$}
We  next recall a few results from \cite{ps-kap} which 
will be needed in our proofs. We start with a definition.
\begin{defn}
Suppose $u\in V(T)$ and $\xi\in\pa T$. Define $$\pa^{\xi}(X_u,X):=\{\eta\in\pa i_{X_u,X}(\pa X_u):\exists\text{ qi lift, }\gm\text{ say, of }[u,\xi)\text{ such that }\gm(\infty)=\eta\}$$and $$\Lm^{\xi}(X_u,X):=\{(p,q)\in\pa^2X_u:\pa i_{X_u,X}(p),\pa i_{X_u,X}(q)\in\pa^{\xi}(X_u,X)\}.$$
\end{defn}

We mention that a CT leaf for the pair $(Y,X)$ that does not come from a fiber, i.e. from $\L_{CT}(Y_u,X_u)$ for some vertex $u$ of $S$, arises
from $\Lm^{\xi}(X_u,X)$ for some $\xi\in\pa T\setminus \pa S$ (see Theorem \ref{thm-both does not con geo ray} for a precise statement).

\begin{theorem}\textup{(\cite[Theorem $8.54$ $(1)$]{ps-kap})}\label{thm-both have bdry flow}
Let $u\in V(T)$. Suppose $\al:\R\map X_u$ is a geodesic line such that $(\al(-\infty),\al(\infty))\in\Lm^{\xi}(X_u,X)$ for some $\xi\in\pa T$. Then both of $\al(\pm\infty)$ have boundary flows in $X_v$ for all vertex $v$ in $[u,\xi)$.
\end{theorem}

%\RED{Avoid flaring altogether unless explicitly used elsewhere:}
We will not go into details regarding the threshold constant $M$ in the following theorem as that is not important to us. It is the Bestvina--Feighn flaring (threshold) constant appearing in Convention \ref{con-trees of metric spaces} $(8)$.

\begin{theorem}\textup{(\cite[Theorem $8.49$, Theorem $8.60$]{ps-kap})}\label{thm-quasigeo-descrip}
	Let $u\in V(T)$ and $\al$ be a geodesic line in $X_u$. Assume that $\al$ is a leaf of the CT lamination $\L_{CT}(X_u,X)$. Then there is $k\ge1$ and a threshold constant $M>0$ depending on $k$ such that for any $R\ge M$, a uniform quasigeodesic joining $\al(-m)$ and $\al(n)$ in $X$ has the form$$\gm_{-m}*[\gm_{-m}(w),\gm_n(w)]_{X_w}*\gm_n$$where $w\in V(T)$, and 
	
	$(1)$ $\gm_{-m}$, $\gm_n$ are $k$-qi lifts of $[u,w]$ such that $\gm_{-m}(u)=\al(-m)$, $\gm_n(u)=\al(n)$, 
	
	$(2)$ $d_{X_v}(\gm_{-m}(v),\gm_n(v))\ge R$ for all vertex $v$ in $[u,w]\setminus\{w\}$ and 
	
	$(3)$ $d_{X_w}(\gm_{-m}(w),\gm_n(w))\le C$ for some fixed $C\ge0$ depending on $R$.

	Moreover, we may assume that all such $w$ lies in $[u,\xi)$ for a unique $\xi\in\pa T$. In particular, $(\al(-\infty),\al(\infty))\in\Lm^{\xi}(X_u,X)$.		
\end{theorem}

\begin{remark}\label{rmk-lifts in geo flow}
	Suppose we have the condition of Theorem \ref{thm-quasigeo-descrip}. Suppose that for all vertex $v$ in $[u,\xi)$, $\al_v$ is a geodesic line in $X_v$ such that $\al_v(\pm\infty)$ are boundary flows of $\al(\pm\infty)$ in $X_v$ (see Theorem \ref{thm-both have bdry flow}). Then we may assume that $\gm_{-m}(v'),\gm_n(v')\in\al_{v'}$ for all vertex $v'$ in $[u,w]$.
\end{remark}

The converse of Theorem \ref{thm-quasigeo-descrip} is also true which is as follows. In Proposition \ref{prop-con of geo des}, it follows from the the description of quasigeodesics (\cite[Chapter $7$]{ps-kap}, see also \cite[Section $3$]{pranab-mahan}) that the concatenation of paths $\gm_{-m}*[\gm_{-m}(w),\gm_n(w)]_{X_w}*\gm_n$ is a uniform quasigeodesic. Since these paths are going far and far from any base point (\cite{mitra-trees}), it follows that endpoints of these paths converge to the same boundary point.

\begin{prop}\label{prop-con of geo des}
%	Suppose $\pi:X\map T$ is a hyperbolic tree of hyperbolic metric spaces with the qi embedded condition. 
Let $u\in V(T)$ and let $\al$ be a geodesic line in $X_u$. Fix $k\ge1$, $M>0$ depending on $k$ and $R\ge M$. Suppose that for all large $n,m\in\N$, we have $w\in V(T)$ and $k$-qi lifts $\gm_{-m}$, $\gm_n$ of $[u,w]$ with $\gm_{-m}(u)=\al(-m)$ and $\gm_n(u)=\al(n)$, satisfying the following. For all vertex $v$ in $[u,w]\setminus\{w\}$, we have $$d_{X_v}(\gm_{-m}(v),\gm_n(v))\ge R\text{ and }d_{X_{w}}(\gm_{-m}(w),\gm_n(w))\le C$$ for some fixed $C\ge0$.
	
	Then $\al$ is a leaf of the CT lamination $\L_{CT}(X_u,X)$.
\end{prop}

We end this subsection with the following result
which works as a motivation for the results proved in this section. Note that for any subtree $T'\sse T$, by Theorem \ref{thm-CT-ps-kap}, the inclusion $X_{T'}\map X$ admits a CT map. 
\begin{theorem}\textup{(\cite[Proposition $8.47$]{ps-kap})}\label{thm-unbdd imp not equal}
Suppose $T'\sse T$ is any subtree. Assume that $\al$ is a quasigeodesic line in $X_{T'}$ such that $\pi(\al)$ is unbounded. Then $\al$ is not a leaf of the CT lamination $\L_{CT}(X_{T'},X)$. %$$\pa i_{X_{T'},X}(\al(-\infty))\ne\pa i_{X_{T'},X}(\al(\infty)).$$
\end{theorem}

%%%%%%%%%%%%%%%%%%%%%%%%%%
%%%%%%%%%%%%%%%%%%%%%%%%%%%%%%%
\subsection{Main results on CT Lamination}\label{subsec-CT lamination}

%In this section we study the CT lamination for the CT map $\pa i_{Y,X}:\pa Y\map \pa X$ where $Y$ and $X$ are as in Theorem \ref{main thm}. 
Following are the main results of this section.

\begin{theorem}\label{thm-at least one contain ray}
Suppose $Y\sse X$ are hyperbolic spaces satisfying the assumptions of Theorem \ref{main thm}. Suppose $\al$ is a geodesic line in $Y$ and that $\al$ is a leaf of the CT lamination
$\L_{CT}(Y,X)$. Then $\pi (\al)$ contains no geodesic ray of $T$. 
\end{theorem}

\begin{theorem}\label{thm-both does not con geo ray}
Suppose $Y\sse X$ are hyperbolic spaces satisfying the assumptions of Theorem \ref{main thm}. Suppose $\al$ is a geodesic line in $Y$ and that $\al$ is a leaf of the CT lamination $\L_{CT}(Y,X)$. Then there is a vertex $u\in V(S)$ and a geodesic line $\gm:\R\map Y_u$ such that $\pa i_{Y_u,Y}(\gm(\pm\infty))=\al(\pm\infty)$, and exactly one of the following holds.  
\begin{enumerate}
	\item $\gamma$ is a leaf of the CT lamination $\L_{CT}(Y_u,X_u)$.
	
	\item $\gamma$ is not a leaf of the CT lamination $\L_{CT}(Y_u,X_u)$, and
	$$(\pa i_{Y_u,X_u}(\gm(-\infty)),\pa i_{Y_u,X_u}(\gm(\infty)))\in\Lm^{\xi}(X_u,X)$$for some $\xi\in\pa T\setminus\pa S$.
\end{enumerate} 
\end{theorem}

The proofs of these theorems are given in Subsections \ref{subsubsec-thm-at least} and \ref{subsubsec-thm-both}. The main results used in these proofs are Propositions \ref{prop-pi alpha contains geo ray} and \ref{lam-in-fib}. The proofs of Proposition \ref{prop-pi alpha contains geo ray} rely on several auxiliary results, which are developed gradually in the next subsection.

\subsubsection{\bf Leaves of $\bm{\L_{CT}(Y,X_S)}$ towards the proof of Theorem \ref{thm-at least one contain ray}}
We note that a major part of the proof of both the theorems mentioned above consists of some results on $\L_{CT}(Y,X_S)$. 
Therefore, we discuss those first.

\begin{convention}
We note that for an edge $e$ of $S$ incident on a vertex $u$, the spaces $Y_{eu}$ and $X_{eu}$ are $\lambda_0$-quasiconvex in $Y_u$ and $X_u$, respectively. We fix $$R=\bm{ R_{fc}}(\lm_0)\text{ and } D=\bm{D_{fc}}(\lm_0),$$ see Definition \ref{defn-flow space} and Remark \ref{rmk-about flow constants} for the constants. Throughout this section, we consider the flow spaces $\F l^Y_{R,D}(Y_{eu})$ and $\F l^X_{R,D}(X_{eu})$ with these constants $R$ and $D$.
\end{convention}

To understand the leaves of $\L_{CT}(Y,X_S)$ and to prove results similar to Theorem \ref{thm-unbdd imp not equal} one needs to relate geodesics of $Y$ to the geodesics of $X_S$. The following result which appears in earlier section (see Corollary \ref{flow-qc}, Lemma \ref{Y flow qc in X}) is extremely useful for that. We restate this for repeated use in this section.

\begin{prop}\label{prop-flow qc in Y and X}{\em ({\bf Quasiconvexity of $Y$-flow in $X_S$})}
	Let $e$ be an edge in $S$ incident on a vertex $u$ of $S$. We have a uniform constant $K_{\ref{prop-flow qc in Y and X}}\ge0$ such that $\F l^Y_{R,D}(Y_{eu})$ is $K_{\ref{prop-flow qc in Y and X}}$-quasiconvex in $Y$ as well in $X_S$.
\end{prop}

Below we construct other quasiconvex subsets of $Y$ which are quasiconvex in $X_S$ also. For that first we need the following lemma.

\begin{comment}
\begin{lemma}[\bf Qi lifts in $\bm{Y}$ obtained from $\bm{X}$-flow spaces]\label{lem-X lift gives Y lift}
Let $u\in V(S)$, and let $e_1$ be an edge in $S$ incident on $u$. Suppose $\F l^X_{R,D}(X_{e_1u})$ is the flow space of $X_{e_1u}$ in $X$. Then given any vertex $w\in \pi(\F l^X_{R,D}(X_{e_1u}))\cap S$, there is a $K_{\ref{lem-X lift gives Y lift}}$-qi lift of $[u,w]$ in $Y$ where $K_{\ref{lem-X lift gives Y lift}}\ge1$ is a constant depending only on $R$ and other structural constants.

Moreover, if $\al$ is such a $K_{\ref{lem-X lift gives Y lift}}$-qi lift, we may assume that $\al(v)\in Y_{ev}$ where $e$ is the edge in $[u,v]\sse [u,w]$ incident on $v$.
\end{lemma}
\end{comment}

\begin{lemma}\label{lem-X lift gives Y lift}
Suppose $u,w\in V(S)$ and $\beta:[u,w]\map X$ is a $k$-qi lift for some $k\ge1$.
Suppose $u=u_0, u_1,\cdots, u_n=w$ are the consecutive vertices on $[u,w]$
and $e_i$ is the edge connecting $u_{i-1}, u_{i}$, $1\leq i\leq n$. Then the
map $\alpha:[u,w]\map Y$ given by $\alpha(u_i)=P^{X_{u_i}}_{Y_{e_i u_i}}(\bt(u_i))$ for $1\le i\le n$ and $\al(u_0)=P^{X_{u_0}}_{Y_{e_1u_0}}(\bt(u_0))$ defines a $K_{\ref{lem-X lift gives Y lift}}$-qi lift in $Y$ where $K_{\ref{lem-X lift gives Y lift}}\ge1$ depends on $k$
and the structural constants of the tree of spaces $X$ and $Y$.
\end{lemma}

We will need the following lemma in the proof. Since it only utilizes the geometry of vertex spaces and the corresponding edge spaces, it can be formulated in the setting of general hyperbolic metric spaces. We restrict ourselves to this formulation as required.

\begin{lemma}\label{lem-required}
Let $K\ge0$. Let $e_1,e_2$ be two edges of $S$ incident on a vertex $u$. Let $x\in X_u$ be such that $x\in N^{X_u}_K(X_{e_iu})$ and $x_i=P^{X_u}_{Y_{e_iu}}(x)$, $i=1,2$. Then $d_{X_u}(x_1,x_2)\le K'$ for some $K'$ depending on $K$ and the other structural constants of the tree of spaces $X$ and $Y$.
\end{lemma}

\begin{proof}
If $d_{X_u}(x_1,x_2)\le K+2\dl$, then we are done, where $\dl$ is the hyperbolicity constant of $X_u$. Otherwise, by Lemma \ref{lem-qg}, we choose points $y_i\in[x,x_i]_{X_u}$ such that $K+2\dl\le d_{X_u}(y_1,y_2)\le K+5\dl$ and the arc length parametrization of $\gm=[x_1,y_1]_{X_u}\cup[y_1,y_2]_{X_u}\cup[y_2,x_2]_{X_u}$ is $K_{\ref{lem-qg}}(\dl,K)$-quasigeodesic. 

First, we prove a uniform bound on $d_{X_u}(y_2,x_2)$. A similar argument then yields the same bound on $d_{X_u}(x_1,y_1)$, which in turn implies a uniform bound on $d_{X_u}(x_1,x_2)$.

Let $P^{X_u}_{X_{e_2u}}(x_1)=s$ and $P^{X_u}_{Y_{e_2u}}(x_1)=t$. Now by Lemma \ref{lem-proj on pair qc}~\ref{lem-proj on pair qc:concatenation qg}, $[x_1,s]_{X_u}\cup[s,x_2]_{X_u}$ is a $(3+2\lm_0)$-quasigeodesic. Let $K_1=max\{K_{\ref{lem-qg}}(\dl,K),3+2\lm_0\}$. By the stability of quasigeodesic (Lemma \ref{ml}), there is $s'\in\gm$ such that $d_{X_u}(s,s')\le D$ where $D=2D_{\ref{ml}}(\dl,K_1)$. By the compatible fiberwise projection condition, we have $$d_{X_u}(s,t)\le R_0.$$Hence, $d_{X_u}(t,s')\le d_{X_u}(t,s)+d_{X_u}(s,s')\le D+R_0$. {\em Now we consider two cases to prove a uniform bound on $d_{X_u}(y_2,x_2)$, depending on the position of the point $s'$.}\smallskip

\noindent {\em Case 1}: Suppose $s'\in[x_1,y_1]_{X_u}\cup[y_1,y_2]_{X_u}$. Since $Y_{e_2u}$ is $\lm_0$-quasiconvex in $X_u$ and $x_2,t\in Y_{e_2u}$, hence, by the stability of quasigeodesic, we have $d_{X_u}(y_2,z)\le D_{\ref{ml}}(\dl,\lm_0)+D+R_0+\dl+\lm_0=D_1$ (say) for some $z\in Y_{e_2u}$. Since $x_2$ is a nearest point projection of $x$ on $Y_{e_2u}$ in $X_u$ and $y_2\in[x,x_2]_{X_u}$, so $d_{X_u}(y_2,x_2)\le d_{X_u}(y_2,z)\le D_1$.\smallskip

\noindent{\em Case 2}: Suppose $s'\in[y_2,x_2]_{X_u}$. Since $x_2$ is a nearest point projection of $x$ on $Y_{e_2u}$ in $X_u$ and $s'\in[x,x_2]_{X_u}$, $d_{X_u}(s',x_2)\le d_{X_u}(s',t)\le D+R_0$. As $d_{X_u}(s,s')\le D$, by the stability of quasigeodesic, we have $d_{X_u}(y_2,s'')\le D_{\ref{ml}}(\dl,K_1)+D+\dl=D_2$ (say) for some $s''\in[x_1,s]_{X_u}$. On the other hand, as $X_{e_2u}$ is $\lm_0$-quasiconvex in $X_u$ and $d_{X_u}(x,X_{e_2u})\le K$, we have $d_{X_u}(y_2,y'_2)\le K+\dl+\lm_0$ for some $y'_2\in X_{e_2u}$. So, by triangle inequality, $d_{X_u}(s'',y'_2)\le D_2+K+\dl+\lm_0=D_3$ (say).

Since $s$ is a nearest point projection of $x_1$ on $X_{e_2u}$ and $s''\in[x_1,s]_{X_u}$, so $d_{X_u}(s'',s)\le d_{X_u}(s'',y'_2)\le D_3$. Hence, by triangle inequality, $d_{X_u}(y_2,s')\le d_{X_u}(y_2,s'')+d_{X_u}(s'',s)+d_{X_u}(s,s')\le D_3+D_2+D$. Again, by triangle inequality, $d_{X_u}(y_2,x_2)\le d_{X_u}(y_2,s')+d_{X_u}(s',x_2)\le D_3+D_2+D+D+R_0=D_4$ (say). This completes Case $2$.

Coming back to the proof of lemma, we have $d_{X_u}(y_2,x_2)\le max\{D_1,D_4\}=D_5$ (say). By the exact same argument, we can conclude $d_{X_u}(y_1,x_1)\le D_5$. Therefore, $d_{X_u}(x_1,x_2)\le d_{X_u}(x_1,y_1)+d_{X_u}(y_1,y_2)+d_{X_u}(y_2,x_2)\le 2D_5+K+5\dl$. This completes the proof of Lemma \ref{lem-required}.
\end{proof}

{\bf\em Proof of Lemma \ref{lem-X lift gives Y lift}.} Fix $1\le i\le n-1$. Since $\bt$ is a $k$-qi embedding and vertex spaces are $\eta$-proper embedding in $X$, we take $p_i\in X_{e_iu_i}$, $q_i\in X_{e_{i+1}u_i}$ and $p_{i+1}\in X_{e_{i+1}u_{i+1}}$ such that $d_{X_{u_i}}(\bt(u_i),p_i)\le K$, $d_{X_{u_i}}(\bt(u_i),q_i)\le K$, $d_{X_{u_iu_{i+1}}}(q_i,p_{i+1})=1$ and $d_{X_{u_{i+1}}}(p_{i+1},\bt(u_{i+1}))\le K$ where $K=\eta(2k)$ and $X_{u_iu_{i+1}}=\pi^{-1}([u_i,u_{i+1}])$. To complete the proof, it is enough to prove $d_X(s_i,s_{i+1})$ is uniformly bounded where  $P^{X_{u_i}}_{Y_{e_iu_i}}(\bt(u_i))=s_i=\al(u_i)$, $P^{X_{u_{i+1}}}_{Y_{e_{i+1}u_{i+1}}}(\bt(u_{i+1}))=s_{i+1}=\al(u_{i+1})$.

By Lemma \ref{lem-required}, $d_{X_{u_i}}(s_i,s'_i)$ is uniformly bounded where $s'_i=P^{X_{u_i}}_{Y_{e_{i+1}u_i}}(\bt(u_i))$. Since a nearest point projection map on a quasiconvex subset in hyperbolic metric space is coarsely Lipschitz (Lemma {\ref{proj-on-qc}~\ref{proj-on-qc:1}}), $d_{X_{u_i}}(s'_i,q'_i)$ and $d_{X_{u_{i+1}}}(p'_{i+1},s_{i+1})$ is uniformly bounded where $q'_i=P^{X_{u_i}}_{Y_{e_{i+1}u_i}}(q_i)$ and $p'_{i+1}=P^{X_{u_{i+1}}}_{Y_{e_{i+1}u_{i+1}}}(p_{i+1})$. Hence, we only need to show a uniform bound on $d_{X_{u_iu_{i+1}}}(q'_i,p'_{i+1})$. However, this can be seen easily by restricting the analysis to $X_{u_iu_{i+1}}$, as follows.

Note that $X_{u_iu_{i+1}}$ is $\dl'_0$-hyperbolic (see Lemma \ref{com-two-hyp-sps}) and $Hd_{X_{u_iu_{i+1}}}(Y_{e_{i+1}u_i},Y_{e_{i+1}u_{i+1}})=1$. Using Lemmas \ref{com-two-hyp-sps} and \ref{qc morph}, we may assume that $Y_{e_{i+1}u_i}$ and $Y_{e_{i+1}u_{i+1}}$ are $\lm_1$-quasiconvex in $X_{u_iu_{i+1}}$ for some uniform constant $\lm_1\ge0$. Since a nearest point projection map on a quasiconvex subset is coarsely Lipschitz (Lemma {\ref{proj-on-qc}~\ref{proj-on-qc:1}}), and thus by Lemma \ref{lem-proj on pair qc}~\ref{lem-proj on pair qc:new proj lemma3}, $d_{X_{u_iu_{i+1}}}(q''_i,p''_{i+1})$ is uniformly bounded where $q''_i=P^{X_{u_iu_{i+1}}}_{Y_{e_{i+1}u_i}}(q_i)$, $p''_{i+1}=P^{X_{u_iu_{i+1}}}_{Y_{e_{i+1}u_{i+1}}}(p_{i+1})$. Since $X_{u_i}$, $X_{u_{i+1}}$ are uniformly qi embedding in $X_{u_iu_{i+1}}$ (Lemma \ref{com-two-hyp-sps}), by Lemma \ref{proj-in-diff-are-close}, $d_{X_{u_iu_{i+1}}}(q'_i,q''_i)$ and $d_{X_{u_iu_{i+1}}}(p'_{i+1},p''_{i+1})$ are uniformly bounded. By triangle inequality, we conclude that $d_{X_{u_iu_{i+1}}}(q'_i,p'_{i+1})$ is uniformly bounded as required. This completes the proof of Lemma \ref{lem-X lift gives Y lift}.\qed\smallskip

As a consequence of Lemma \ref{lem-X lift gives Y lift}, we have the following.

\begin{cor}\label{cor-union flow over ray qc}{\em ({\bf Construction of sets quasiconvex in $\bm{Y}$ and $\bm{X_S}$})}
We have a uniform constant $K_{\ref{cor-union flow over ray qc}}\ge0$ satisfying the following. Suppose $[u,\xi)$ is a geodesic ray in $S$. Let $e_n=[u_n,u_{n+1}]$ be edges on $[u,\xi)$ such that $u_1=u$ and $d_S(u,u_{n+1})=n$. Assume $\{y'_n\}\sse Y$ and $y_n\in Y_{u_n}$ such that the nearest point projection of $\pi(y'_n)$ onto $[u,\xi)$ is $u$, and $\lim^{X}_{n\map\infty}y_n,~\LMX y'_n$ exist in $\pa X$ and they are equal. Then there is $i_0$ depending on the sequences and other structural constants such that $$\bigcup_{n\ge i_0}\F l^Y_{R,D}(Y_{e_nu_n})$$ is $K_{\ref{cor-union flow over ray qc}}$-quasiconvex in $Y$ as well in $X_S$.
\end{cor}

First, we prove the following general result, whose special case when $S'$ is a segment will be used in the proof of Corollary \ref{cor-union flow over ray qc}.
\begin{lemma}\label{lem-union flows qc}
We have a uniform constant $K_{\ref{lem-union flows qc}}\ge0$ satisfying the following. Let $e'$ be an edge of $S$ incident on a vertex $u$. Let $S'\sse \pi(\F l^X_{R,D}(X_{e'u}))\cap S$ be any subtree containing $u$. For an edge $e$ of $S$, let $t(e)$ denote the vertex on $e$ away from $u$. Then $\bigcup_{e\in E(S')}\F l^Y_{R,D}(Y_{e~t(e)})$ is 
$K_{\ref{lem-union flows qc}}$-quasiconvex in $Y$ as well in $X_S$.
\end{lemma}

\begin{proof}

We show that for any vertex $w$ of $S'$, the set $Z=\bigcup_{e\in E([u,w])}\F l^Y_{R,D}(Y_{e~t(e)})$ is uniformly quasiconvex in $Y$ as well in $X_S$. Then the result follows from Lemma \ref{lem-union qc qg} and Proposition \ref{prop-flow qc in Y and X}.

Now we have uniform constants $K_1,K_2\ge 1$ with the following. Fix $x\in\F l^X_{R,D}(X_{e'u})\cap X_w$. Lemma \ref{lem-lift in flow space} says that there is a $K_1$-qi lift in $X$ through $x$ over $[u,w]$ such that the image lies in $\F l^X_{R,D}(X_{e'u})$; and hence, by Lemma \ref{lem-X lift gives Y lift}, there is a $K_2$-qi lift, $\gm$ say, of $[u,w]$ in $Y$ such that $\gm(t(e))\in Y_{e~t(e)}$ for all edge $e$ of $[u,w]$. Note that $\gm$ is also a $K_2$-qi lift in $X_S$ (Lemma \ref{lem-X lift gives Y lift}), and thus by the stability of quasigeodesic $\gm$ is $K'$-quasiconvex in $Y$ as well in $X_S$ where $Y$ and $X_S$ are $\dl$-hyperbolic and $K'=D_{\ref{ml}}(\dl,K_{\ref{lem-X lift gives Y lift}})$. On the other hand, $\F l^Y_{R,D}(Y_{e~t(e)})$ is $K_{\ref{prop-flow qc in Y and X}}$-quasiconvex in $Y$ as well in $X_S$ by Proposition \ref{prop-flow qc in Y and X}. Hence by Lemma \ref{lem-union qc qg}, $Z$ is uniformly quasiconvex in $Y$ as well in $X_S$. This completes the proof.
\end{proof}

{\bf \em Proof of Corollary \ref{cor-union flow over ray qc}}:
By Lemma \ref{not-hv-pro-sub}, we have $d_X(\F l^X_{R,D}(X_{e_1u}),\F l^X_{R,D}(X_{e_nu_n}))\le D'$ for some uniform constant $D'$ and  for all $n\in\N$.

If $[u,\xi)\cap\pi(\F l^X_{R,D}(X_{e_1u}))=[u,u_i]$ for some $i$. Let $i_0=i+[D']$. Then from the above we have that for all large $n\ge i_0$, $X_{e_{i_0}u_{i_0}}\cap \F l^X_{R,D}(X_{e_nu_n})\ne\emptyset$. Otherwise, $[u,\xi)\sse\pi(\F l^X_{R,D}(X_{e_1u}))$ implies, in particular, $\F l^X_{R,D}(X_{e_{1}u})\cap X_{e_nu_n}\ne\emptyset$, and from the definition of flow space (Remark \ref{rmk-properties used} $(1)$), we have, in particular, $\F l^X_{R,D}(X_{e_{i_0},u_{i_0}})\cap X_{e_nu_n}\ne\emptyset$ for all $n\ge i_0$. For all $m\ge i_0$, let $$Z_m=\bigcup_{i_0\le n\le m}\F l^Y_{R,D}(Y_{e_nu_n}).$$Then by Lemma \ref{lem-union flows qc}, $Z_m$ is $K_{\ref{lem-union flows qc}}$-quasiconvex in $Y$ as well in $X_S$. Since $Z=\bigcup_{m\ge i_0}Z_m=\bigcup_{n\ge i_0}\F l^Y_{R,D}(Y_{e_nu_n})$ is an increasing union of uniform quasiconvex subsets in $Y$ as well in $X_S$, the result follows.\qed\smallskip

%{\em Claim:} $Z_m$ is uniformly quasiconvex in $Y$ and in $X_S$.

%{\em Proof of claim:} In either case, , we have a $K_1$-qi lift of $[u_{i_0},u_n]$ in $X$ (by Lemma ); and hence, by Lemma \ref{lem-X lift gives Y lift}, there is a $K_2$-qi lift, $\gm$ say, of $[u,w]$ in $Y$ 

We are now ready to harvest the main result of this subsection.

{\bf Notation.} For a geodesic line $\al$, we denote the restrictions of $\al$ on $[0,\infty)$ and $(-\infty, 0]$ by $\al_+$ and $\al_-$ respectively.

\begin{theorem}\label{thm-harvest}
Let $K_{\ref{thm-harvest}}=K_{\ref{lem-nonconical in both}}$. Suppose $\al$ is a geodesic line in $Y$ such that $\al$ is a CT leaf in $\L_{CT}(Y,X)$.

\begin{enumerate}
\item Assume that $\pi(\al_+)$ contains a geodesic ray, while $\pi(\al_-)$ does not. Then $\al|_{[r,\infty)}$ is a $K_{\ref{thm-harvest}}$-quasigeodesic ray in $X_S$ for some large $r\ge0$.

\item Assume that both $\pi(\al_{\pm})$ contain geodesic rays. Then $\al|_{(-\infty,-r]}$ and $\al|_{[r,\infty)}$ are $K_{\ref{thm-harvest}}$-quasigeodesic rays in $X_S$ for some large $r\ge0$.

In particular, $\al$ is a quasigeodesic line in $X_S$.
\end{enumerate}
\end{theorem}

We need the following result in the proof of Theorem \ref{thm-harvest}.

\begin{lemma}\label{lem-nonconical in both}
We have a uniform constant $K_{\ref{lem-nonconical in both}}\ge1$ satisfying the following. Let $\al$ be geodesic ray in $Y$, and let $\{y_n\}\sse Y$ be such that the nearest point projection of $\pi(y_n)$ onto $\pi(\al)$ is $\pi(\al(0))$. Assume that $\pi(\al)$ contains a geodesic ray, and $\LMX y_n$ exists in $\pa X$ and $\LMX y_n=\pa i_{Y,X}(\al(\infty))$. Then $\al|_{[r,\infty)}$ is a $K_{\ref{lem-nonconical in both}}$-quasigeodesic in $X_S$ for some large $r\ge0$.
	
In particular, $\al$ is a quasigeodesic ray in $X_S$.
\end{lemma}
\begin{proof}
	%Since CT maps are functorial (see Lemma \ref{functo-ct-map}), we have $\LMX\al'(n)=\LMX\al(n)$ by the assumption. 
	Suppose $\xi\in\pa\pi(\al)$, and so $[u,\xi)\sse\pi(\al)$ where $\al(0)=u$. Let $e_n=[u_n,u_{n+1}]$ be edges on $[u,\xi)$ such that $u_1=u$ and $d_S(u,u_{n+1})=n$. Suppose $\{r_n\}\sse\R_{\ge0}$ is an unbounded sequence such that $\al(r_n)\in Y_{e_nu_n}$. Note that $\LMX y_n=\pa i_{Y,X}(\al(\infty))=\LMX\al(r_n)$. Then by Corollary \ref{cor-union flow over ray qc}, $Z=\bigcup_{n\ge i_0}\F l^Y_{R,D}(Y_{e_nu_n})$ is $K_{\ref{cor-union flow over ray qc}}$-quasiconvex in $Y$ as well in $X_S$ for some $i_0\in\N$. %Note that $K'_{\ref{cor-union flow over ray qc}}$ depends on the sequence $\al(r_n)$ and the other structural constants.

Assume that $Y$, $X_S$ are $\dl$-hyperbolic, the inclusion $Y\map X_S$ is $\eta$-proper embedding. Thus by Lemma \ref{lem-same limits}, $W=N^Y_{K_{\ref{cor-union flow over ray qc}}+1}(Z)$ is $L$-qi embedding in $Y$ as well in $X_S$ where $L=L_{\ref{lem-same limits}}(\dl,K_{\ref{cor-union flow over ray qc}},\eta)$. For all $n\ge i_0$, if $\gm_n$ is a geodesic in $W$ joining $\al(r_{i_0})$ and $\al(r_n)$ then $\gm_n$ is an $L$-quasigeodesic in $X_S$, and hence, by the stability of quasigeodesic in $X_S$, we have $Hd_{X_S}(\gm_n,\al|_{[r_{i_0},r_n]})\le D_1$ for some constant $D_1$ depending on $\dl$ (hyperbolicity constant of $X_S$) and $L$. Note that $\al$ is $1$-Lipschitz in $X_S$ and an $\eta$-proper as the inclusion $Y\map X_S$ is $\eta$-proper. Thus by Lemma \ref{lem-giving qg}, $\al|_{[r_{i_0},r_n]}$ is a $K$-quasigeodesic in $X_S$ for some constant $K\ge1$ depending on $L$, $D_1$ and $\eta$. This is true for all $n\ge i_0$. Therefore, $\al|_{[r_{i_0},\infty)}$ is $K$-quasigeodesic in $X_S$. This completes the proof.
\end{proof}

{\bf \em Proof of Theorem \ref{thm-harvest}}:

$(1)$ Since $\pi(\al_-)$ does not contain a geodesic ray, we have a vertex $v$ of $S$ and an unbounded sequence $\{s_n\}\sse \R$ such that $\al(-s_n)\in Y_v$ (see Lemma \ref{lem-not geo ray imp seq in vertex}). By given condition we have $\LMX\al(-s_n)=\pa i_{Y,X}(\al(-\infty))=\pa i_{Y,X}(\al(\infty))$. Then $(1)$ follows from Lemma \ref{lem-nonconical in both}.

$(2)$ Suppose $\al(0)=u$, $\xi'\in\pa\pi(\al_-)$ and $\xi\in\pa\pi(\al_+)$. Let $e'_n=[u'_n,u'_{n+1}]$ and $e_n=[u_n,u_{n+1}]$ be successive edges on $[u,\xi')$ and $[u,\xi)$ respectively, where $u_1=u'_1=u$. Let $\{r_n\},\{s_n\}\sse\R_{\ge0}$ be unbounded sequences such that $\al(-s_n)\in Y_{e'_nu'_n},~\al(r_n)\in Y_{e_nu_n}$. Now we consider two cases depending on whether $\xi'=\xi$. Note that in Case $1$, we do not use the full hypothesis.

\noindent{\bf Case 1}: Suppose $\xi'=\xi$. Note that by Proposition \ref{prop-flow qc in Y and X}, $Z_n=\F l^Y_{R,D}(Y_{e_nu_n})$ is $K_{\ref{prop-flow qc in Y and X}}$-quasiconvex in $Y$ as well in $X_S$, and $\al(-s_n),\al(r_n)\in Y_{e_nu_n}\sse Z_n$. As explained in the proof of Lemma \ref{lem-nonconical in both}, it follows that $\al$ is $K_{\ref{lem-nonconical in both}}$-quasigeodesic in $X_S$.

\noindent{\bf Case 2}: Suppose $\xi'\ne\xi$. Let $v$ be the nearest point projection of $u$ on the geodesic line joining $\xi'$ and $\xi$. Let $d_T(u,v)=i$. Consider the sequences $\{\al(-s_n)\}_{n\ge i}$, $\{\al(r_n)\}_{n\ge i}$, and note that $\LMX\al(-s_n)=\pa i_{Y,X}(\al(-\infty))=\pa i_{Y,X}(\al(\infty))=\LMX\al(r_n)$. We apply Lemma \ref{lem-nonconical in both} twice to conclude the result.

This completes the proof of $(2)$.\qed\smallskip

We are now in a position to prove main ingredients used in the proof of the main theorems of this section. As a consequence of Theorem \ref{thm-harvest} we will prove the following. 

\begin{prop}\label{prop-pi alpha contains geo ray}
Suppose $\al$ is a geodesic line in $Y$ such that $\pi(\al)$ contains a geodesic ray. Then $\al$ is not a leaf of the CT lamination $\L_{CT}(Y,X_S)$. %$$\pa i_{Y,X_S}(\al(-\infty))\ne\pa i_{Y,X_S}(\al(\infty)).$$
\end{prop}

\begin{proof}
On contrary, suppose that $\al$ is a leaf of the CT lamination $\L_{CT}(Y,X_S)$. Hence $\al$ is a leaf of the CT lamination $\L_{CT}(Y,X)$ as $\pa i_{Y,X}=\pa i_{X_S,X}\circ\pa i_{Y,X_S}$ (Lemma \ref{functo-ct-map}). We consider the following two cases. %Let $\al(0)=u$.

\noindent{\bf Case 1}: Suppose both $\pi(\al_{\pm})$ contain geodesic rays. Then by Theorem \ref{thm-harvest} $(2)$, $\al$ is a quasigeodesic line in $X_S$ -- which gives a contradiction to our assumption that $\al$ is a CT leaf in $\L_{CT}(Y,X_S)$. Hence the result follows in this case.

\noindent{\bf Case 2}: Without loss of generality, we assume that $\pi(\al_+)$ contains a geodesic ray, while $\pi(\al_-)$ does not. {\em Now we restrict our study on the subtree of spaces over $S$ only, i.e., on the inclusion $Y\map X_S$}. The idea is to construct a sequence $\{z_n\}\sse Y$ that creates a bridge showing that $\al(-\infty)=\al(\infty)$ in $Y$, contradicting our hypothesis.

Let $\xi\in\pa\pi(\al_+)$ and $[u,\xi)\sse \pi(\al_+)$. Since $\pi(\al_-)$ does not contain a geodesic ray, there is a vertex $u$ of $S$ and a sequence $\{y_n\}\sse Y_u$ such that $\lim^{Y_u}_{n\map\infty}y_n$ exists in $\pa Y_u$ and $\LMY y_n=\al(-\infty)$ (see Lemma \ref{lem-not geo ray imp seq in vertex}). Then by our main theorem (Theorem \ref{main thm}), we have $\lim^{X_S}_{n\map\infty} y_n=\pa i_{Y,X_S}(\al(\infty))$. 

Let $\bt$ be a geodesic ray in $Y_u$ such that $\lim^{Y_u}_{n\map\infty}y_n=\bt(\infty)$, and let $\gm$ be a geodesic ray in $X_u$ such that $\gm(\infty)=\pa i_{Y_u,X_u}(\bt(\infty))$. Note that $\pa i_{Y,X_S}(\al(\infty))=\lim^{X_S}_{n\map\infty} \gm(n)=\pa i_{X_u,X_S}(\gm(\infty))$ as the inclusions $Y_u\map X_u\map X_S$ and $Y_u\map Y\map X_S$ admit CT maps and that the CT map is functorial property (Lemma \ref{functo-ct-map}). Again, by 
Theorem \ref{thm-harvest} $(1)$, $\al|_{[r,\infty)}$ is $K_{\ref{thm-harvest}}$-quasigeodesic ray in $X_S$ for some large $r\ge0$. Since $\pi(\al_+)$ contains a geodesic ray then by Lemma \ref{lem-equ-im-bdry-flow}, $\gm(\infty)$ has boundary flow in $X_v$ for all vertex $v$ of $[u,\xi)$. In particular, $[u,\xi)\sse\pi(\F l^{X_S}_{R,D}(X_u))$; note that $\F l^{X_S}_{R,D}(X_u)$ is the restriction of $\F l^X_{R,D}(X_u)$ in $X_S$. Let $e_n=[u_n,u_{n+1}]$ be consecutive edges on $[u,\xi)$ such that $u_1=u$.
Then by Lemma \ref{lem-union flows qc}, $Z=\bigcup_{n\in\N}\F l^Y_{R,D}(Y_{e_nu_n})$ is $K_{\ref{lem-union flows qc}}$-quasiconvex in $Y$ as well in $X_S$.

Now we apply Lemma \ref{qua-im-qua} to $Y_u\sse X_u$ and get $Hd_{X_u}(\bt,\gm)<\infty$. On the other hand, $\bt\sse N^{Y_u}_{D'}(Y_{e_1u})$ for some $D'\ge0$. Since $Y_u\map Y$ admits the CT map, thus we get a sequence $\{z_n\}\sse Y_{e_1u}$ such that $d_{Y_u}(\bt(n),z_n)\le D'$, $\lim^{Y_u}_{n\map\infty}z_n$ exists in $\pa Y_u$ and $\LMY z_n=\al(-\infty)$. Since the inclusions $Y_u\map X_u\map X_S$ admit CT maps, we have $\lim^{X_S}_{n\map\infty}\bt(n)=\lim^{X_S}_{n\map\infty}\gm(n)=\pa i_{Y,X_S}(\al(\infty))$, and thus we get $\lim^{X_S}_{n\map\infty} z_n=\pa i_{Y,X_S}(\al(\infty))$.

Let $\{r_n\}\sse\R_{\ge0}$ be an unbounded sequence such that $\al(r_n)\in Y_{e_nu_n}$. Note that $\lim^{X_S}_{n\map\infty} z_n=\pa i_{Y,X_S}(\al(\infty))=\lim^{X_S}_{n\map\infty}\al(r_n)$ as $Y\map X_S$ admits the CT map and that the CT map is functorial (Lemma \ref{functo-ct-map}). Further, we have $z_n,\al(r_n)\in Z$ for all $n\in\N$. Hence by Lemma \ref{lem-same limits}, $\LMY z_n=\LMY\al(r_n)$. Therefore, $\al(-\infty)=\LMY z_n=\LMY\al(r_n)=\al(\infty)$ $-$ which is a contradiction to the assumption that $\al$ is a geodesic line in $Y$. This completes the proof.
\end{proof}

\subsubsection{\bf Proof of Theorem \ref{thm-at least one contain ray}}\label{subsubsec-thm-at least}
\begin{proof}
We prove this by contradiction. Suppose that $\pi(\alpha)$ contains a geodesic ray. Hence by Proposition \ref{prop-pi alpha contains geo ray}, we have $\partial i_{Y,X_S}(\alpha(-\infty))\ne \partial i_{Y,X_S}(\alpha(\infty))$. 

Again, by Theorem \ref{thm-harvest}, after reparametrization, if necessary, we assume that $\al|_{[r,\infty)}$ is a quasigeodesic in $X_S$ and $\pi(\al|_{[r,\infty)})$ contains a geodesic ray. Let $\bt$ be a geodesic line joining $\partial i_{Y,X_S}(\alpha(-\infty))$ and $\partial i_{Y,X_S}(\alpha(\infty))$  in $X_S$. Note that $\pi(\bt)$ (contains a geodesic ray, hence) is unbounded. Now as $\al$ is a CT leaf of $\L_{CT}(Y,X)$ and $\pa i_{Y,X}=\pa i_{X_S,X}\circ\pa i_{Y,X_S}$ (Lemma \ref{functo-ct-map}), $\pa i_{X_S,X}(\bt(-\infty))=\pa i_{X_S,X}(\bt(\infty))$. This contradicts to Theorem \ref{thm-unbdd imp not equal}.
\end{proof}	

\subsubsection{\bf Proof of Theorem \ref{thm-both does not con geo ray}}\label{subsubsec-thm-both}
We need the following result in the proof of Theorem \ref{thm-both does not con geo ray}. We mention that the proof of the following proposition does not require the study of quasiconvexity of flow spaces carried out above. It only requires the description of quasigeodesics in the ambient space corresponding to a CT leaf given in Theorem \ref{thm-quasigeo-descrip} and Proposition \ref{prop-con of geo des}.

\begin{prop}[\bf Relation between  $\bm{\L_{CT}(Y_u,Y)}$ and $\bm{\L_{CT}(X_u,X)}$]\label{lam-in-fib}
	Let $u\in V(S)$. Suppose $\al:\R\map Y_u$ is a geodesic line in $Y_u$ such that $\al$ is not a leaf of the CT lamination $\L_{CT}(Y_u,Y)$ and $\L_{CT}(Y_u,X_u)$. Assume that $(\pa i_{Y_u,Y}(\al(-\infty)),\pa i_{Y_u,Y}(\al(\infty)))\in\L_{CT}(Y,X)$. Then $$(\pa i_{Y_u,X_u}(\al(-\infty)),\pa i_{Y_u,X_u}(\al(\infty)))\in\Lm^{\xi}(X_u,X)$$for some $\xi\in\pa T\setminus\pa S.$
\end{prop}

\begin{proof}
	%Suppose $\pa i_{Y_u,X_u}(\gm(-\infty))\ne\pa i_{Y_u,X_u}(\gm(\infty))$. 
	Let $\bt:\R \ri X_u$ be a geodesic line in $X_u$ such that $\bt(\pm\infty)=\pa i_{Y_u,X_u}(\al(\pm\infty))$ respectively. Since the CT map is a functorial property (Lemma \ref{functo-ct-map}), i.e. $\pa i_{X_u,X}\circ\pa i_{Y_u,X_u}=\pa i_{Y_u,X}=\pa i_{Y,X}\circ\pa i_{Y_u,Y}$, it follows from the given condition that $\bt$ is a leaf of the CT lamination $\L_{CT}(X_u,X)$. 
	%$\pa i_{X_u,X}(\bt(-\infty))=\pa i_{X_u,X}(\bt(\infty))$, b
	Hence by Theorem \ref{thm-quasigeo-descrip}, we have $$(\pa i_{Y_u,X_u}(\al(-\infty)),\pa i_{Y_u,X_u}(\al(\infty)))=(\bt(-\infty),\bt(\infty))\in\Lm^{\xi}(X_u,X)$$for a unique $\xi\in\pa T$.\smallskip
	
	\noindent	{\em Claim}: $\xi\in\pa T\setminus\pa S$.\smallskip
	
	\noindent	{\em Proof of the claim}: On contrary, suppose that $\xi\in\pa S$. Note that by Theorem \ref{thm-both have bdry flow}, both of $\bt(\pm\infty)$ have boundary flows in $X_v$ for all vertex $v$ in $[u,\xi)$. Now by repeated application of Lemma \ref{cor-line-imp-line}, we conclude that both of $\al(\pm\infty)$ have boundary flows in $Y_v$ for all vertex $v$ in $[u,\xi)$. Moreover, suppose $\al_v$ and $\bt_v$ are geodesic lines in $Y_v$ and $X_v$ respectively such that $\al_v(\pm\infty)$ are boundary flows of $\al(\pm\infty)$ in $Y_v$ and $\bt_v(\pm\infty)$ are boundary flows of $\bt(\pm\infty)$ in $X_v$. Then $Hd_{X_v}(\al_v,\bt_v)\le D'$ for some uniform constant $D'\ge0$.
	
	Note that Theorem \ref{thm-quasigeo-descrip} provides a description of a uniform quasigeodesic joining $\bt(-m)$ and $\bt(n)$, where $m,n\in\N$. 
	The idea is to use this quasigeodesic to construct a uniform quasigeodesic joining $\al(-m')$ and $\al(n')$ of a similar form. 
	By Proposition \ref{prop-con of geo des}, this leads to a contradiction, since $\alpha$ is not a CT leaf of $\L_{CT}(Y_u,Y)$.

	Note that vertex spaces are $\eta_0$-proper embedding in the total space for both $Y$ and $X$. For the proper map $\eta_0:\R_{\ge0}\map\R_{\ge0}$, let $\eta'_0:\R_{\ge0}\map\R_{\ge0}$ be an associated proper map depending on $\eta$, as in Lemma \ref{cor-inverse of proper map}. We will denote the threshold constants appearing in Theorem \ref{thm-quasigeo-descrip} by $M^Y$ and $M^X$ respectively to distinguish between $Y$ and $X$. Fix constants $k\ge1$ and $R\ge M^X$ where $M^X$ depends on $k$ such that $$K=max\{\eta_0(k+2D'),1\}\text{ and }\eta'_0(R-2D')>M^Y$$ where $M^Y$ depends on $K$ and the constant $D'$ is defined above. 
	
	Then we have a fixed constant $C\ge0$ depending on $R$ with the following. For all large $m,n\in\N$, there exist $k$-qi lifts $\gm_{-m}$ and $\gm_n$ (through $\bt(-m)$ and $\bt(n)$ respectively) of $[u,w]$ in $X$ such that a uniform quasigeodesic in $X$ joining $\bt(-m)$ and $\bt(n)$ has the form$$\gm_{-m}*[\gm_{-m}(w),\gm_n(w)]_{X_w}*\gm_n$$where $w\in[u,\xi)$, with $$d_{X_v}(\gm_{-m}(v),\gm_n(v))\ge R$$ for all vertex $v$ in $[u,w]\setminus\{w\}$, and $$d_{X_w}(\gm_{-m}(w),\gm_n(w))\le C.$$We may assume that $\gm_{-m}(v),\gm_n(v)\in\bt_w$ for all vertex $v$ in $[u,w]$ (see Remark \ref{rmk-lifts in geo flow}). 
	
	Now, from $\gm_{-m}$ and $\gm_n$, we construct qi lifts $\gm'_{-m}$ and $\gm'_n$ of $[u,w]$ in $Y$, respectively, such that $\gm'_{-m}(v),\gm'_n(v)\in\al_v$ for all vertex $v$ in $[u,w]$. For a vertex $v$ in $[u,w]$, choose $x_v,y_v\in\al_v$ such that $d_{X_v}(x_v,\gm_{-m}(v))\le D'$ and $d_{X_v}(y_v,\gm_n(v))\le D'$. Now define $\gm'_{-m}(v):=x_v$ and $\gm'_n(v):=y_v$. It is easy to see that $\gm'_{-m}$ and $\gm'_n$ are $k'$-qi lifts of $[u,w]$ in $X$ where $k'=k+2D'$. Note that the inclusion $Y\map X$ is $\eta_0$-proper embedding. Hence $\gm'_{-m}$ and $\gm_n$ are $K$-qi lifts in $Y$ where $K=max\{\eta_0(k+2D'),1\}$. 
	
	Again, since $d_{X_w}(\gm_{-m}(w),\gm_n(w))\le C$, we have  $d_{X_w}(\gm'_{-m}(w),\gm'_n(w))\le C+2D'$, and hence, $d_{Y_w}(\gm'_{-m}(w),\gm'_n(w))\le \eta_0\circ\eta_0(C+2D')$. 
	
	Finally, for all vertex $v$ in $[u,w]\setminus\{w\}$, we have $d_{X_v}(\gm'_{-m}(v),\gm'_n(v))>R-2D'$ as $d_{X_v}(\gm_{-m}(v),\gm_n(v))>R$. Again $Y_v\map X_v$ is $\eta_0$-proper embedding, so $d_{Y_v}(\gm'_{-m}(v),\gm_n(v))\ge\eta'_0(R-2D')\ge M^Y$ (Lemma \ref{cor-inverse of proper map}). 
	
	Therefore, by Proposition \ref{prop-con of geo des}, we conclude that $\al$ is a leaf of the CT lamination $\L_{CT}(Y_u,Y)$. This is a contradiction to our assumption. This completes the proof.
\end{proof}

{\bf \em Proof of Theorem  \ref{thm-both does not con geo ray}}:
Since $\al$ is a leaf of the CT lamination $\L_{CT}(Y,X)$, by Theorem \ref{thm-at least one contain ray}, $\pi(\al)$ does not contain a geodesic ray. In other words, both of $\pi(\al_{\pm})$ do not contain geodesic rays. Thus we have vertices $v_1,v_2\in V(S)$ and $\{y_{n,j}\}\sse Y_{v_j}$ %$\{z_n\}\sse Y_{v_2}$ 
	such that $\lim^{Y_{v_j}}_{n\map\infty}y_{n,j}=\xi_j\in\pa Y_{v_j}$ where $j=1,2$, %$\lim^{Y_v}_{n\map\infty}y_n=\xi\in\pa Y_w$. 
	and $\LMY y_{n,1}=\al(-\infty)$ and $\LMY y_{n,2}=\al(\infty)$ (see Lemma \ref{lem-not geo ray imp seq in vertex}). Let $\al_j:[0,\infty)\map Y_{v_j}$ %and $\al_2:[0,\infty)\map Y_w$ 
	be a geodesic ray in $Y_{v_j}$ %and $Y_w$, respectively 
	such that $\al_j(\infty)=\xi_j$ where $j=1,2$. %and $\al_2(\infty)=\xi$.
	
	Let $\eta_j=\pa i_{Y_{v_j},X_{v_j}}(\xi_j)$ where $j=1,2$. %and $\eta=\pa i_{Y_w,X_w}(\xi)$. 
	Let $\bt_j:[0,\infty)\map X_{v_j}$ %and $\bt:[0,\infty)\map X_w$ 
	be a geodesic ray in $X_{v_j}$ %and $X_w$ respectively 
	such that $\bt_j(\infty)=\eta_j$ where $j=1,2$. %and $\bt(\infty)=\eta$. 
	Note that $\pa i_{Y_{v_1},Y}(\xi_1)=\al(-\infty)$ and $\pa i_{Y_{v_2},Y}(\xi_2)=\al(\infty)$. Thus $\pa i_{X_{v_1},X}(\eta_1)=\pa i_{X_{v_2},X}(\eta_2)$ in $\pa X$ as $\al$ is a CT leaf in $\L_{CT}(Y,X)$, the inclusions $Y_{v_j}\map X_{v_j}\map X$ admit CT maps and that the CT map is a functorial property (Lemma \ref{functo-ct-map}). Now we consider the following cases.\smallskip
	
\noindent	{\bf Case 1}: Suppose $v_1=v_2=u$ (say). Let $\gm:\R\map Y_u$ be a geodesic line in $Y_u$ such that $\gm(-\infty)=\xi_1$ and $\gm(\infty)=\xi_2$. Then $\pa i_{Y_u,Y}(\gm(\pm\infty))=\al(\pm\infty)$.
	
If $\pa i_{Y_u,X_u}(\xi_1)=\pa i_{Y_u,X_u}(\xi_2)$, then $(1)$ holds.
	
Now suppose that $\pa i_{Y_u,X_u}(\xi_1)\ne\pa i_{Y_u,X_u}(\xi_2)$, i.e., $\eta_1\ne\eta_2$. Again we have $\pa i_{X_u,X}(\eta_1)=\pa i_{X_u,X}(\eta_2)$. In this case, result follows from Proposition \ref{lam-in-fib}.\smallskip

\noindent	{\bf Case 2}: Suppose $v_1\ne v_2$. Since $\pa i_{X_{v_1},X}(\eta_1)=\pa i_{X_{v_2},X}(\eta_2)$, by Proposition \ref{bdry flow prop}, there is a vertex $u$ in $[v_1,v_2]$ such that both $\eta_1$ and $\eta_2$ have boundary flows in $X_u$. We denote the boundary flow of $\eta_j$ in $\pa X_u$ by $\eta'_j$ where $j=1,2$.  Let $\bt'_j$ be a geodesic ray in $X_u$ such that $\bt'_j(\infty)=\eta'_j$. 

Applying Lemma \ref{qua-im-qua}, we will show that $\xi_j$ has boundary flow in $Y_u$ and that coincide with $\eta'_j$ in $X_u$. Then we apply Proposition \ref{lam-in-fib} to conclude the result. %and $\eta_u$ respectively.

Let $e$ be an edge in $[v_1,v_2]$ incident on, say $v_2$. Since $\eta_2$ has boundary flow, we note that $\lim^{X_{v_2}}_{n\map\infty}\al_2(n)=\bt_2(\infty)\in\Lm_{X_{v_2}}(X_{ev_2})$.  Then by repeated application of Lemma \ref{qua-im-qua}, we note that $\xi_2$ has boundary flow in $Y_u$. Similarly, $\xi_1$ has boundary flow in $Y_u$. We denote their boundary flows in $Y_u$ by $\xi'_2$ and $\xi'_1$ respectively. Let $\al'_j:[0,\infty)\map Y_u$ be a geodesic ray in $Y_u$ such that $\al'_j(\infty)=\xi'_j$ where $j=1,2$. %and $\al_1(\infty)=\xi_u$.

Now by Proposition \ref{bdry flow lemma} $(3)$, $Hd_X(\bt_j,\bt'_j)<\infty$, and $Hd_Y(\al_j,\al'_j)<\infty$ where $j=1,2$. Let $e_j$ be the edge in $[u,v_j]$ incident on $u$ where $j=1,2$. Then by repeated application of Lemma \ref{qua-im-qua}, we have $\xi'_j\in\Lambda_{Y_u}(Y_{e_ju})$ and $\eta'_j\in\Lambda_{X_u}(X_{e_ju})$ where $j=1,2$.\smallskip

\noindent	{\em Claim}: $\eta'_1\ne\eta'_2$. We will prove the claim by considering two cases. \smallskip

\noindent	{\em Case $2A$}: $u\in\{v_1,v_2\}$. Without loss of generality, we assume that $v_1=u$. We have $Hd_X(\bt_2,\bt'_2)<\infty$, and $Hd_Y(\al_2,\al'_2)<\infty$. Since by Lemma \ref{qua-im-qua}, $Hd_X(\al_2,\bt_2)<\infty$, and thus we have $Hd_X(\al'_2,\bt'_2)<\infty$.

Again $\al'_1(\infty)\ne\al'_2(\infty)$ in $\pa Y_u$ as $\pa i_{Y_{v_1},Y}(\al_1(\infty))\ne\pa i_{Y_{v_2},Y}(\al_2(\infty))$. Note that $\al'_2(\infty)\in\Lambda_{Y_u}(Y_{e_2u})$ and $\bt'_2(\infty)\in\Lambda_{X_u}(X_{e_2u})$. If $\al'_1(\infty)\notin\Lambda_{X_u}(X_{e_2u})$ then $\bt'_1(\infty)\ne\bt'_2(\infty)$. Otherwise, suppose that $\al'_1(\infty)\in\Lambda_{X_u}(X_{e_2u})$ and $\bt'_1(\infty)=\bt'_2(\infty)$, i.e., $Hd_X(\bt'_1,\bt'_2)<\infty$. Then by applying Lemma \ref{qua-im-qua} in $X_u$, we have $Hd_X(\al'_1,\bt'_1)<\infty$, and thus $Hd_X(\al'_1,\al'_2)<\infty$. Hence $Hd_{Y_u}(\al'_1,\al'_2)<\infty$ (as the inclusions $Y_u\map Y\map X$ are proper embeddings). This is not possible as $\al'_1(\infty)\ne\al'_2(\infty)$ in $\pa Y_u$. Therefore, in either case, we have $\bt'_1(\infty)\ne\bt'_2(\infty)$, i.e., $\eta'_1\ne\eta'_2$.\smallskip

\noindent	{\em Case $2B$}: $u\notin\{v_1,v_2\}$. By a similar argument as in Case $2A$, one may conclude that $\eta'_1\ne\eta'_2$.
	
Therefore, we have $(\xi'_1,\xi'_2)\notin\L_{CT}(Y_u,Y)$, $(\eta'_1,\eta'_2)\notin\L_{CT}(Y_u,X_u)$ and $\pa i_{X_u,X}(\eta'_1)=\pa i_{X_{v_1},X}(\eta_1)=\pa i_{X_{v_2},X}(\eta_2)=\pa i_{X_u,X}(\eta'_2)$. Hence, the result follows from Proposition \ref{lam-in-fib}.\qed

\section{Other applications, examples and related results}\label{sec-application}

In this section we discuss a number of applications of Theorem \ref{main thm} for graphs of hyperbolic groups.
For details on graphs of groups and the Bass--Serre theory one is referred to \cite{serre-trees}. Graphs of hyperbolic
groups was first considered by Bestvina and Feighn (see \cite{BF}) where trees of spaces were
also introduced. We first briefly recall some of these. Note that all graphs of groups considered
here are defined over finite graphs.

\begin{defn}
Suppose $\Y$ is a finite directed graph. A {\em graph of groups}  $(\G,\Y)$ over $\Y$
consists of the following data:
\begin{itemize}
\item For each vertex $v$ of $\Y$ there is a group $G_v$. (Such groups are called {\em vertex groups} of $(\G, \Y)$.)

\item For each undirected edge $e$ of $\Y$ there is a group $G_e$. (Such groups are referred to as {\em edge groups}.)

\item For any directed edge $e$ with end point $v$ there is an injective group homomorphism $G_e\map G_v$.
(These homomorphisms will be referred to as {\em incidence homomorphisms}.)
\end{itemize}
A graph of groups $(\G, \Y)$ is called a {\em graph of hyperbolic groups with qi embedded condition},
after \cite{BF}, if all the vertex groups and the edge groups are hyperbolic and the incidence homomorphisms are qi embeddings.
\end{defn}
Next we recall the definition of the fundamental group of a graph of (hyperbolic) groups and how one may obtain a
tree of spaces from a graph of groups.

\smallskip
\noindent{\bf Construction of trees of spaces from graphs of groups.}\\
Suppose we have a graph of hyperbolic groups $(\G, \Y)$ with qi embedded condition as above
with a choice of an orientation on $\Y$.
It is a standard fact that hyperbolic groups are finitely presented and thus one may choose, for each
vertex group $G_v$, a finite simplicial complex $Z_v$ and for each edge group $G_e$, a finite simplicial
complex $Z_e$ (such that both $Z_v$ and $Z_e$ have single $0$-cell) along with simplicial maps $\phi_{e,v}: Z_e\map Z_v$ which induces the given incidence homomorphism
$G_e\map G_v$ at the level of fundamental groups. Then from these one constructs a
{\em graph of space} (see \cite{scott-wall}, \cite{BF}) $g: Z\map \Y$ where
$$Z=\bigsqcup_{v\in V(\Y)} Z_v \sqcup\bigsqcup_{e\in E(\Y)} Z_e\times [0,1]/\sim $$
where one glues $Z_e\times \{0\}$ to $Z_u$ and $Z_e\times \{1\}$ to $Z_v$, if $e$ joins $u$ and $v$,
using the maps $\phi_{e,u}$ and $\phi_{e,v}$ respectively. The map $g:Z\map \Y$ is obtained by collapsing
$Z_e\times \{t\}$, $t\in (0,1)$ and $Z_v$'s to points. The fundamental group of $Z$ is called the {\bf fundamental
group of $(\G, \Y)$} which is known to be independent of the choices of spaces $Z_v$'s and $Z_e$'s
and the maps $\phi_{e,v}$'s. It is denoted by $\pi_1(\G,\Y)$.

Let $p:\tilde{Z}\map Z$ be the universal cover. If one collapses the
connected components of the inverse images of points of $\Y$ under $f=g\circ p:\tilde{Z}\map \Y$, say,
then obtains the {\em Bass--Serre tree} $\T$ of $(\G, \Y)$ such that the quotient map $\pi: \tilde{Z}\map \T$
is a $\pi_1(\G,\Y)$-equivariant map and is a tree of hyperbolic spaces.
It is easy to verify that the vertex spaces are uniformly quasiisometric to the vertex groups
and the edge spaces are uniformly quasiisometric to the edge groups and that the qi embedded condition
is satisfied. For {\em graph model} of this tree of spaces one is referred to \cite[Section 3]{ps-conical}.

\smallskip
\noindent{\bf Subgraphs  of subgroups of graphs of groups.}\\
There is a more general notion of morphisms of graphs of groups to which our result can be applied. But for the
ease of the exposition we shall discuss only subgraphs of subgroups of graphs of groups.

\begin{defn}\label{subgraph of gps}
Suppose $(\G, \Y)$ is a graph of hyperbolic groups. Then a subgraph of subgroups of
$(\G, \Y)$ consists of the following:
\begin{enumerate}
\item A connected subgraph $\Y'$ of $\Y$.
\item For each vertex $v$ of $\Y'$ there is subgroup $G'_v$ of $G_v$
and for each unoriented edge $e$ of $\Y'$ a subgroup $G'_e$ of $G_e$
such that restrictions of the incidence homomorphisms of $(\G, \Y)$
make this collection into a graph of groups $(\G',\Y')$.
\item The induced map at the level of Bass--Serre trees $\T'\map \T$ is injective.
\end{enumerate}
\end{defn}
An explanation of the last condition in Definition \ref{subgraph of gps} is in order.
Given $(\G, \Y)$ and $(\G', \Y')$ we may construct simplicial complexes as above along with
inclusion maps $Z'_v\map Z_v$ and $Z'_e\map Z_e$ for all vertices and edges of $\Y'$. (Note that the prime $'$ in superscript refers corresponding complexes or groups for $(\G',\Y')$.)
This results in an inclusion map $Z'\map Z$. {\em It is standard that condition $(3)$ of Definition \ref{subgraph of gps} gives an injective homomorphism at the level of fundamental group.} However, the converse is not true in general. Thus we have a subtree of subspaces
$\pi': \tilde{Z}'\map \T'$ of $\pi:\tilde{Z}\map \T$.

The following result of Bass is relevant and we include this for the sake completeness of the discussion.  

\begin{prop}\textup{(\cite[Propositions $2.7$, $2.15$]{bass-cov})}\label{prop-inter property}
Suppose $(\G',\Y')$ is a subgraph of subgroups of a graph of groups $(\G,\Y)$ as in Definition \ref{subgraph of gps} satisfying conditions $(1)$ and $(2)$ only. Then the natural homomorphism $\pi_1(\G',\Y')\map\pi_1(\G,\Y)$ is injective and condition $(3)$ holds if and only if for all edge $e$ of $\Y'$ incident on a vertex $u$ of $\Y'$, one has $G'_u\cap G_e=G'_e$; think of $G'_e<G_e< G_u$ via injective homomorphisms.
\end{prop} 

\begin{comment}
As noted above, a subgraph of subgroups of a graph of groups gives rise to a subtree of spaces of a tree of spaces. Consequently, we may apply our main theorem (Theorem \ref{main thm}) to obtain results on the existence of CT maps. Hence, in view of Proposition \ref{prop-inter property}, in all our applications we may assume either that the intersection pattern (as in Proposition \ref{prop-inter property}) among the edge groups and vertex groups of the subgraph of subgroups with that of the ambient graph of groups, or that Definition \ref{subgraph of gps} holds as stated.
\end{comment}

We are now ready to state the group-theoretic applications of our main theorem. We denote a vertex (resp. an edge) of $\Y$ by a fraktur letter, namely $\mathfrak{u}$ (resp. $\mathfrak{e}$). Vertices (resp. edges) of the tree in the associated tree of spaces will be denoted by lowercase letters $u$ (resp. $e$).

%%%%%%%%%%%%%%%%%%%%%%%%%%%%%%%%%%%%%%%%%%%%

\begin{theorem}\label{main-app-ct-gen}
	Suppose $(\G,\Y)$ is a graph of hyperbolic groups with the qi embedded condition such that the fundamental group $G=\pi_1(\G,\Y)$ is hyperbolic. Let $(\G',\Y')$ be a subgraph of subgroups over $\Y'$ of $(\G,\Y)$ as in Definition \ref{subgraph of gps}. We also assume the following.
	
	\begin{enumerate}
		\item For each vertex $\mathfrak{u}$ of $\Y'$, $G'_{\mathfrak{u}}$ is hyperbolic and the inclusion $G'_{\mathfrak{u}}\to G_{\mathfrak{u}}$ admits a CT map.
		
		\item Let $\mathfrak{e}$ be an edge of $\Y'$ incident on a vertex $\mathfrak{u}$ of $\Y'$. Then
		
		\begin{enumerate}
			%\item $G_{i(e)}\cap i_e(G\pr_e)=i_e(G_e)$ and $G_{t(e)}\cap t_e(G\pr_e)=t_e(G_e)$.
			
			\item the inclusion $G'_{\mathfrak{e}}\to G_{\mathfrak{e}}$ is a qi embedding, and
			
			\item (Compatible pairwise projection condition) there is a constant $R_0\ge0$ such that for all $g\in G'_{\mathfrak{u}}$, we have $$d_{G_{\mathfrak{u}}}(P^{G_{\mathfrak{u}}}_{G_{\mathfrak{e}}}(g),P^{G'_{\mathfrak{u}}}_{G'_{\mathfrak{e}}}(g))\le R_0$$ where $P^X_{U}:X\ri U$ denotes a nearest point projection map from a metric space $X$ onto a subset $U$.
		\end{enumerate} 
	
	\end{enumerate}
	Then the CT map $\pa \pi_1(\G',\Y')\ri \pa \pi_1(\G,\Y)$ exists.
\end{theorem}
Let $G'=\pi_1(\G',\Y')$.
We note that we can identify $G'$ with a subgroup of $G$. Hence, for the proof below, we shall pretend that $G'$ is a subgroup of $G$. Also, from the work of Bestvina and Feighn (\cite{BF-Adn}) and Gersten (\cite[Corollary 6.7]{gersten}) it immediately follows
that $\pi_1(\G',\Y')$ is hyperbolic. Hence, it makes to talk about
the CT map $\pa G'\map \pa G$.

\smallskip
\noindent{\bf \em Proof of Theorem \ref{main-app-ct-gen}}:  We recall that from the given data about $G, G'$ we have (i) a tree of hyperbolic spaces $\pi:X\map T$ and a subtree of hyperbolic subspaces $\pi'=\pi|_{Y}:Y\map S$ such that both satisfy the qi embedded condition. (ii) $G$ (resp. $G'$) acts on $X$ (resp. on $Y$) properly and cocompactly by isometries.  (iii) The inclusion $Y\map X$ is equivariant with respect to the inclusion $G'\map G$. (iv)
$X$ and $Y$ are both hyperbolic. 

By choosing a point $y\in Y\subset X$ and considering the corresponding orbit
maps $G\map X $ and $G'\map Y$, it is easy to see that 
 it is enough to prove that the inclusion $Y\map X$ admits a CT map,
 as the orbit maps, which are quasiisometries, induce homeomorphisms
 $\pa G\map \pa X$ and $\pa G' \map \pa Y$ respectively. To do so we will check that this pair of trees of spaces satisfy all the conditions of Theorem \ref{main thm}. It is a standard fact that finitely generated subgroup in a finitely generated group is properly embedded with respect to their finite generating sets. It follows that the inclusion $Y\map X$ is a proper embedding. On the other hand, for all edge $e$ of $S$ incident on a vertex $u$ of $S$, $\pi^{-1}(u)=X_u$ and $\pi'^{-1}(u)=Y_u$ are acted upon properly and cocompactly by the stabilizers of $u$ in $\pi_1(\G,\Y)$ and $\pi_1(\G',\Y')$ respectively. In fact, as mentioned above, $X_u$ and $Y_u$ are uniformly quasiisometric to corresponding vertex groups of $(\G,\Y)$ and $(\G',\Y')$. Thus by condition $(1)$ and $(2)(a)$, for any edge $e$ of $S$ incident on a vertex $u$ of $S$, the inclusion $Y_u\map X_u$ admits a CT map, and $Y_e\map X_e$ is a (uniform) qi embedding. Since $G'$ acts on $Y$ and $X$ by isometry, to check the compatible fiberwise projection condition in Theorem \ref{main thm}, it is enough to show that for any edge $\mathfrak{e}$ of $\Y'$ incident on a vertex $\mathfrak{u}$ and $h\in G'_{\mathfrak{u}}$, the quadruple $(G_{\mathfrak{u}},G'_{\mathfrak{u}},hG_{\mathfrak{e}},hG'_{\mathfrak{e}})$ satisfies the compatible projection condition (Definition \ref{defn-compatible proj}). However, it follows from the condition $(2) (b)$ as $G'_{\mathfrak{u}}$ acts on Cayley graph of $G_{\mathfrak{u}}$ by isometry. This completes the proof.\qed\smallskip

We next give a few applications of Theorem~\ref{main-app-ct-gen}. 

For instance, suppose that for every edge $\mathfrak{e}$ of $\Y'$, the edge group $G'_{\mathfrak{e}}$ is a finite-index subgroup of $G_{\mathfrak{e}}$ in Theorem \ref{main-app-ct-gen}. Then the compatible pairwise projection condition in $(2)(b)$ of Theorem \ref{main-app-ct-gen} follows from the mere fact that the inclusion $G'_{\mathfrak{u}}\to G_{\mathfrak{u}}$ admits a CT map (see Lemma \ref{lem-uniform Mitra imp proj con}). Consequently, we obtain the following.

\begin{theorem}\label{thm-app-ct-finite index}
Suppose $(\G,\Y)$ is a graph of hyperbolic groups with the qi embedded condition such that the fundamental group $\pi_1(\G,\Y)$ is hyperbolic. Let $(\G',\Y')$ be a subgraph of subgroups over $\Y'$ of $(\G,\Y)$ as in Definition \ref{subgraph of gps}. We also assume the following.

\begin{enumerate}
	\item For each vertex $\mathfrak{u}$ of $\Y'$, $G'_{\mathfrak{u}}$ is hyperbolic and the inclusion $G'_{\mathfrak{u}}\ri G_{\mathfrak{u}}$ admits a CT map.
	
	\item For each edge $\mathfrak{e}$ of $\Y'$, the inclusion $G'_{\mathfrak{e}}\to G_{\mathfrak{e}}$ is an isomorphism onto a finite index subgroup of $G_{\mathfrak{e}}$.
	
%	\item Suppose the natural homomorphism $\pi_1(\G',\Y')\ri\pi_1(\G,\Y)$ is injective.
\end{enumerate}
Then $(\pi_1(\G',\Y')$ is hyperbolic and$)$ the $($injective$)$ homomorphism $\pi_1(\G',\Y')\ri\pi_1(\G,\Y)$ admits a CT map.
\end{theorem}

We end this section by recording two special cases of the above theorem.
\begin{cor}\label{cor-amal onto finite index}
Suppose $G_1$ and $G_2$ are hyperbolic groups with a common quasiconvex subgroup $H$ such that the free product with amalgamated
$G = G_1 *_H G_2$ is hyperbolic. Let $K_i < G_i$, $i=1,2$, be hyperbolic subgroups such that the inclusions
$K_i \to G_i$ admit CT maps. Let $H' < H$ be a finite-index subgroup satisfying
$K_i \cap H = H'$ for $i=1,2$, and set $K = K_1 *_{H'} K_2$.

Then $K$ is hyperbolic, and the inclusion $K \to G$ admits a CT map.
In particular, the same conclusion holds if $H < K_i$ for $i=1,2$, in which case $K = K_1 *_H K_2$.
\end{cor}

\begin{cor}\label{cor-HNN ct}
Suppose $G$ is a hyperbolic group and $H<G$ is a quasiconvex subgroup. Assume that $\phi:H\map \phi(H)<G$ an isomorphism onto a subgroup $\phi(H)$. Let $K$ be a hyperbolic subgroup of $G$ such that the inclusion $K\map G$ admits a CT map and $H,\phi(H)<K$. Consider the HNN extensions $K*_{H=\phi(H)}=K'<G'=G*_{H=\phi(H)}$. Suppose that $G'$ is hyperbolic. 

Then $K'$ is hyperbolic and the inclusion $K'\map G'$ admits a CT map. 
\end{cor}

\noindent{\bf Example.} Suppose $G_1<G_2$ hyperbolic groups where the inclusion $G_1\map G_2$
admits a CT map. Suppose $H<G_1$ is a free group such that $H$ is weakly malnormal and
quasiconvex in $G_2$. Suppose $\phi:H \map H$ is a hyperbolic automorphism. Let
$G=G_2 *_{H=\phi(H)}$ and $K=G_1 *_{H=\phi(H)}$ be the corresponding HNN extensions. Then $G$ is hyperbolic (\cite{BF}), and by Corollary \ref{cor-HNN ct}, $K$ is hyperbolic and the inclusion
$K\map G$ admits a CT map.

\subsection{A quasiconvex embedding theorem} As application of CT lamination studied in Section \ref{subsec-CT lamination} we prove the following.

{\em A subgroup $A$ of a group $B$ is said to be {\em weakly malnormal} if for all $g\in B\setminus A$, one has $A\cap gAg^{-1}$ is finite.}

\begin{theorem}\label{thm-qc}
Suppose $(\G,\Y)$ is a graph of hyperbolic groups with the qi embedded condition such that the fundamental group $\pi_1(\G,\Y)$ is hyperbolic. Assume that $(\G',\Y')$ is a subgraph of subgroups with the qi embedded condition. Moreover, suppose that:

\begin{enumerate}
	\item $\Y'=\Y$.
	
	\item For each vertex $\mathfrak{u}$ of $\Y'$, the vertex group $G'_{\mathfrak{u}}$ of $(\G',\Y')$ corresponding to $\mathfrak{u}$, is weakly malnormal and quasiconvex subgroup of $G_{\mathfrak{u}}$.
	
	\item For each edge $\mathfrak{e}$ of $\Y'$, we have $G'_{\mathfrak{e}}=G_{\mathfrak{e}}$.
\end{enumerate}
Then $\pi_1(\G',\Y')$ is quasiconvex in $\pi_1(\G,\Y)$.
\end{theorem}

\begin{proof}
Suppose $\pi:X\map T$ is a tree of spaces corresponding to $(\G,\Y)$, and $\pi'=\pi|_Y:Y\map S$ is a subtree of subspaces corresponding to $(\G',\Y')$, as discussed in the proof of Theorem \ref{main-app-ct-gen}. By Theorem \ref{thm-app-ct-finite index}, the inclusion $Y\map X$ admits a CT map $\pa i_{Y,X}:\pa Y\map \pa X$. Note that $\pi_1(\G,\Y)$ (resp. $\pi_1(\G',\Y')$) acts on $X$ (resp. on $Y$) by properly and cocompactly via deck transformations. Thus, identifying $\pa Y$ and $\pa X$ with the boundaries of $\pi_1(\G',\Y')$ and $\pi_1(\G,\Y)$ respectively through orbit maps, it remains to prove that the CT map $\pa i_{Y,X}:\pa Y\map \pa X$ is injective (see Lemma \ref{lem-qc and inj ct}).

On contrary, assume that $\al:\R\map Y$ is a geodesic line in $Y$ such that $\pa i_{Y,X}(\al(-\infty))=\pa i_{Y,X}(\al(\infty))$. Then by Theorem \ref{thm-both does not con geo ray}, we have a vertex $u\in V(S)$ and a geodesic line $\gm:\R\map Y_u$ in $Y_u$ such that $\pa i_{Y_u,Y}(\gm(\pm\infty))=\al(\pm\infty)$. Moreover, $$(\pa i_{Y_u,X_u}(\gm(-\infty)),\pa i_{Y_u,X_u}(\gm(\infty)))\in\Lambda^{\xi}(X_u,X)\hspace{7mm} (*)$$ for some $\xi\in\pa T\setminus\pa S$ as $Y_u\map X_u$ is quasiconvex. Note that there is a vertex $\mathfrak{u}$ of $\Y'=\Y$ and $g\in\pi_1(\G',\Y')$ such that $Y_u=N^{Y_u}_1(gG'_{\mathfrak{u}}.y)$ and $X_u=N^{X_u}_1(gG_{\mathfrak{u}}.y)$ for a presumed fixed point $y\in Y$. (Here $N^{W}_1(p)$ denotes $1$-radius ball centered at $p\in W$ in a metric space $W$.) On the other hand, $G'_{\mathfrak{u}}$ is weakly malnormal in $G_{\mathfrak{u}}$, i.e., $G'_{\mathfrak{u}}\cap hG'_{\mathfrak{u}}h^{-1}$ is finite for any $h\in G_{\mathfrak{u}}\setminus G'_{\mathfrak{u}}$. Thus for any vertex $\mathfrak{u}$ of $\Y'=\Y$ and $h\in G_{\mathfrak{u}}\setminus G'_{\mathfrak{u}}$, we have $\Lambda_{G_{\mathfrak{u}}}(G'_{\mathfrak{u}})\cap\Lambda_{G_{\mathfrak{u}}}(hG'_{\mathfrak{u}})=\emptyset$ (see \cite[Proposition $3$]{short} and \cite[Lemma $2.6$]{GMRS}). Again for any edge $\mathfrak{e}$ incident on $\mathfrak{u}$, $G_{\mathfrak{e}}=G'_{\mathfrak{e}}< G'_{\mathfrak{u}}$ and so $\Lambda_{G_{\mathfrak{u}}}(hG_{\mathfrak{e}})\sse \Lambda_{G_{\mathfrak{u}}}(hG'_{\mathfrak{u}})$. Thus $\Lambda_{G_{\mathfrak{u}}}(hG_{\mathfrak{e}})\cap \Lambda_{G_{\mathfrak{u}}}(G'_{\mathfrak{u}})=\emptyset$. This says that the geodesic line $\gm$ has no boundary flow in $(T\setminus S)\cup\{u\}$. This is a contradiction to $(*)$. Hence, we are done.
\end{proof}

\noindent{\bf Example.} Suppose $H$ is a hyperbolic group. Let $F_1<H$ be a weakly malnormal quasiconvex free subgroup of rank at least $3$, and $F_2<F_1$ be malnormal in $F_1$. Thus $F_2$ is weakly malnormal in $H$. Consider the doubles $G=H*_{F_2}H$ and $G'=F_1*_{F_2}F_1$. Note that $G$ and $G'$ are hyperbolic groups (\cite{BF}, \cite[Theorem $2$]{KM98}). By Theorem \ref{thm-qc}, $G'$ is quasiconvex in $G$.

\begin{cor}\label{cor-Swarup's qsn}
Suppose $(\G,\Y)$ is a graph of hyperbolic groups with the qi embedded condition such that the fundamental group $\pi_1(\G,\Y)$ is hyperbolic. Assume that $(\G',\Y')$ is a subgraph of subgroups with the qi embedded condition. Moreover, suppose that:

\begin{enumerate}
	\item $\Y'=\Y$.
	
	\item For each vertex $\mathfrak{u}$ of $\Y'$, the vertex group $G'_{\mathfrak{u}}$, of $(\G',\Y')$ corresponding to $\mathfrak{u}$, is quasiconvex subgroup of $G_{\mathfrak{u}}$.
	
	\item For each edge $\mathfrak{e}$ of $\Y'$, we have $G'_{\mathfrak{e}}=G_{\mathfrak{e}}$.
\end{enumerate}
If $\pi_1(\G',\Y')$ is weakly malnormal in $\pi_1(\G,\Y)$ then $\pi_1(\G',\Y')$ is quasiconvex in $\pi_1(\G,\Y)$.
\end{cor}

The proof of Corollary \ref{cor-Swarup's qsn} follows from the fact that $G'_{\mathfrak{u}}$ is weakly malnormal in $G_{\mathfrak{u}}$ as $\pi_1(\G',\Y')$ is weakly malnormal in $\pi_1(\G,\Y)$, and that Theorem \ref{thm-qc}.

We recall that an infinite subgroup $H$ of a group $G$ is said to have \emph{height $n\in\N$} if there exist $n$ distinct cosets
$\{g_iH : 1 \le i \le n\}$ such that $\bigcap_{i=1}^{n} g_i H g_i^{-1}$
is infinite, and $n$ is maximal with this property. It was proved that quasiconvex subgroups of hyperbolic groups have finite height in \cite{GMRS}. Then G. A. Swarup asked whether the converse is true (see \cite[Question 1.8]{bestvinaprob}). We note that even if $H$ is of height $1$ in $G$, i.e. $H$ is weakly malnormal in $G$, Swarup’s question remains open. It is known only in certain special cases; see \cite{mitra-ht}, \cite[Corollary 4]{ilya-kap-qc-amal} and \cite[Section 6]{KM98}. As an application of Theorem \ref{thm-qc}, we obtain another instance in which Swarup’s question has an affirmative answer when $H$ is weakly malnormal
in the form of the above corollary.

\subsection{Nonnecessity of the compatible fiberwise projection condition}
In this section we show that the compatible fiberwise projection condition in Theorem \ref{main thm} is not necessary
for the existence of CT maps even in the presence of all the remaining hypotheses.
Lemma \ref{acyl-case} and Example \ref{acyl example} below verify this. We start
by recalling the following definition from \cite{ps-kap}.

	\begin{defn}\label{acyl}\textup{(\cite[Definition $2.50$]{ps-kap})}
		Suppose $k\ge1$. A tree of metric spaces $\pi:X\ri T$ is said to be $k$-acylindrical if there are constants $M_k\ge0, \tau_k\ge0$ depending on $k$ such that for every pair of $k$-qi lifts, $\gm_0$, $\gm_1$ say, of a geodesic segment $I\sse T$ of length $>\tau_k$ we have that for all vertices $v$ of $I$, $$d_{X_v}(\gm_0(v),\gm_1(v))\le M_k$$

		We say $\pi:X\ri T$ is acylindrical tree of metric spaces if it is $k$-acylindrical for some $k\ge1$.
	\end{defn}

	\begin{comment}
	\RED{I dont see any use of the folllowing:}

{\tiny	Note that if $T$ is finite then any tree of metric spaces $\pi:X\ri T$ is acylindrical. This leads to the following lemma follows from \cite[Lemma $7.3$]{ps-kap}.

	\begin{lemma}\label{T-is-of-fi-ball}
		Given $n\in\N$ there is $\dl_{\ref{T-is-of-fi-ball}}=\dl_{\ref{T-is-of-fi-ball}}(\dl,L_0,n)$ and $L_{\ref{T-is-of-fi-ball}}=L_{\ref{T-is-of-fi-ball}}(\dl,L_0,n)$ such that the following holds.

		Suppose $\pi:X\ri T$ is a tree of hyperbolic metric graphs with the qi embedded condition. Let $u$ be a vertex in $T$ such that $T\sse d_T(u,n)$. Then $X$ is $\dl_{\ref{T-is-of-fi-ball}}$-hyperbolic and vertex spaces (hence edge spaces) are $L_{\ref{T-is-of-fi-ball}}$-qi embedding in $X$.
	\end{lemma}

	\begin{proof}
		Note that $\pi:X\ri T$ is $(M,k,n+1)$-acylindrical for a fixed number $M$ and for all $k\ge1$. Therefore, by \cite[Lemma $7.3$]{ps-kap}, $X$ is $\dl_{\ref{T-is-of-fi-ball}}$-hyperbolic and vertex (hence edge) spaces are $L_{\ref{T-is-of-fi-ball}}$-qi embedding in $X$.
	\end{proof}
}
\end{comment}
We are thankful to Ravi Tomar for pointing out the following lemma.
The proof is immediate from the description of uniform quasigeodesics in
acylindrical tree of hyperbolic spaces with qi embedded condition as
described in \cite[Subsection $7.3$]{ps-kap}. Therefore, we skip writing
any proof for it.
\begin{lemma} \label{acyl-case}
Suppose $\pi'=\pi|_{Y}:Y\ri S$ is an induced subtree of hyperbolic spaces of an acylindrical tree
of hyperbolic spaces $\pi:X\ri T$. Moreover, suppose that both these trees of spaces
satisfy the qi embedded condition and that $X$, $Y$ are hyperbolic. Finally, for all vertex $u$ of $S$, the inclusion $Y_u\map X_u$ admits a CT map.
Then there is a CT map for the inclusion $i:Y\to X$.
\end{lemma}

%Note that since $\pi:X\map T$ is acylindrical tree of hyperbolic spaces with qi embedded condition, $X$ is hyperbolic (\cite) is in Lemma \ref{acyl-case}
\begin{comment}
\begin{proof}
	 We prove it by contradiction. Suppose the inclusion
	$Y\hri X$ does not admit a CT-map. Then there are sequences $\{y_n\}$ and $\{y\pr_n\}$ of
	$Y$ such that $[\{y_n\}],[\{y\pr_n\}]\in\pa_s Y$ and $[\{y_n\}],[\{y\pr_n\}]\in\pa_s X$ but
	$[\{y_n\}]=[\{y\pr_n\}]$ in $\pa_s Y$ and $[\{y_n\}]\ne[\{y\pr_n\}]$ in $\pa_s X$. Let
	$b_n=\pi(y_n)$ and $b\pr_n=\pi(y\pr_n)$.

	{\bf Case 1}: Suppose sequences $\{b_n\}$ and $\{b\pr_n\}$ are bounded, then we are done by
	Proposition \ref{finite-proj-case}.

	{\bf Case 2}: Suppose $\{b\pr_n\}$ is bounded and $\{b_n\}$ is unbounded. Since $Y$ is
	acylindrical, $[u,b_n]$ converges to a geodesic ray, say, $[u,\xi)$ in $T$ for some $\xi\in\pa T$
	(cf. \cite[Subsection $7.3$]{ps-kap}). Let $c_n$ be the center of triangle $\triangle(u,b_n,\xi)$.
	As in Proposition \ref{finite-proj-case}, we can assume that $\{y\pr_n\}\sse Y_u$ for some vertex
	$u\in S$. Then $[\{y\pr_n\}]\ne[\{y_n\}]$ in $\pa_s X$ $-$ which contradicts to our assumption.

	{\bf Case 3}: Suppose both $\{b_n\}$ and $\{b\pr_n\}$ are unbounded. Since $[\{y\pr_n\}]\ne[\{y_n\}]$
	in $\pa_s Y$, then we get two different directions in $T$ corresponding to $\{y\pr_n\}$ and $\{y_n\}$,
	which gives us a contradiction that $[\{y\pr_n\}]=[\{y_n\}]$ in $\pa_s X$.
\end{proof}
\end{comment}

\begin{cor}\label{new cor sec 6}
Suppose $G_1, G_2$ are two hyperbolic groups and $\lgl g_1\rgl< G_1$ and $\lgl g_2\rgl< G_2$ are two
infinite cyclic malnormal subgroups in them respectively. Let $H_1<G_1$ and $H_2<G_2$ be
infinite hyperbolic subgroups. Let $G$ be the amalgam of $G_1, G_2$ obtained using an isomorphism
of the cyclic subgroups $\lgl g_1\rgl$, $\lgl g_2\rgl$ and let  $H=H_1 * H_2$.
Now suppose that the following additional conditions are satisfied.
\begin{enumerate}
\item The CT map exists for the inclusion $H_i\map G_i$, $i=1,2$.
\item $H_i\cap \lgl g_i\rgl  =(1)$, but $g^{\pm \infty}_i\in \Lambda_{G_i}(H_i)$ for $i=1,2$.
\end{enumerate}
%(It follows from condition $(2)$ that The natural homomorphism $H\map G$ is injective and the resulting morphism of Bass--Serre trees of $H$ into the Bass--Serre tree of $G$ is injective; see \cite[Propositions $2.7$, $2.15$]{bass-cov}.)

Then $G, H$ are hyperbolic groups, and the CT map for the inclusion $H\map G$ exists. 
However, the compatible fiberwise projection condition for the inclusion of the tree of spaces corresponding for
$H$ into the tree of spaces for $G$ fails to hold.
\end{cor}
\begin{proof}
Note that $G={G_1}_{\lgl g_1\rgl=\lgl g_2\rgl}G_2$. As constructed above, we will have a tree of spaces $\pi:X\map T$ corresponding to ${G_1}_{\lgl g_1\rgl=\lgl g_2\rgl}G_2$ and a subtree of subspaces $\pi'=\pi|_{Y}:Y\map S$ in it corresponding to $H_1*H_2$. It is proved in \cite[Corollary $2.56$]{ps-kap} that $\pi:X\map T$ is $3$-acylindrical. Hence $X$ is hyperbolic, equivalently, $G$ is hyperbolic (\cite{BF}, \cite[Theorem $2$]{KM98}) whereas $H$ is also hyperbolic. By Lemma \ref{acyl-case}, the inclusion $Y\map X$, equivalently, $H\map G$ admits a CT map.

Now we show that the compatible fiberwise projection condition fails in this case. Note that edge spaces of $Y$ are single points. For an edge $e$ in $S$ incident on a vertex $u$, let $Y_{eu}=\{p\}$ denote the image of the edge space of $Y$ corresponding to $e$ in the vertex space $\pi'^{-1}(u)=Y_u$. Then $P^{Y_u}_{Y_{eu}}(y)=p$ for all $y\in Y_u$.

Note that we have fixed word metrics on $G_i$ and $H_i$. By the correspondence between the vertex spaces and the orbit maps, as explained in the proof of Theorem \ref{thm-qc}, it is enough to show that $\{P^{G_1}_{\lgl g_1\rgl}(H_1)\}$ has unbounded diameter in $G_1$ where $P^{G_1}_{\lgl g_1\rgl}$ denotes a nearest point projection map from $G_1$ onto $\lgl g_1\rgl$ with respect to the above metric. By the condition $(2)$, we have $\{h_n\}\sse H_1$ such that $\lim^{G_1}_{n\map\infty}h_n=g^{\infty}_1$. Let $P^{G_1}_{\lgl g_1\rgl}(h_n)=g^{m(n)}_1$ for some $m(n)\in\Z$. Since the orbit of $\lgl g_1\rgl$ in $G_1$ is quasiconvex, $[h_n,g^{m(n)}_1]_{G_1}\cup[g^{m(n)}_1,1]_{G_1}$ are uniform quasigeodesics in $G_1$ (see Lemma \ref{lem-proj on pair qc}~\ref{lem-proj on pair qc:concatenation qg}), and that converges to $g^{\infty}_1$. Thus $\lim_{n\map\infty}m(n)=\infty$. Hence $P^{G_1}_{\lgl g_1\rgl}(\{h_n\})$ has unbounded diameter in $G_1$. This completes the proof.
\end{proof}

\begin{example}\label{acyl example}
Suppose $G_1= \mathbb F_3\rtimes \Z$ is a hyperbolic, free-(on three generators)-by-cyclic group
and let $t$ be a generator of the copy of $\Z$. Let $G$ be the double of $G_1$ taken along $\Z$.
Let $H_1=\mathbb F_3\subset  G_1$. Then all the hypotheses of Corollary \ref{new cor sec 6}
are satisfied. So the CT map exists for the inclusion $H_1* H_1\map G_1 *_{\Z} G_1$ but
the compatible fiberwise projection condition for the inclusion of the corresponding trees of spaces fails to hold.
\end{example}

\subsection{Nonexistence of CT maps}

In this section, we illustrate with an example that in our main theorem, if the compatible fiberwise projection condition is dropped, then the CT map may fail to exist even when all the remaining hypotheses are satisfied. Building on the idea of this example, more general results describing obstructions to the existence of CT maps are obtained in the paper \cite{HMS-landing} of
the authors jointly with Mahan Mj. Accordingly, this example also appears in \cite[Example $1.2$]{HMS-landing}.
Our purpose here is to show the absence of the compatible fiberwise projection condition.

\begin{example}\label{counterexample}
Consider a hyperbolic group $G$ of the form $G=N\rtimes Q$, where $N$ is
	either the fundamental group of a closed orientable surface of genus at least $2$ or a finitely generated
	free group of rank at least $3$, and $Q$ is a finitely generated free group of rank at least $3$.
	Examples of this sort are well-known; e.g. see \textup{\cite{farb-mosher}, \cite{BFH-lam}}. It is easy to see
	that $Q$ is a malnormal quasiconvex subgroup of $G$ and $N$ is a nonquasiconvex subgroup of $G$ with
	$\Lambda_G(Q)=\partial G$. Let $\phi:Q\map Q$ be a hyperbolic automorphism. Consider the HNN extension $G'=G*_{Q=\phi(Q)}$ with stable letter $t$. Note that $K'=\lgl N,t\rgl\simeq N*\lgl t\rgl$. Then the inclusion $K'\map G'$ does not admit a CT map \textup{(\cite[Example $1.2$]{HMS-landing})}. Moreover, note that $\Lambda_G(Q)\sse\Lambda_G(N)$. This says for any infinite order element $q\in Q$, $q^{\pm\infty}\in\Lambda_G(N)$. Then the similar proof as in Corollary \ref{new cor sec 6} will show that the induced subtree of spaces in the ambient tree of spaces fails to have the compatible fiberwise projection condition.
\end{example}

\bibliography{Ubib}
\bibliographystyle{amsalpha}
\end{document}